\documentclass[english,oneside,reqno]{amsart}
\usepackage{hyperref}

\usepackage[utf8]{inputenc} % input encoding
\usepackage[english]{babel}

\usepackage[T1]{fontenc}

\usepackage{geometry}

\usepackage{xcolor}

\usepackage{todonotes}

\usepackage{graphicx} 				% \includegraphics command
\usepackage{booktabs} 				% for better tables
\usepackage{array}    				% better matrices
\usepackage[parfill]{parskip}		% begin a paragraph with an empty line rather than an indent.
\usepackage{enumitem} 				% environment name: enumerate, augments: \alph  \Alph, \arabic, \roman,  \Roman,

\usepackage{nicefrac}
\usepackage{bbm}

\usepackage{tikz}
\usepackage{tikz-cd} % commutative diagrams
\usetikzlibrary{decorations.markings,positioning}
\usetikzlibrary{matrix,decorations.pathmorphing,decorations.pathreplacing,calligraphy,}
\usetikzlibrary{calc, intersections,hobby}
\usetikzlibrary{arrows,shapes.geometric,calc}	% arrow tip library
\usetikzlibrary{bending}		% better arrow head for bended lines
\usepackage{rank-2-roots}
\usepackage{dynkin-diagrams}

\usepackage{mathrsfs}
\usepackage{dutchcal} % dutch mathcal font
\usepackage{amsfonts}
\usepackage{amssymb}

\usepackage{amsmath}

\usepackage{ytableau}

\usepackage{amsthm}
\usepackage{thmtools}

\usepackage{textgreek}
\usepackage{newunicodechar}
\newunicodechar{ϛ}{\ensuremath{\varsigma}} % use varsigma as a stand-in

\usepackage[alphabetic, initials]{amsrefs}

\usepackage{hyperref}
\usepackage[capitalize, nameinlink]{cleveref}

\crefalias{lemma}{lem}

\theoremstyle{plain}
\newtheorem{lem}{Lemma}[section]
\newtheorem{lemma}[lem]{Lemma}
\newtheorem{prop}[lem]{Proposition}
\newtheorem{cor}[lem]{Corollary}
\newtheorem{thm}[lem]{Theorem}
\newtheorem{thmintro}{Theorem}

\theoremstyle{definition}

\newtheorem{definition}[lem]{Definition}

\newtheorem{rem}[lem]{Remark}
\newtheorem{remark}[lem]{Remark}

\newtheorem{notation}[lem]{Notation}
\newtheorem{observation}[lem]{Observation}
\newtheorem{ex}[lem]{Example}
\newtheorem{eg}[lem]{Example}
\newtheorem{conj}[lem]{Conjecture}

\crefname{equation}{}{}

\allowdisplaybreaks

\renewcommand*\arraystretch{1.4}  % bigger table height

\newcommand*{\StrikeThruDistance}{0.1cm}%

\tikzset{strike thru/.style={
decoration={markings, mark=at position 0.5 with {
\draw [, ,-] 
++ (-\StrikeThruDistance,-\StrikeThruDistance) 
-- ( \StrikeThruDistance, \StrikeThruDistance);}
},
postaction={decorate},
}}

\tikzset{
	anchorbase/.style={baseline={([yshift=-0.5ex]current bounding box.center)}},
	every path/.style={draw,very thick},
	every picture/.style={anchorbase,auto},
	>=To,
}

\newcommand{\ytabcenter}[1]{
\ytableausetup{smalltableaux}
	\vcenter{\hbox{$\begin{ytableau}#1\end{ytableau}$}}
}

\newcommand{\youngcenter}[1]{
{{\ytableausetup{boxsize=4pt}
	\vcenter{\hbox{$\ydiagram{#1}$}}}}
}

\newcommand{\csaZW}{{U}_0'}
\newcommand*\circled[1]{\tikz[baseline=(char.base)]{
            \node[shape=circle,draw,inner sep=1pt] (char) {#1};}}

\newcommand{\nocontentsline}[3]{}
\let\origcontentsline\addcontentsline	
\newcommand\stoptoc{\let\addcontentsline\nocontentsline}
\newcommand\resumetoc{\let\addcontentsline\origcontentsline}	

	\newcommand{\superFunnyVector}{\Omega}

\title{Weight modules for quantum symmetric pair subalgebras}

\author{Catharina Stroppel}
\address{CS: Mathematisches Institut, University of Bonn, Endenicher Allee 60, 53115 Bonn, Germany, stroppel@math.uni-bonn.de,}

\author{Liao Wang}
\address{LW: Mathematisches Institut, University of Bonn, Endenicher Allee 60, 53115 Bonn, Germany, s6liwang@math.uni-bonn.de}
\begin{document}

\newcommand{\classicalFixPointSubalg}{\gl_4^\theta}
\newcommand{\classicalGLProd}{\gl_2^+\times \gl_2^-}

\newcommand{\specializedGroundRing}{\C}

%%%%%%%%%%%% PBW of coideal
\newcommand{\bo}{B_0}
\newcommand{\eCoideal}{B_{1}}
\newcommand{\fCoideal}{B_{-1}}
\newcommand{\dCoideal}{\hat{D}}

\newcommand{\bbZ}{\mathbb{Z}}
\newcommand{\bbN}{\mathbb{N}}

\newcommand{\eNewCoideal}{X}
\newcommand{\fNewCoideal}{Y}

\newcommand{\csaCoideal}{U_0'}

\newcommand{\qcommutator}[3][-1]{[#2,#3]_{q^{#1}}}

\newcommand{\coideal}{U_q'}
\newcommand{\UEm}{\mathcal{E}}

\newcommand{\inv}{^{-1}}

\newcommand{\specht}{\mathbb{S}}
	\newcommand{\verma}{\mathbb{M}}

\newcommand{\basis}{e}

\newcommand{\kCoideal}{\hat{K}}
\newcommand{\Epm}{E}
		\newcommand{\Fpm}{F}
	
\newcommand{\content}{\operatorname{cont}}

\newcommand{\bClassical}{b}
\newcommand{\vecrepClassical}{\mathbb{V}}
\newcommand{\kmatrix}{\mathbb{K}}
\newcommand{\weylGroup}[1][d]{W(\operatorname{B}_{#1})}
\newcommand{\id}{\operatorname{id}}
\newcommand{\plusSubspace}{M_+}
\newcommand{\minusSubspace}{M_-}

\newcommand{\csaCclassical}{z}
\newcommand{\adl}{\operatorname{ad}}

\newcommand{\negidx}[1]{\overline{#1}}
\newcommand{\half}{{\diamondsuit}}
\newcommand{\heckeSpecialized}[1]{\mathscr{H}_{1,q}(\operatorname{B}_{#1})}

\newcommand{\graphGeneral}{\Gamma}
\newcommand{\linearGraph}{\Gamma_1}
\newcommand{\source}[1]{\sigma(#1)}
\newcommand{\target}[1]{\tau(#1)}

\newcommand{\graphVertices}{J}
\newcommand{\graphEdges}{I}

\newcommand{\vecrep}[1][q]{\mathbb{V}_{#1}}

\newcommand{\Hom}{\operatorname{Hom}}

\newcommand{\C}{\mathbb{C}}
\newcommand{\Z}{\mathbb{Z}}
\newcommand{\R}{\mathbb{R}}
\newcommand{\Q}{\mathbb{Q}}
\newcommand{\F}{\mathbb{F}}
\newcommand{\N}{\mathbb{Z}_{\geq 0}}

\newcommand{\groundring}{\C(q)}

\newcommand{\gl}{\mathfrak{gl}}

\renewcommand{\sl}{\mathfrak{sl}}

\newcommand{\Uq}{U_q}

\newcommand{\dottedArrowLeft}{
\;\begin{tikzpicture}[scale=0.5, every path/.style={thin}]
\draw[->,dotted] (0,0) to +(-1,0);
\end{tikzpicture}\;
}

\newcommand{\noAdjacent}{
\;\begin{tikzpicture}[scale=0.5, every path/.style={thin}]
\draw[strike thru] (0,0) to +(-0.5,0);
\end{tikzpicture}\;
}

\newcommand{\vecBasis}[1][i]{e_{#1}}

\newcommand{\fixedSubAlg}[2]{\mathfrak{gl}_{#1}\times\mathfrak{gl}_{#2}}
\newcommand{\rootOfMinusOne}{\iota}

\newcommand{\plusCoideal}{{U}_{+}'}
		\newcommand{\minusCoideal}{{U}_{-}'}
		\newcommand{\CartanPart}{{U}_0'}

\newcommand{\ratIrrep}{\mathsf{L}}
\begin{abstract}We develop a  theory of weights for a quantum analogue of $(\gl_4,\gl_2\times\gl_2)$ realised as a quantum symmetric pair subalgebra. Based on Letzter's triangular decomposition we define Verma modules. Using ``magical'  operators that are compatible with weight spaces, we classify weight Verma modules and characterise their irreducible finite dimensional quotients. We then prove the existence of weight bases in tensor products by explicitly constructing some highest weight vectors. These constructions allow us to mimic the important aspects of the classical finite dimensional representation theory. 
Applications include a definition of rational representations, the BGG resolution, a Clebsch--Gordan formula, the Harish-Chandra isomorphism and central characters, as well as a classification and description of all irreducible polynomial representations.
\end{abstract}
\maketitle

\tableofcontents
\section*{Introduction}
In the classical representation theory of complex reductive Lie algebras, a thorough understanding of the basic cases $\mathfrak{sl}_2$ and $\mathfrak{gl}_2$ forms the cornerstone of the entire theory. Motivated by this principle, we study the finite-dimensional representation theory of the quantum symmetric pair subalgebra associated with $(\gl_4,\gl_2 \times \gl_2)$, a basic case of type $\operatorname{AIII}$. Our primary goal is to develop a \emph{(highest) weight theory} analogous to that for complex reductive Lie algebras, in which $\gl_2 \times \gl_2$ is viewed as a Lie subalgebra of $\gl_4$ arising as the fixed-points of an involution; see \cref{S:classical}.

Quantum symmetric pairs were introduced in \cite{Letzter-symmetric-pairs-quantized-enveloping} and \cite{Letzter-coideal-subalgebras} as quantum analogues of the universal enveloping algebras $U(\mathfrak{g}^\theta) \subset U(\mathfrak{g})$ associated with symmetric pairs $(\mathfrak{g}, \mathfrak{g}^\theta)$, where $\mathfrak{g}^\theta$ denotes the fixed-point Lie algebra of an involution $\theta$ of a reductive complex Lie algebra $\mathfrak{g}$. 

A significant achievement of \cite{Letzter-symmetric-pairs-quantized-enveloping}, \cite{Letzter-coideal-subalgebras}
 was the uniform formulation of the $U_q'(\mathfrak{g}^\theta)$ in terms of Serre-type presentations, which, modulo lower-order terms, parallel the usual presentations of the quantised enveloping algebras of reductive Lie algebras.  The algebras $U_q'(\mathfrak{g}^\theta)$ are constructed as subalgebras of the quantised enveloping algebras  $U_q(\mathfrak{g})$ of $\mathfrak{g}$, and specialise to $U(\mathfrak{g}^\theta)$ in a suitable limit $q\rightarrow 1$. The subalgebras are not Hopf subalgebras, but satisfy the \emph{coideal} property. They are  therefore called   \emph{quantum symmetric pair coideal subalgebras}. Letzter's construction has been extended to the Kac–Moody setting in \cite{KolbQSP}  and further  in \cite{RV}. 

A new perspective on quantum symmetric pair coideal subalgebras arose with the introduction of the bar involution and canonical basis in \cite{BKolb}, \cite{BaoWang}, \cite{EhSt-nw-algebras-howe}.
Despite the ensuing activity, the \emph{representation theory} of quantum symmetric pair coideal subalgebras remains surprisingly underdeveloped. In particular, a full classification of finite-dimensional irreducible representations is still missing, despite recent progress in \cite{Watanabe0}, \cite{Watanabe1}, \cite{Watanabe2}, \cite{Watanabe-crystal-basis-quantum-symmetric-pair}, and \cite{Wenzl}. Unlike the case of complex reductive Lie algebras or compact Lie groups, there is no satisfactory highest weight framework. Existing approaches therefore rely on alternatives such as Gelfand–Tsetlin-basis type constructions \cite{IK} or partial weight theories \cite{Watanabe}, \cite{Watanabe1}, \cite{Watanabe2}, often based on the subalgebra $\mathfrak{h}^\theta \subset \mathfrak{h} \subset \mathfrak{g}$. As $\mathfrak{h}^\theta$ is typically much smaller (e.g.\  of half the desired dimension in our case) than a Cartan subalgebra of $\mathfrak{g}^\theta$, it cannot capture the full structure. Apart from special cases, \cite{Wenzl}, \cite{Molev}, \cite{Kolb-Stephens-very-non-standard}, no suitable notion of Verma modules is currently available due to the absence of a (good) triangular decomposition for quantum symmetric pairs. 

Both, \cite{BaoWang} and \cite{EhSt-nw-algebras-howe}, considered Satake diagrams of type $\operatorname{AIII}$ without black nodes. The associated quantum symmetric pairs are divided into 2 very different subfamilies according to the parity of the number of nodes, see \cref{ourgraphs}. Previous works on the classification of (good) finite dimensional irreducible modules include \cite{StWoj-coidealDiagrammatics} for the commutative odd rank $1$ case, \cite{ARK-branching-rules}, \cite{ARK-branching-rules}, \cite[Thm. 3.1.6]{Watanabe-crystal-basis-quantum-symmetric-pair} for the even rank $1$ case, and \cite[Thm.4.3.7]{Watanabe-crystal-basis-quantum-symmetric-pair} for all even rank cases inside a suitable category $\mathcal{O}_{\operatorname{int}}$. We consider the smallest example in the odd subfamily, and in particular complete the classification of good finite dimensional modules in all  basic cases.

To describe the results of the paper in more detail, consider the quantum symmetric pair coideal subalgebra $\coideal$ associated to the following Satake diagram.
	\begin{equation}\label{eqinintro}
		\dynkin[%
		edge length=1cm,
		labels={1,0,-1},
		label directions={above,above,above},
		involution/.style={blue!50,stealth-stealth,thick},
		involutions={13}
		]{A}{ooo}
		\quad
\begin{array}[c]{c}		
\text{generators of $\coideal:$}\\
\bo, B_{\pm 1},\dCoideal_\half, \dCoideal_1.
\end{array}  
\quad\quad	
\begin{array}[l]{c}
\text{possible}\\
\text{Cartan elements:}
\end{array}  		
\begin{array}[c]{c}	
Z = [\eCoideal, \qcommutator{\bo}{\fCoideal}]_{q^{-1}} ,\\ 
W = [\fCoideal,\qcommutator{\bo}{\eCoideal}]_{q^{-1}}.
\end{array}
	\end{equation}

Building on the notion of Cartan subalgebras  from \cite{Le-cartan-coideal}, we construct a triangular decomposition $\minusCoideal\otimes \CartanPart\otimes  \plusCoideal$ of $\coideal$ in \Cref{triangular}.   There are precisely two choices of a Cartan subalgebra of $\coideal$ that contain both the fixed-point subalgebra of the Cartan subalgebra of the ambient quantum group and the coideal generator $B_0$. These are the algebras generated by $\dCoideal_{\half}, \dCoideal_1, \bo$  with $Z$, respectively $W$. The two choices—in fact $Z$, $W$ themselves—coincide in the classical limit.  Our preferred choice, $Z$ over $W$, differs from the alternative one and  Letzter’s, only by application of the bar involution. Our triangular decomposition lifts the classical triangular decomposition (see \cref{classical CSA}) and facilitates the definition of Verma modules and the study of finite dimensional $\coideal$-modules. The latter arise, after a suitable field extension, as quotient of Verma modules, \Cref{hw vec exists}.

Standard arguments however do not apply here, since the triangular parts of $\coideal$ do not yield a weight theory and Verma modules generally lack a weight space decomposition. We therefore distinguish \emph{good} Verma modules, on which $\csaZW$ acts diagonalisable, from \emph{exceptional} ones.

\begin{thmintro}[\cref{fd quot good verma}]\label{thm1intro}
				Consider a good Verma module $\verma=\verma({\kappa_\half},{\kappa_1},[\mu;0],\zeta)$  with highest weight vector $v$. Then it has a finite dimensional irreducible quotient $L$ if and only if  
\begin{equation}
					{\zeta = [\mu;\kappa-2i] - q^{-\kappa} [\mu;0]}\quad\text{for some $i\in \N$ with $0\le i \le \kappa$.}
\end{equation}
				In this case $L$ is unique and has a weight basis 
					$\{\Fpm_+^a \Fpm_-^b v\mid 0\le a\leq i 
							,\, 0\le b\le \kappa-i\}.$
			\end{thmintro}
The good finite-dimensional irreducible modules with ``integral'' parameters form the \emph{rational representations} and lift classical finite-dimensional modules (see \Cref{classification rational rep}). Although finite-dimensional quotients of exceptional Verma modules come with a nice condition on their highest weights (see \eqref{zeta value in nonclassical verma}), they may not fit into a unifying theory as we illustrate in examples.

We next consider  the irreducible rational representations which are  \emph{polynomial} that is which appear in tensor powers $\vecrep^{\otimes d}$ of the natural representation of $\Uq(\gl_4)$ under restriction to $\coideal$.   Since, unlike in classical or quantum group settings, the triangular parts of $\coideal$ are not generated by primitive or grouplike elements, tensor products of weight vectors are usually not weight vectors and the action of the coideal generators is hard to control. We introduce  the concept of  \emph{funny (weight) vectors} which are inductively defined weight vectors in  $\vecrep^{\otimes d}$ and describe the  action of the $\coideal$ generators $B_{\pm1}$ by  a family of weight space depending \emph{magical operators}, see  \Cref{catharina operators}.  These operators allow to create new weight vector from given weight vectors, see \Cref{E+- F+- act on weight space any verma}  and  \Cref{thm1intro}. In  \Cref{funny vectors} we show that they interact with the tensor product structure in an intriguing way. We use this  to assign to any $2$-row bipartition $(\lambda,\mu)$ of $d$  inductively  a  weight vector  $\superFunnyVector(\lambda,\mu)$ in $\vecrep^{\otimes d}$. 
The following result should indicate  that we mimic classical weight theory: 
\begin{thmintro} \begin{enumerate}\item
Taking the highest weight gives a bijection	
			 \begin{align*}	
			 	\left\{ 
				\begin{minipage}[c]{5.2cm}
				isomorphism classes of irreducible 
				 \\polynomial representations of $\coideal$ 
				\end{minipage}
				\right\}\;\;
				& \cong\qquad\qquad\qquad
					\left\{ \text{ $2$-row bipartitions of $d$} 
				\right\}.\\
			\ratIrrep(\kappa_\half,\kappa_1,[n],[n+\kappa-2i] - q^{-\kappa}[n])\quad\;\;& \mapsto \;\;
				\left((\frac{\kappa_\half+ n}{2}, \frac{\kappa_\half +n }{2} -i ), (\frac{\kappa_\half - n}{2}, \frac{\kappa_\half - n }{2} -\kappa+ i ) \right).\nonumber
		\end{align*} 	
\item The vector $\superFunnyVector(\lambda,\mu)$ is a weight vector of $(\dCoideal_\half,\dCoideal_1,\bo,Z,W )$-weight $(q^{\kappa_\half},q^{\kappa_1},[n],\zeta,\omega ) $ with 
					\begin{equation*}
						\begin{split}
						\kappa = \kappa_\half-\kappa_1,\qquad
						\kappa_\half = \lambda_1+\mu_1 
						, \qquad
						\kappa_1 = \lambda_2+\mu_2
						,\qquad
						n = \lambda_1-\mu_1
						,\\
						\zeta = [\lambda_2-\mu_2] - q^{-\kappa} [n]
						,
						\quad
						\omega = q^{-2} [\lambda_2-\mu_2]  -q^{\kappa-2} [n].
					\end{split}
					\end{equation*}
It is a highest weight vector of an irreducible summand of isomorphism type given by $(\lambda,\mu)$.	
\end{enumerate}
			\end{thmintro}

 Using Schur--Weyl duality, we deduce a stronger result, namely an explicit  complete decomposition of $\vecrep^\otimes $  into isotypical components. As an application we show a Clebsch--Gordan rule. 
 
In  \Cref{HC} we propose an analogue of the Harish–Chandra homomorphism for the centre $Z(\coideal)$ of $\coideal$  and a replacement of the Weyl group invariants. We  realise $Z(\coideal)$ in \Cref{centre}  as a polynomial algebra giving a concrete incarnation of the general result of Kolb and Letzter \cite{MR2439008}. We deduce that central characters  separate irreducible rational representations, \Cref{centralcharacters}.

\subsection*{Acknowledgements}{We thank Stefan Kolb for valuable feedback and for generously sharing insights, and Martina Balagovi\'c for helpful discussions.
We are grateful for the support by the Gottfried-Wilhelm-Leibniz Prize of the German Research Foundation and the Hausdorff Center for Mathematics in Bonn. The results of this article are part of the second author's PhD thesis. }

\section{Preliminaries on quantum symmetric pairs}

	\begin{notation}
	For a vector space $W$ with an  endomorphism $x$ we denote by $W(x\vert\lambda)$ the $\lambda$-eigenspace of $x$. More generally, given pairwise commuting endomorphisms $x_1,\ldots, x_m$ of $W$  we denote by $W(x_1,\ldots, x_m\vert\lambda_1,\ldots,\lambda_m)$ the simultaneous eigenspaces of $W$  with eigenvalue $\lambda_i$ for $x_i$. 

	We denote by $\sl_{r}\subset\gl_{r}$ the complex special respectively  general  Lie algebra of $r\times r$ matrices. 
	
	\end{notation}

	\subsection{Quantum groups and quantum symmetric pairs}\label{SS:QSP}
	
	We fix as ground field $\groundring$,  the field of rational functions in $q$ and choose a root  $\rootOfMinusOne = \sqrt{-1}\in \mathbb{C}$. 
	Define for  $n\in \Z$ the \emph{quantum integers}  
	\[
		[n] \coloneqq \frac{q^n-q^{-n}}{q-q^{-1}}\in\bbZ[q,q^{-1}]\subset\groundring, 
		\quad\text{and set} \quad \alpha_n=q^n-q^{-n}.
	\]

	For a totally ordered set $\graphVertices$ we consider the linearly oriented graph $\graphGeneral_J$ with vertex set $\graphVertices$  and set of edges $\graphEdges= \{i\to i+1\mid i,i+1\in \graphVertices\}$, where $i+1$ is the cover of $i$.  Specifically we use  
	\begin{equation*}
	\begin{gathered}
	\graphGeneral_r \; = \; \begin{tikzpicture}
			\newcommand{\leng}{1.5}
			\foreach \i/\j/\k in {7/4/{r},0/{-4}/{\negidx{r}},2/-2/{\negidx{1}},3/-1/{\negidx{\half}},4/1/{\half},5/2/{1}}{
				\node[] (\j) at (\i*\leng,0) {$\k$};
			}
			\node [] (-3) at (1*\leng,0) {$\cdots$};
			\node [] (3) at (6*\leng,0) {$\cdots$};
			\foreach \i/\j/\la in {1/2/1,3/4/r,2/3/2}{
				\draw [thin,->] (\i) to node[above]{$-\la$} (\j);
				\draw [thin,->] (-\j) to node[above]{$\la$} (-\i);
			}
			\draw [thin,->] (-1) to node[above]{$0$} (1);
		\end{tikzpicture},\\
	\graphGeneral'_r \; = \; \begin{tikzpicture}
			\newcommand{\leng}{1.85}
			\foreach \i/\la in {-3/{-r},-2/{\cdots},-1/-1,0/0,1/1,2/{\cdots},3/r}{
				\node[] (\i) at (\i*\leng,0) {$\la$};
			}

			\foreach \i/\la  [evaluate=\i as \j using int(\i+1)]  in {-3/{r+\half},-2/{1+\half},-1/{\half},0/{-\half},1/{-1-\half},2/{-r-\half}}{
				\draw [thin,->] (\i) to node[above]{$\la$} (\j);
			}
		\end{tikzpicture}.
		\end{gathered}
	\end{equation*}
	We denote the source and target of an edge $e$ by $\source{e}$ and $\target{e}$ respectively.

	\begin{definition}\label{convention: quantum group}
		The  \emph{quantum group} $\Uq(\graphGeneral_J)\cong \Uq(\gl_{|\graphVertices|})$ associated to $\graphGeneral_J$ is the $\C(q)$-algebra generated by $E_i,F_i$ for $i\in \graphEdges$ and $D_j ^{\pm1}$ for $j\in \graphVertices$, subject to the following relations 
		\begin{equation}\label{glrel1}
		\begin{gathered}
				 D_j D_j\inv = 1=  D_j\inv D_j, \quad  D_jE_iD_j\inv = q^{\delta_{\source{i},j}-\delta_{\target{i},j}} E_i, \\
				 D_jF_iD_j\inv = q^{-\delta_{\source{i},j}+\delta_{\target{i},j}} F_i ,\quad
				E_iF_{i'} -F_{i'}E_i = \delta_{ii'} \frac{K_i-K_i\inv}{q-q\inv },\\
			\end{gathered}
		\end{equation}
		\begin{equation}\label{glrel2}
			\begin{split}
			\mbox{ if } i-i':&\quad\quad	\quad 	\quad 		\quad E_{i}E_{i'} =E_{i'}E_i  ,\quad\quad\quad\quad	\quad		\quad	\quad\quad	\quad 	 F_iF_{i'} =F_{i'} F_i , \\
			\mbox{ if }i\noAdjacent i':&\quad	 E_i^2 E_{i'} - [2]_{q} E_i E_{i'} E_i  + E_{i'} E_i^2 = 0 \,\;\;\text{and} \;\;\; F_i^2 F_{i'} - [2]_{q} F_i F_{i'} F_i  + F_{i'} F_i^2 = 0 \, ,
			\end{split}
		\end{equation}
		where $\delta$ is the Kronecker symbol and $\displaystyle K_i =D_{\source{i}} D_{\target{i}}\inv\in\basis_{\target{i}} $.
	\end{definition}

	For $p\in\C(q)^*$, $A,B\in\Uq(\graphGeneral_J)$ let $[A,B]_p=AB-p BA$. Note that $[A,B]_p=-p [B,A]_{p^{-1}}$ and 
	\begin{equation}\label{fancyJacobi}
	[A,[B,C]_p]+[B,[C,A]_p]+[C,[A,B]_p] =0\quad\text{for all $A$, $B$, $C\in\Uq(\graphGeneral_J)$, $p\in\C(q)$}.
	\end{equation}

	The quantum group $\Uq(\graphGeneral_J)$ is a Hopf algebra with comultiplication $\Delta$ given by \[
		\Delta(E_i) = E_i\otimes1 +  K_{i} \otimes E_i,\quad \Delta(F_i) = F_i\otimes  K_{i}\inv  + 1\otimes F_i,\quad \Delta(D_j) = D_j\otimes D_j.
	\]
	Let $\vecrep$ be the natural representation of $\Uq(\graphGeneral_J)$. It has basis $\basis_j, j\in J$ and action $F_i\basis_j=\delta_{j,\sigma(i)}\basis_{\tau(i)}$, $E_i\basis_j=\delta_{j,\tau(i)}\basis_{\sigma(i)}$. Thus, the linear graph $\graphGeneral_J$ is in fact the crystal graph of $\vecrep$. %We use the vertices of $\graphGeneral_J$ to label the standard basis vectors of $\vecrep$ and edges of $\graphGeneral$ to label the generators of $Uq(\graphGeneral)$ as well as the Dynkin diagram of $\Uq(\graphGeneral) $.

	\newcommand{\involutionCoideal}{\tau}	
	%We are interested in the graphs $\Gamma_r$ and $\Gamma_{r>}$. 
Associated to the  linear graphs $\graphGeneral_r$, respectively $\graphGeneral_r'$, there is a quantum symmetric pair coideal subalgebra, \cite{Letzter-symmetric-pairs-quantized-enveloping} \cite{Letzter-coideal-subalgebras},  corresponding to the type $\operatorname{AIII}$ Satake diagram without black nodes:
	\begin{equation}\label{ourgraphs}
	%\begin{gathered}
			\dynkin[%
			edge length=0.75cm,
			labels={r,r{-}1,1,0,-1,1{-}r,-r},
			label directions={above,above,above,above,above,above,above,above},
			involution/.style={blue!50,stealth-stealth,thick},
			involutions={1{7};26;35}
			]{A}{oo.ooo.oo}
			\;,\quad\quad
				\dynkin[%
				edge length=0.85cm,
				labels={{r+\half},{1+\half},\half,-\half,{-1-\half},{-r-\half}},
				label directions={above,above,above,above,above,above},
				involution/.style={blue!50,stealth-stealth,thick},
				involutions={1{6};25;34}
				]{A}{o.oooo.o}
				\;.
		\end{equation}
                 Since the finite dimensional representation theory of  $\graphGeneral_r'$ was addressed in detail in  \cite{Watanabe-crystal-basis-quantum-symmetric-pair}, we restrict ourselves to the case of  $\graphGeneral_r$ and thus consider only the first diagram in \eqref{ourgraphs}.     
		
		\begin{definition}
			The \emph{quantum symmetric pair coideal subalgebra} $\Uq'(\graphGeneral_r)$, associated to the first Satake diagram in \cref{ourgraphs}, is the subalgebra of $\Uq(\graphGeneral_r)$ generated by the following elements:
		\begin{gather*}
			\begin{aligned}
				B_{\pm i} \coloneqq 	F_{\pm{i}}+  E_{\mp i}K_{\pm i}\inv   \text{ for }	1\leq i\leq r,\quad
				B_0&\coloneqq F_0+q\inv E_0 K_0\inv
				,\quad 
				\dCoideal_j^{\pm1}:=(D_{i}D_{\negidx{j}} )^{\pm1} \text{ for }  j\in \graphVertices
				.
			\end{aligned}
		\end{gather*}
		\end{definition}
 The defining relations for a presentation in these generators are:

	The $\{B_i\}_{i\ne 0} $ and $\{D_jD_{-j}\}$ satisfy the \emph{$ \Uq(\graphGeneral_{r})$-relations}\footnote{That is the usual quantum group relations with $E_i, F_i,D_j$ replaced by the $\{B_{-i}\}_{i>0}$, $\{B_{i}\}_{i>0}$, $\dCoideal_j$ respectively.}  from \cref{glrel1}, \cref{glrel2}, the relation $[\bo,\dCoideal_j]=0$, and the following quantum Serre relations:
			\begin{align}
				\bo^2 B _{\pm1}  -[2]  \bo B _{\pm1}   \bo +  B _{\pm1}  \bo^2 =   B _{\pm1}
				,\,\quad
				B _{\pm1} ^2  \bo -[2]   B _{\pm1} \bo  B _{\pm1} +  \bo  B _{\pm1}^2 = 0, \label{serre relations 1}
			\end{align}
	or equivalently $[\bo,[\bo,B _{\pm1}]_{q^{-1}}]_q=B _{\pm1}$, $[B _{\pm1},[B _{\pm1},\bo]_{q^{-1}}]_q=0$ in the notation from \eqref{fancyJacobi}.
		
	We set $\kCoideal_i =K_i K_{-i}\inv$ for $i<0$.  The following  lemma is clear from the defining relations:		
	\begin{lem}\label{involution} The assignment $\involutionCoideal(B_i) =B_{-i}$,  $\involutionCoideal(\dCoideal_j) = \dCoideal_j\inv$ defines an algebra involution $\involutionCoideal$ on $\coideal$. 
	\end{lem}

	The pair $(\coideal(\graphGeneral_r),\Uq(\graphGeneral_r))$ is a quantum analogue of the symmetric pair $(\gl_{r+1}\times \gl_{r+1},\gl_{2r+2})$.  In contrast to the classical situation, $\coideal(\graphGeneral_r)$ is not a Hopf subalgebra of $\Uq(\graphGeneral_r)$, e.g. \[
				\Delta(\eCoideal) = 1\otimes \eCoideal + \eCoideal\otimes K_1\inv + (K_1\inv K_{-1}-1)\otimes E_{-1}K_1\inv \notin\coideal(\graphGeneral_r)\otimes\coideal(\graphGeneral_r).
			\]
			Instead, it is a right coideal subalgebra and hence acts on the tensor powers of $\vecrep$.

	\subsection{Schur--Weyl duality with the  type $\operatorname{B}$ Hecke algebra  \texorpdfstring{$\heckeSpecialized{d}$}{}}\label{hecke algebra section} Let $d\in \bbN$. 
		\begin{definition}\label{DefHecke}
			The Hecke algebra $\heckeSpecialized{d}$ is generated by $H_0,H_1,\ldots,H_{d-1}$ subject to 
			\begin{eqnarray*}
				H_0^2=1, \qquad  H_{0}H_1H_0H_1=H_{1}H_0H_{1}H_0
					,\\
				H_i H_{i+1} H_i = H_{i+1} H_i H_{i+1}
					,\qquad 
				H_i H_j = H_j H_i \text{ for } |i-j|>1.
			\end{eqnarray*}
		\end{definition}
		The classical counterpart of $\heckeSpecialized{d}$ is the group algebra of the type $\operatorname{B}$ Weyl group $\weylGroup[d]$.
		\begin{definition}\label{Js}
			The algebra $\heckeSpecialized{d}$ contains the recursively defined \emph{Jucys--Murphy elements} \[
				J_1 \coloneqq H_0,\qquad J_i = H_{i-1} J_{i-1} H_{i-1}, \text{ for } 2\leq i\leq d.
			\] 
		which pairwise commute, \cite{ArKo-hecke-algebra}.
		\end{definition}
			\begin{definition}\label{Kmatrix}
                 The \emph{$\kmatrix$-matrix} is the endomorphism of $\vecrep$ given by the matrix $\kmatrix=(\delta_{i,2r-i+2})$.  
                 	\end{definition}
           
           There is an action of $\heckeSpecialized{d}$ on $\vecrep^{\otimes d}$ given as follows:  $H_0$ acts on the first tensor factor $\vecrep$ via $\kmatrix$
		and $H_i$ acts on the $i$-th and $(i+1)$-th tensor factor via 
		$H \colon \vecrep^{\otimes 2}\to \vecrep^{\otimes 2}$ defined by 
		\begin{eqnarray}
			H (\basis_{i}\otimes \basis_{j} )&=&
			\begin{cases}
				\basis_{j}\otimes \basis_{i} +(q \inv-q)\cdot \basis_{i}\otimes \basis_{j} & \text{if } i \dottedArrowLeft j  \text{ in } \graphGeneral, 
				\\
					\basis_{j}\otimes \basis_{i} & \text{if } j \dottedArrowLeft i \text{ in } \graphGeneral,
				\\
				q\inv \cdot \basis_{i}\otimes \basis_{i} & \text{if }i=j,
			\end{cases}
			\label{eqn:R-matrix generic form}
		\end{eqnarray}
		where $i \dottedArrowLeft j$ means there exists a path in $\graphGeneral$ from $j$ to $i$. 
				\begin{thm}[Ehrig--Stroppel, Bao--Wang]\label{quantum Schur Weyl} The following \emph{Schur--Weyl duality} result  holds: \hfill\\
			The action of $\coideal$ on $\vecrep^{\otimes d }$ and $\heckeSpecialized{d}$ generate each other's centralizer.
		\end{thm}
		\begin{cor}\label{Cor:ss}
			The $\coideal$-action on $\vecrep^{\otimes d}$ is semisimple.
		\end{cor}
		\begin{proof}[Proof of Schur--Weyl duality and \Cref{Cor:ss}]
			The  Schur--Weyl duality is \cite[Thm. 10.4]{EhSt-nw-algebras-howe}. The double centralizer Theorem implies then that the action of $\coideal$ on $\vecrep^{\otimes d}$ is semisimple, since $\heckeSpecialized{d}$ is semisimple, \cite[Rk. 5.4, Thm. 5.5]{DipperJames}.
		\end{proof} 
		Similar to $\weylGroup[d]$-modules, \cite[5.5.4]{GeckPfeiffer},  \cite[5.5.4]{MaSt-complex-reflection-groups}  simple modules over $\heckeSpecialized{d}$ are labeled by bipartitions of $d$ and are also known as \emph{Specht modules} $\specht({\lambda,\mu})$, see \cite[Thm. 3.10]{ArKo-hecke-algebra}. We denote the Specht module associated to the bipartition $(\lambda,\mu)$ by $\specht_q(\lambda,\mu)$. It has a basis labeled by \emph{standard bitableaux} of shape $(\lambda,\mu)$, which are tableaux of shape $\lambda$ and $\mu$ with total entries $\{1,2,\ldots,d\}$ such that, for each part, the entries in each row and column are strictly increasing. Given a standard bitableau $T$ of shape $(\lambda,\mu)$, the content, i.e. its column number minus its row number, of the box filled with $i$ in $T$ is denoted by $\content_T(i)$. The action of the Jucys--Murphy element $J_j$ on the basis element corresponding to $T$ is then $q^{-2\content_T(i)}$ if $i$ is in the first partition of $T$ and $-q^{-2\content_T(i)}$ if $i$ is in the second partition of $T$; see \cite{ArKo-hecke-algebra} and \cite[\S5]{AO-cyclotomic-JM} for details.
	
		\begin{lem}\label{Jucys--Murphy spectrum}
			The action of the Jucys--Murphy elements $J_i$ on $\vecrep^{\otimes d}$ is simultaneously  diagonalisable  with eigenvalues in $\pm q^{\pm 2\Z}$. The simultaneous eigenvalues separate the simple modules of $\heckeSpecialized{d}$. 
		\end{lem}
		\begin{proof}
	The representation $\vecrep^{\otimes d}$ of $\heckeSpecialized{d}$ is semisimple by \Cref{Cor:ss}, and, by construction of the simple modules, the Jucys--Murphy elements are diagonalisable with eigenvalues encoding the contents of the bipartition and therefore its shape, see \cite{ArKo-hecke-algebra}, \cite[Thm. 14]{AO-cyclotomic-JM} for details. 
		\end{proof}

	\subsection{The main player}
	In the rest of this paper we consider the linearly oriented graph $\linearGraph$ where the edges are labeled by $-1,0,1$ and vertices by $\negidx{1},\negidx{\half},\half, 1$ with its associated Satake diagram 
	\begin{equation}
		\linearGraph \; = \; \begin{tikzpicture}[baseline=-0.5ex]
			\newcommand{\leng}{2}
			\foreach \i/\j/\k in {2/-2/{\negidx{1}},3/-1/{\negidx{\half}},4/1/{\half},5/2/{1}}{
				\node[] (\j) at (\i*\leng,0) {$\k$};
			}
			\foreach \i/\j in {1/2}{
				\draw [thin,->] (\i) to node[above]{$-\i$} (\j);
				\draw [thin,->] (-\j) to node[above]{$\i$} (-\i);
			}
			\draw [thin,->] (-1) to node[above]{$0$} (1);
		\end{tikzpicture}\quad\quad\quad\quad
		\dynkin[
		edge length=1cm,
		labels={1,0,-1},
		label directions={above,above,above},
		involution/.style={blue!50,stealth-stealth,thick},
		involutions={13}
		]{A}{ooo}
		\;.
		\label{ourgraph}
	\end{equation}
	\begin{notation}\label{nota}
		We abbreviate $\coideal=U_q'(\gl_{2}\times \gl_{2})$ and consider the following elements of $\coideal$:
		\begin{eqnarray*}
			\bo = F_0 + q\inv E_0 K_0\inv ,
			\quad
			\kCoideal =K_1\inv K_{-1}= \dCoideal_\half\dCoideal_1\inv
			,\quad 
			\dCoideal_\half = D_\half D_{\negidx{\half}}
			,\quad
			\dCoideal_1 = D_1 D_{\negidx{1}}
			.
			%\label{coideal generator aliases}
		\end{eqnarray*}
				\end{notation}
			\begin{notation}\label{nota2}
According to \eqref{ourgraph}, we denote the standard basis of $\vecrep$ by $\basis_{\negidx{1}},\basis_{\negidx{\half}},\basis_{\half},\basis_{1}$. 			
	\end{notation}	
\section{Classical symmetric pair $(\gl_4,\gl_2\times\gl_2)$}\label{S:classical}
	\newcommand{\bOneClassical}{b_1}
	\newcommand{\bMinusOneClassical}{b_{-1}}
	\newcommand{\bZeroClassical}{b_0}
	\renewcommand{\phi}{\varphi}
	We next revisit the classical representation theory of $\gl_2\times \gl_2$ from the point of view of symmetric pairs. In particular, we provide a canonical labelling set of the irreducible representations.
For this section we fix the ground field $\mathbb{C}$. 
		The classical counterpart of  $\coideal$ is the universal enveloping algebra of the Lie algebra $\classicalFixPointSubalg$  of fixed points of  the involution $\theta$  on $\gl_4$ that is given by rotating a matrix by $180^\circ$. Explicitly, $\classicalFixPointSubalg$ is generated by the elements
		\[
			\bMinusOneClassical=\left(\begin{smallmatrix}
				0 & 1 \\
				& 0 \\
				&& 0 \\
				&&1& 0
			\end{smallmatrix}\right)
			,\;\;
			\bOneClassical=\left(\begin{smallmatrix}
				0 &  \\
				1 & 0 \\
				&& 0 &1 \\
				&&& 0
			\end{smallmatrix}\right)
			,\;\;
			\bClassical_0=\left(\begin{smallmatrix}
				&&&0 \\
				&&1 \\
				&1 \\
				0
			\end{smallmatrix}\right)
			,\;\; 
			d_0 = \left(\begin{smallmatrix}
				0   \\
				& 1 & \\
				&& 1 \\
				&&&0
			\end{smallmatrix}\right)
			,\;\; 
			\kappa_1= \left(\begin{smallmatrix}
				1 &  \\
				& 0 \\
				&& 0 \\
				&&&1
			\end{smallmatrix}\right)
			.
		\]		
		There is a Lie algebra isomorphism $ \phi:  \gl_2^+\times \gl_2^-:=\gl_2\times\gl_2\cong\classicalFixPointSubalg$  given by 
		\begin{equation*}
			(x,0) \mapsto \frac{1}{2}\left( \begin{smallmatrix}
				JxJ &  Jx\\
				xJ & x
			\end{smallmatrix} \right)
			,\quad 
			(0,x) \mapsto \frac{1}{2}\left( \begin{smallmatrix}
				JxJ &  -Jx\\
				-xJ &  x
			\end{smallmatrix} \right),
			\text{ where } J = \left( \begin{smallmatrix}
				&1\\
				1&
			\end{smallmatrix} \right).\label{iso gl2 times gl2}
		\end{equation*}
		Denote the standard basis of $\gl_2 $ by $
			l_0 = \left( \begin{smallmatrix}
				1 & 0\\
				0 & 0
			\end{smallmatrix} \right)
			,\;
			l_1 = \left( \begin{smallmatrix}
				0 & 0\\
				0 & 1
			\end{smallmatrix} \right)	
			,\;
			e = \left( \begin{smallmatrix}
				0 & 1\\
				0 & 0
			\end{smallmatrix} \right)
			,\;
			f = \left( \begin{smallmatrix}
				0 & 0\\
				1 & 0
			\end{smallmatrix} \right).
		$ Then $b_0,\bOneClassical,\bMinusOneClassical,d_0,\kappa_1$ correspond under $\varphi$ to $(l_0,- l_0),(e,e),(f,f),(l_0,l_0),(l_1,l_1)$ respectively.
		Thus, $ b_{\pm1},d_0,\kappa_1$ generate the diagonally embedded copy of $\gl_2\subset \classicalFixPointSubalg$.
		The inverse of $\varphi$ is given by
		\begin{equation}\label{whyhalves}
\def\arraystretch{1.5}
			\begin{tabular}{|cccc|}
				\hline
				$(l_0,0)$ & $(l_1,0)$ & $(e,0)$ & $(f,0)$ 
				\\
				\hline
				$\frac{1}{2} (d_0\pm b_0)$ &
				$\frac{1}{2} (\kappa_1\pm\csaCclassical\pm b_0)$ &
				$\frac{1}{2} (\bOneClassical\pm[d_0,\bOneClassical])$ &
				$\frac{1}{2} (\bMinusOneClassical\pm[d_0,\bMinusOneClassical])$ \\
				%\\
				%\hline
				%\hline
				%$(0,l_0)$ & $(0,l_1)$ & $(0,e)$ & $(0,f)$ \\
				%\hline
				%$\frac{1}{2} (d_0-b_0)$ &
				%$\frac{1}{2} (\kappa_1-\csaCclassical-b_0)$ &
				%$\frac{1}{2} (\bOneClassical-[b_0,\bOneClassical])$ &
				%$\frac{1}{2} (\bMinusOneClassical-[b_0,\bMinusOneClassical])$ \\
				%\hline
				%\hline
			         \hline
			\end{tabular}.
			\end{equation}
		In case of two signs the top/bottom here refers to the first/second copy $\gl_2^\pm$ of $\gl_2\times \gl_2$  and 
		 \[
			\csaCclassical =\phi(l_1-l_0,l_0-l_1)= [\bOneClassical,[b_0,\bMinusOneClassical]]=\left( \begin{smallmatrix}
				&&&1\\
				&&-1&\\
				&-1&&\\
				1&&&\\
				\end{smallmatrix} 			\right) \in \classicalFixPointSubalg
				.
		\]
		We will later quantise these elements and record for this also the following obvious fact: 
		\begin{lem}\label{classical CSA}
				The standard Cartan subalgebra $\mathfrak{h}$ of $\classicalFixPointSubalg$ has a basis given by the diagonal Cartan elements $d_0,\kappa_1$ and the anti-diagonal Cartan elements $b_0,\csaCclassical$.
		\end{lem}

	\subsection{Classical tensor power decomposition}
		Denote by $\vecrepClassical$  the defining representation of $\gl_4$  with standard basis $\basis_{\negidx{1}},\basis_{\negidx{\half}},$ $\basis_\half,\basis_1$. Restricted to $\classicalFixPointSubalg$, $\vecrepClassical $ decomposes into two irreducible summands:
		\[
			\vecrepClassical = \plusSubspace \oplus \minusSubspace, \text{ where } M_\pm = \operatorname{span}\{ w_\half^\pm , w_1^\pm\}, \text{ with } w_\half^\pm  = \basis_{\half}\pm \basis_{\negidx{\half}},\quad w_1^\pm = \basis_{1}\pm \basis_{\negidx{1}}.
		\] 
		The factors $\varphi(\gl_2^\pm)\subset \classicalFixPointSubalg$ act nontrivially only on $M_\pm$, the defining representation of $\gl_2$.
		
		There is an action of\, $\weylGroup[d]$ on $\vecrepClassical^{\otimes d}$, where the special type $\operatorname{B}$ generator $s_0\in \weylGroup[d]$ acts on the first tensor factor by the classical $K$-matrix
		with  $\pm1$ eigenspace decomposition $\vecrepClassical=M_+\oplus M_-$, and the other generators permute the tensor factors,  see e.g. \cite[\S 3.3]{MaSt-complex-reflection-groups} or specialise \eqref{eqn:R-matrix generic form}.
		\begin{definition}
			Let $d \in \N$. A \emph{two-row partition} $\lambda$ of $d$ is a partition of the form $(\lambda_1,\lambda_2)$ such that $|\lambda| =\lambda_1+\lambda_2=d$. A \emph{two-row bipartition} is a bipartition $(\lambda,\mu)$ such that both $\lambda$ and $\mu$ are two-row partitions.
		\end{definition}
		Recall that the irreducible \emph{polynomial representations}, i.e. the ones that occur as summands of tensor powers of the natural representation, of $\gl_2$ are the highest weight modules $L_\lambda$ labeled by two-row partitions $\lambda$. Thus, the irreducible \emph{polynomial representations} of $\classicalFixPointSubalg$, i.e. the summands of some $\vecrepClassical^{\otimes d}$,  are $L(\lambda ,\mu)$ for two-row {bipartitions} $(\lambda,\mu)$ with $L(\lambda ,\mu)=L_\lambda \boxtimes L_\mu$ as vector spaces.  
		\begin{thm}
			The multiplicity of  $L(\lambda,\mu)$ in the tensor product $\vecrepClassical^{\otimes d}$ is given by 	
		 \begin{equation}\label{classSchurWeyl}
				\vecrepClassical^{\otimes d} = \bigoplus_{|\lambda|+|\mu|=d} L(\lambda ,\mu)  \boxtimes \specht({\lambda,\mu}).
			\end{equation}
		\end{thm}
		\begin{proof}This follows from Schur-Weyl duality in this setting, see e.g. \cite[\S3.4]{MaSt-complex-reflection-groups}. 
		\end{proof}
		\begin{remark}
			One might define ``classical'' {Jucys-Murphy elements} in the group algebra of $\weylGroup$ inductively as $J_1=s_0$ and $s_i J_i s_i$.
			These elements simultaneously diagonalise $\vecrepClassical^{\otimes d}$ with eigenvalues $\pm 1$, but the eigenvalues do \emph{not} distinguish irreducible $\classicalFixPointSubalg$ or $\weylGroup$-modules. The correct classical analogue of the Jucys-Murphy elements are certain limits of $(J_i-1)/(q-q^{-1})$, see \cite[\S6]{AO-cyclotomic-JM}.
		\end{remark}
		Define the following elements, and note that $e_\pm,f_\pm$ act only nontrivially on $M_\pm$ and kill $M_\mp$,
		\[
			e_+ = \phi(e,0),\quad e_- = \phi(0,e),\quad f_+ = \phi(f,0),\quad f_- = \phi(0,f)\quad\in\classicalFixPointSubalg.
		\]

		\newcommand{\hwvecClassical}{\omega}
		\begin{thm}[$\bo$-eigenspaces]
			The following holds for representations $M$ of $\classicalFixPointSubalg$.
			\begin{enumerate}
				\item We have $f_\pm M (b_0\mid\lambda) \subset M (b_0\mid\lambda\mp1) $ and $e_\pm M (b_0\mid\lambda) \subset M (b_0\mid\lambda\pm1)$.
				
				\item The action of $b_0$ on $M=\vecrepClassical^{\otimes d}$ is diagonalizable with eigenvalues in $\Z$. In fact, \begin{equation}\label{classical b0 eigenbasis}
					b_0 (w_{i_1}^{\epsilon_1} \otimes \cdots \otimes  w_{i_d}^{\epsilon_d} ) =(\sum \epsilon_j \delta_{i_j , \half}) w_{i_1}^{\epsilon_1} \otimes \cdots \otimes  w_{i_d}^{\epsilon_d}.
				\end{equation}
				\item The space of maximal vectors in $M=L(\lambda,\mu)\boxtimes \specht (\lambda,\mu) \subset \vecrepClassical^{\otimes d}$, for any two-row bipartition $(\lambda,\mu)$ of $d$, 
				 is the $\specializedGroundRing S_d$-module (acting by permuting the factors) generated by \[
					\hwvecClassical(\lambda,\mu)
					 \coloneqq
					(w_\half^+ \wedge w_1^+)^{\otimes \lambda_2} 
					\otimes (w_\half^- \wedge w_1^-)^{\otimes \mu_2} 
					\otimes (w_\half^+) ^{\otimes (\lambda_1-\lambda_2)} 
					\otimes (w_\half^-) ^{\otimes (\mu_1-\mu_2)} .
			\]	
			Furthermore, we have $b_0\cdot \hwvecClassical(\lambda,\mu)
			=(\lambda_1-\mu_1)  \hwvecClassical(\lambda,\mu)$. 
			\end{enumerate}
		\end{thm}
		\begin{proof}
			For the first part note that $
				[b_0,f_\pm] = \mp f_\pm,$ and $ [b_0,e_\pm] = \pm e_\mp$. By definition the action of $b_0$ on $w_\half^{\pm}$ is given by $\pm 1$, and it kills $w_1^\pm$. Thus, the pure tensors in the  $w_\half^\pm, w_1^\pm$ are eigenvectors for $b_0$. One directly verifies \cref{classical b0 eigenbasis} and notes that this gives already an eigenbasis of $\vecrepClassical^{\otimes d}$. Since  $x\in \classicalFixPointSubalg$ acts as the sum  $ \sum 1\otimes  \cdots  \otimes 1 \otimes  x\otimes1 \otimes  \cdots \otimes 1$  of the action on each tensor factor, we have  $e_\pm \hwvecClassical(\lambda,\mu) =0 $ and $ b_0\hwvecClassical(\lambda,\mu)=(\lambda_1-\mu_1)  \hwvecClassical(\lambda,\mu)$. Since $\varphi(\gl_2^+)$ kills $(w_\half^- \wedge w_1^-)^{\otimes k}$ and $ (w_\half^-) ^{\otimes l}$ for any $k,l$, it turns out that $
			\varphi(\gl_2^+ ) \hwvecClassical(\lambda,\mu)
				\cong\varphi(\gl_2^+)\hwvecClassical(\lambda,\emptyset) \cong L_\lambda,  
			$ and similarly, $\varphi(\gl_2^-)\hwvecClassical(\lambda,\mu) \cong \varphi(\gl_2^-)\hwvecClassical(\emptyset,\mu) \cong  L_\mu$. Since $ \classicalFixPointSubalg \hwvecClassical(\lambda,\mu)$ is irreducible, we get $ \classicalFixPointSubalg \hwvecClassical(\lambda,\mu) \cong L(\lambda,\mu)$.  Now, the maximal vectors of $L(\lambda,\mu)\boxtimes \specht (\lambda,\mu) $ are exactly the span of the vectors obtained from $\hwvecClassical(\lambda,\mu)$  by permuting the factors, since 
 \cref{classSchurWeyl} is in fact a bimodule decomposition by the Schur-Weyl duality for $\gl_2\times\gl_2$, see e.g. \cite[\S3.4]{MaSt-complex-reflection-groups}. 
 \end{proof}
		\subsection{Finite dimensional representations for  $\classicalFixPointSubalg$.}\label{classical finite dim reps}
		Finite dimensional representations of $\classicalFixPointSubalg$ and $\gl_2\times\gl_2$ are semisimple. The  isomorphism classes of  such irreducible representations for $\gl_2\times\gl_2$  are determined by their highest weights which are in bijection with \[
 			\{ (\lambda_1,\lambda_2) \in \C^2  \mid  \lambda_1-\lambda_2\in \N \}^2
		\]
		such that $(l_0,0)$ $(0,l_0)$ $(l_1,0)$ $(0,l_1)$ act on a maximal vector by $\lambda_1$, $\mu_1$, $\lambda_2$, $\mu_2$ respectively. 	
				
		Recalling the Cartan subalgebra $\mathfrak{h}\subset\classicalFixPointSubalg$ from \Cref{classical CSA} we summarize our observations: 
		\begin{thm}
			Taking the highest weight, $L\mapsto (\lambda,\mu)$, induces a bijection
			 \begin{align}\label{fancyhighestweights}
				\left\{ 
				\begin{minipage}[c]{6.5cm}
				isomorphism classes of 
				finite dimensional \\irreducible representations of $\classicalFixPointSubalg$
				\end{minipage}
				\right\}
				\quad& \cong\quad
				\{ (\lambda_1,\lambda_2) \in \C^2  \mid  \lambda_1-\lambda_2\in \N \}^2
			%	\\
			%	L\quad &\mapsto \quad  (\lambda,\mu), \nonumber
			\end{align}
			with inverse $(\lambda,\mu)\mapsto L(\lambda,\mu)$ such that $\lambda_1,\mu_1,\lambda_2,\mu_2$ are the eigenvalues for the action of the elements 
			$\frac{1}{2}(b_0+d_0)$, $\frac{1}{2} (d_0- b_0)$, $\frac{1}{2} (\kappa_1+\csaCclassical+ b_0)$
				$\frac{1}{2} (\kappa_1-\csaCclassical - b_0)$ respectively
			on a highest weight vector.
		\end{thm}
		\begin{remark} The action of $	\bClassical_0 ,\csaCclassical, d_0,\kappa_1 $ on a highest weight vector of $L(\lambda,\mu)$ is given by multiplication by $\lambda_1-\mu_1$,  $\lambda_2-\mu_2-\lambda_1+\mu_1$, $\lambda_1+\mu_1$ and $\lambda_2+\mu_2$ respectively.
	\end{remark}
	
	%%%%%%%%%%%%%%%%%%%%%%%%%%%%%%%%%%%%%%%%%%%%%%%%%%%%%%%%%%%%%%%%%%%%%

\section{Quantized Cartan subalgebra and PBW theorem}

	\newcommand{\PBWmonomial}{\operatorname{m}}
	\newcommand{\PPBWmonomial}{\operatorname{p}}
	\newcommand{\op}{\operatorname}
	We construct a quantum analogue of the Cartan subalgebra from \cref{classical CSA} inside $\coideal$ and a triangular decomposition for $\coideal$ from \cref{nota}. This will allow us to define Verma modules.
	\subsection{Quantized Cartan subalgebra}	
		We quantize the elements from \cref{classical CSA} and $\mathfrak{h}$:
			\begin{definition}\label{DefsXYZW}
				Define the following elements in $\coideal$: \[
					\eNewCoideal = \qcommutator{\bo}{\eCoideal} 
					,\quad 
					\fNewCoideal = \qcommutator{\bo}{\fCoideal}
					,
					\quad
						Z = [\eCoideal,\fNewCoideal]_{q^{-1}} ,\quad W = [\fCoideal,\eNewCoideal]_{q^{-1}} \in \coideal.
					\] 
				Let $\csaZW$ be the subalgebra of $\coideal$ generated by $\dCoideal_\half,\dCoideal_1 , \bo$ and $Z$.
			\end{definition}
		\begin{lemma} \label{easycomm} 
	The following defining relations directly imply the relations in the next two lines: 
		\begin{equation*}
						\dCoideal_\half \eCoideal  = q \eCoideal \dCoideal_\half 
						,\quad
						\dCoideal_\half {\fCoideal}  = q\inv {\fCoideal} \dCoideal_\half
						,\quad
						\dCoideal_1 \eCoideal = q\inv \eCoideal \dCoideal_1
						,\quad
						\dCoideal_1 \fCoideal = q \fCoideal \dCoideal_1
						,
	\end{equation*}
	\begin{align*}\label{below}					
						\dCoideal_\half X &=q X \dCoideal_\half
						,&\quad
						\dCoideal_\half Y &=q\inv Y \dCoideal_\half,
						&\quad
						\dCoideal_1 \fNewCoideal &= q \fNewCoideal \dCoideal_1
						,&\quad
						\dCoideal_1 \eNewCoideal &=q\inv \eNewCoideal \dCoideal_1.
						\\
						\kCoideal \eCoideal  &=q^2 \eCoideal \kCoideal 
						,&\quad
						\kCoideal {\fCoideal}  &=q^{-2} {\fCoideal} \kCoideal
						,&\quad
						\kCoideal \eNewCoideal&=q^2 \eNewCoideal \kCoideal 
						,&\quad
						\kCoideal {\fNewCoideal} &=q^{-2} {\fNewCoideal} \kCoideal
						.
						\end{align*}
		\end{lemma}
		
		\begin{thm}\label{Cartanalg}
		The algebra $\csaZW$ is commutative, since the following relations hold with $i,j\in\graphVertices$: 
			\begin{equation}
				[ \dCoideal_i,\dCoideal_j]   = [\bo, \dCoideal_i ]= [Z,\bo]=[\dCoideal_i,Z] =0,\quad\text{and}\quad [W,\dCoideal_i ] = [W,\bo] =[Z,W] =0.
				\label{csa}
			\end{equation}	
		\end{thm}		
	
\begin{proof}
			Clearly,  $[\dCoideal_i,\dCoideal_j]=0$ and  $[\bo,\dCoideal_i]=0$ by definition. Using the latter  we calculate $[Z,\bo]$ as
			\begin{eqnarray*}
			\underbrace{\eCoideal \fNewCoideal \bo}_{\scalebox{.7}{\;\;\;\circled{1}}}-  \underbrace{ q\inv \bo \fCoideal \eCoideal \bo }_{\scalebox{.7}{\;\;\;\;\;\;\;\circled{2}}}+ \underbrace{ q^{-2} \fCoideal \bo \eCoideal \bo}_{\scalebox{.7}{\;\;\;\;\;\;\;\circled{3}}}-\underbrace{\bo \eCoideal \bo \fCoideal } _{\scalebox{.7}{\;\;\;\;\;\circled{4}}}+ \underbrace{ q\inv  \bo \eCoideal \fCoideal \bo}_{\scalebox{.7}{\;\;\;\;\;\;\;\circled{5}}} +\underbrace{ q \inv \bo \fNewCoideal \eCoideal} _{\scalebox{.7}{\;\;\;\;\;\circled{6}}} \text{ with}
			\end{eqnarray*}
			\begin{eqnarray*}
		\text{\scalebox{.7}{\circled{1}}}-{\scalebox{.7}{\circled{4}}}&=&\eCoideal(\bo \fCoideal\bo-q\inv [2][2]\inv \fCoideal\bo^2-[2]\inv\bo^2 \fCoideal+[2]\inv \fCoideal)-[2]\inv\bo^2 \eCoideal \fCoideal\\
		&=&-[2]\inv(q^{-2} \eCoideal \fCoideal\bo^2+ \bo^2 \eCoideal \fCoideal)\\
		{\scalebox{.7}{\circled{6}}}+{\scalebox{.7}{\circled{3}}}&=&q^{-2}((q[2][2]\inv \bo\bo \fCoideal-\bo \fCoideal\bo+[2]\inv \fCoideal \bo\bo-[2]\inv \fCoideal)\eCoideal+[2]\inv \fCoideal \eCoideal \bo^2)\\
		&=&[2]\inv(\bo^2 \fCoideal \eCoideal +q^{-2} \fCoideal \eCoideal \bo^2).
		\end{eqnarray*}
		We get $[Z,\bo ] =-[2]\inv (q^{-2}[\eCoideal , \fCoideal]\bo^2 + \bo^2 [\eCoideal , \fCoideal])+q\inv\bo [\eCoideal, \fCoideal] \bo=(\frac{q^{-2}+1}{-[2]}+q\inv)\bo^2 [\eCoideal, \fCoideal] $, and thus $[Z,\bo ]=0$. 	Expanding the definitions we see that $[\dCoideal_i,Z]=0$. Thus, $\csaZW$ is commutative.

			Applying the involution from  \cref{involution} to $[Z,\bo]=0$  gives $[W,\bo] =0$, and $Z,W$ commute since
			\begin{eqnarray*}
				[Z,W] &=& Z \fCoideal \eNewCoideal -q\inv Z \eNewCoideal \fCoideal - \fCoideal \eNewCoideal Z + q^{-1}\eNewCoideal \fCoideal Z \\
				&=&
				q( \fCoideal Z - q^{-2} [2] \kCoideal \inv \fNewCoideal ) \eNewCoideal -q^{-2}  ( \eNewCoideal Z -  [2] \kCoideal \inv \eCoideal) \fCoideal -\fCoideal \eNewCoideal Z +q\inv \eNewCoideal \fCoideal Z
				\\
				&=& -\fCoideal [\eNewCoideal,Z]_q +q\inv \eNewCoideal[\fCoideal, Z]_{q\inv } -q\inv [2] \kCoideal\inv \fNewCoideal \eNewCoideal +q^{-2} [2] \kCoideal\inv \eCoideal \fCoideal
				\\
				&=& [2] (-\fCoideal \kCoideal\inv \eCoideal  +q^{-3} \eNewCoideal \kCoideal\inv  \fNewCoideal  -q\inv \kCoideal\inv \fNewCoideal \eNewCoideal +q^{-2} \kCoideal\inv \eCoideal \fCoideal)
				\\
				&=& [2] \kCoideal\inv ( q^{-2}[\eCoideal,\fCoideal] +q^{-1} [\eNewCoideal, \fNewCoideal] ) =0 ,
			\end{eqnarray*}
			where we used \cref{EZ FZ commutator} and \cref{XZ YZ commutator} and finally \cref{csa and [XY]} from the lemmas below.
					\end{proof}		
		\begin{lem}\label{commutation relations}
			Using \Cref{nota}, the following equalities hold in $\coideal$.
			\begin{align}
				{\eCoideal }{\eNewCoideal } = q^{-1}{\eNewCoideal }{\eCoideal }
				,\qquad 
				{\fCoideal }{\fNewCoideal } = &q^{-1}{\fNewCoideal }{\fCoideal }
				,\qquad
				{[{\bo},{\eNewCoideal }]_q =  \eCoideal  }
				,\qquad 
				{[{\bo},{\fNewCoideal }]_q =  \fCoideal  }
				,\label{serre relations as commutators}
				\\
				[\eNewCoideal,\fNewCoideal ]  = -q^{-1} [\kCoideal ;0]&
				,\qquad 
				\kCoideal \eNewCoideal = q^2 \eNewCoideal \kCoideal
				,\qquad
				\kCoideal \fNewCoideal = q^{-2} \fNewCoideal \kCoideal.
				\label{csa and [X\fNewCoideal ]}
					\end{align}
			\end{lem}		
		\begin{proof}
		Note that  \cref{serre relations as commutators} are equivalent to the relations \cref{serre relations 1}, e.g. $[\bo,X]=[\bo,[\bo\eCoideal]_{q\inv}]_q= \eCoideal $. 
			For \cref{csa and [XY]} we have 	$\kCoideal \fNewCoideal = \kCoideal \bo \fCoideal - q\inv \kCoideal\fCoideal \bo
				=q^{-2}   \bo \fCoideal \kCoideal-  q\inv q^{-2}\fCoideal\kCoideal \bo 
				= q^{-2} \fNewCoideal \kCoideal, $  and the analogue  for $\eNewCoideal$ follows by applying  the involution in \cref{involution}. 		\end{proof}
			
				\begin{lemma}\label{MPIrelations}
				The following relations hold
				\begin{align}
				[\fCoideal ,Z]_{q\inv} = q^{-2}[2] \kCoideal \inv \fNewCoideal 
				,& \quad
				[Z,\eCoideal ]_{q\inv} = -q^{-3}[2] \eNewCoideal \kCoideal \inv  + (q\inv -q^{-5})\eCoideal \bo \kCoideal \inv \label{EZ FZ commutator}
				,\\
				[\eNewCoideal,Z]_q = [2] \kCoideal \inv \eCoideal , &\quad
				[Z,\fNewCoideal]_{q} = -q[2] \fCoideal \kCoideal \inv - (q^3 -q^{-1}) \fNewCoideal \bo \kCoideal \inv  
				\label{XZ YZ commutator}
				,\\
				[\eCoideal ,W]_{q\inv} = q^{-2}[2] \kCoideal \eNewCoideal, &
				\quad [\fCoideal ,W]_q = q^{-2}[2]  \fNewCoideal  \kCoideal  -(1-q^{-4} ) \fCoideal \bo \kCoideal
				,\\
				[\fNewCoideal ,W]_q = [2] \kCoideal  \fCoideal  , &\quad 
				[\eNewCoideal,W]_{q\inv } = q^{-4}[2] \kCoideal \eCoideal  + (q\inv -q^{-5})  \bo\kCoideal \eNewCoideal
				\label{YW commutator}
				, \\
				W =q^{-2} Z - q^{-2}(q-q\inv &)[\kCoideal;0]\bo -(q^{-2}-1)\fCoideal \eNewCoideal -q^{-2} (q-q\inv )\fNewCoideal \eCoideal.\label{Z in W}
				\end{align}
			\end{lemma}

		\begin{proof}The proof is a calculation using  \Cref{commutation relations}.
		\end{proof}
						\begin{remark} The construction in \Cref{Cartanalg} fits into the more general framework of \cite{Le-cartan-coideal}, where Letzter—motivated in part by discussions with the first author based on our small example—introduced analogues of Cartan subalgebras for quantum symmetric pair coideal subalgebras of reductive Lie algebras. Up to applying the bar involution, our algebra coincides with Letzter’s construction. Due to their rather involved definition, these algebras have not been studied extensively. While they share some of the desirable properties of quantized Cartan and Borel subalgebras, they do not share all of them, as we will explain in detail for our fundamental example.
		\end{remark}

	\subsection{A PBW type basis}
		\newcommand{\weightDBZW}{\Xi}
		\newcommand{\filt}{\mathcal{F}}

		Fix a total ordering on $R = \{\fCoideal,\fNewCoideal,\dCoideal_i,\bo,Z,\eCoideal,\eNewCoideal\}$ given by \[
				\fCoideal > \fNewCoideal>\eCoideal>\eNewCoideal> \bo>Z> \dCoideal_\half>\dCoideal_1.
		\]
	We abbreviate 
	${\PBWmonomial}_{f,y,e,x,b,z,k_\half,k_1}\coloneqq\fCoideal^f \fNewCoideal^y \eCoideal^e\eNewCoideal^x \bo^b Z^z  \dCoideal_\half^{k_\half} \dCoideal_1^{k_1}$  for  $ f,y,b,z,e,x\in \N,\,k_i\in \Z $, and we use the convention 
	${\PBWmonomial}_{f,y,e,x,b,z,k_\half,k_1}=0$ if some $ f,y,b,z,e,x\notin \N$, or some $k_i\notin \Z $.
		\begin{thm}[PBW basis]\label{PBW basis}
			The following \emph{PBW monomials} form a basis of $\coideal$:\[
				\left\{{\PBWmonomial}_{f,y,e,x,b,z,k_\half,k_1} \mid f,y,b,z,e,x\in \N,\,k_\half,k_1\in \Z \right\}.
			\]
			There is a filtration $\filt$ on $\coideal$, where $\filt_{\leq i}$ is defined as 
				the span of $\{\PBWmonomial_{f,y,e,x,b,z, k_\half,k_1} \mid f+y+e+x\leq i\}$, 
			with associated graded algebra  $\op{gr}_\filt\coideal\cong \groundring[\mathtt{f},\mathtt{y},\mathtt{e},\mathtt{x},\mathtt{b},\mathtt{z}][\mathtt{k}_\half^{\pm1},\mathtt{k}_1^{\pm1}]$. 					
			\end{thm}
		\begin{proof}  By applying the commutator relations from \cref{commutation relations} and \Cref{MPIrelations}, any monomial in the generators can be written as a linear combination of PBW monomials. We need to show that the  PBW monomials are linearly independent. For this let $P$ be the $\groundring$-vector space on basis $\PPBWmonomial_\lambda$ where $\lambda = {(f,y,b,z,e,x,k_\half,k_1)}$ runs through the same labelling set as the PBW monomials  ${\PBWmonomial}_\lambda$. 
		
We define an action of $\coideal$ on $P$ by letting the generators act on $\PPBWmonomial_\lambda$ for  $\lambda = {(f,y,b,z,e,x,k_\half,k_1)}$ as
			\begin{eqnarray}
				\fCoideal \cdot\PPBWmonomial_\lambda &=&\PPBWmonomial_{(f+1,y,e,x,b,z,k_\half,k_1)}
				,\label{fformula}\\
				\dCoideal_\half \cdot\PPBWmonomial_\lambda &=& q^{e+x-f-y}\PPBWmonomial_{(f,y,e,x,b,z,k_\half+1,k_1)}
				,\label{ddformula}\\
				\dCoideal_1 \cdot\PPBWmonomial_\lambda &=& q^{f+y-e-x}\PPBWmonomial_{(f,y,e,x,b,z,k_\half,k_1+1)}
				,\label{d1formula}\\
				\bo \cdot\PPBWmonomial_\lambda &=& q^{x+y-e-f}\PPBWmonomial_{(f,y,e,x,b+1,z,k_\half,k_1)} 
				+q^{y-f-1}[y]\PPBWmonomial_{(f+1,y-1,e,x,b,z,k_\half,k_1)}
					+[f]\PPBWmonomial_{(f-1,y+1,e,x,b,z,k_\half,k_1)}\nonumber\\
					&&
					+q^{y-f}[e]\PPBWmonomial_{(f,y,e-1,x+1,b,z,k_\half,k_1)}
					+q^{y-f-e+x-1}[x]\PPBWmonomial_{(f,y,e+1,x-1,b,z,k_\half,k_1)}
					,\label{bformula}\\
 				\eCoideal \cdot\PPBWmonomial_\lambda &=& 
				q^{1-2y+2e+2x-f}[f]/(q-q\inv) \PPBWmonomial_{(f-1,y,e,x,b,z,k_\half+1,k_1-1)} \nonumber
				\\
				&&-q^{f+2y-2e-2x-1}[f]/(q-q\inv)  \PPBWmonomial_{(f-1,y,e,x,b,z,k_\half-1,k_1+1)} 
				+q^{-y}  \PPBWmonomial_{(f,y,e+1,x,b,z,k_\half,k_1)}\nonumber
				\\
				&&
				-q^{2y-2e-2x-3} [y][y-1]\PPBWmonomial_{(f+1,y-2,e,x,b,z,k_\half-1,k_1+1)}+q^{-x-e}[y]  \PPBWmonomial_{(f,y-1,e,x,b,z+1,k_\half,k_1)}\nonumber
				\\
				&&
				-q^{-3e-x-2}( (q^{2y}-1)[y-1][x]+[2][y] [x]) \PPBWmonomial_{(f,y-1,e+1,x-1, b,z,k_\half-1,k_1+1)}\nonumber
				\\
				&& 
				- q^{-2e-2x-1} [e]([2][y]+(q^{2y}-1)[y-1])\PPBWmonomial_{(f,y-1,e-1,x+1,b,z,k_\half-1,k_1+1)}\nonumber
				\\
				&& 
				+(q^{-2e-x-1}(q^e-q^{-e})[2] [y]
				- q^{-3e-x-1}(q^{2y}-1)[y-1])\PPBWmonomial_{(f,y-1,e,x,b+1,z,k_\half-1,k_1+1)}.\label{eformula}
			\end{eqnarray}
			We verified that these formulas indeed define a representation using a computer, \cite{code}. 
			
			 Note that for $\nu=(0,0,e,x,b,z,k_\half,k_1)$ and $\nu'=(0,0,e+1,x,b,z,k_\half,k_1)$,
			  $\nu''=(0,1,e,x,b,z,k_\half,k_1)$
\begin{equation}\label{magice}
\eCoideal.\PPBWmonomial_\nu=\PPBWmonomial_{\nu'}\quad\text{and}\quad
\fNewCoideal.\PPBWmonomial_\nu=\PPBWmonomial_{\nu''}.
\end{equation}

The first equality is clear by \eqref{eformula}. Recalling the definition of $\fNewCoideal$, the second equality holds, since
\begin{align*}
\bo\fCoideal\PPBWmonomial_\nu &=q{\inv}\PPBWmonomial_{(1,0,0,0,b+1,z,k_\half,k_1)}+\PPBWmonomial_{\nu''}+q\inv[e]\PPBWmonomial_{(1,0,e-1,x+1,b,z,k_\half,k_1)},\\
q\inv\fCoideal\bo\PPBWmonomial_\nu&=q{\inv}\PPBWmonomial_{(1,0,0,0,b+1,z,k_\half,k_1)}
+q\inv[e]\PPBWmonomial_{(1,0,e-1,x+1,b,z,k_\half,k_1)}.
\end{align*}
Next, we claim that for any PBW monomial we can act on $1\in P$ and obtain exactly the corresponding basis elements of $P$. In formulas 		
\begin{equation}\label{PBWmonact}
			\PPBWmonomial_\lambda.1=\PPBWmonomial_\lambda.
			\end{equation}
                          This shows that the PBW monomials are linearly independent.                           	 		 		
Obviously, the claim holds for $\lambda=(0,0,0,0,0,0,k_\half,k_1)$. If it holds for $\lambda'=(0,0,0,0,0,z,k_\half,k_1)$ then also for  $\lambda=(0,0,0,0,0,z+1,k_\half,k_1)$ by definition of $Z$ and \Cref{magice}.  Thus, \Cref{PBWmonact}  holds for $\lambda=(0,0,0,0,0,z,k_\half,k_1)$ and then also  for  $\lambda=(0,0,0,0,b,z,k_\half,k_1)$ by \eqref{bformula}. We next show \Cref{PBWmonact}  for $\lambda=(0,0,0,x+1,b,z,k_\half,k_1)$ assuming it holds for $\lambda'=(0,0,0,x,b,z,k_\half,k_1)$. Indeed,
\begin{align*}
\bo\eCoideal\PPBWmonomial_{\lambda'}&=q^{x-1}\PPBWmonomial_{(0,0,1,x,b+1,z,k_\half,k_1)}+\PPBWmonomial_{(0,0,0,x+1,b,z,k_\half,k_1)}+
q^{x-2}[x]\PPBWmonomial_{(0,0,2,x-1,b,z,k_\half,k_1)}\\
q\inv\eCoideal\bo\PPBWmonomial_{\lambda'}&=q\inv(q^x\PPBWmonomial_{(0,0,1,x,b+1,z,k_\half,k_1)}+q^{x-1}[x]\PPBWmonomial_{(0,0,2,x-1,b,z,k_\half,k_1)}
\end{align*}
implies this. Then \eqref{eformula} implies the claim for $\lambda=(0,0,e,x,b,z,k_\half,k_1)$. Using the following slight generalisation of \eqref{eformula}  for  $\nu=(0,y,e,x,b,z,k_\half,k_1)$  with $\nu=(0,y+1,e,x,b,z,k_\half,k_1)$ 
\begin{align*}
\bo\fCoideal\PPBWmonomial_\nu&=q^{x+y-e-1}\PPBWmonomial_{(1,y,e,x,b+1,z,k_\half,k_1)}+q^{y-2}\PPBWmonomial_{(2,y-1,e,x,b,z,k_\half,k_1)}+
\PPBWmonomial_{\nu''}\\
&\quad+q^{y-1}[e]\PPBWmonomial_{(1,y,e-1,x+1,b,z,k_\half,k_1)}+q^{y-1-e+x-1}[x]\PPBWmonomial_{(1,y,e+1,x-1,b,z,k_\half,k_1)}
,\\
q\inv\fCoideal\bo\PPBWmonomial_\nu&=q{\inv}(q^{x+y-e}\PPBWmonomial_{(1,y,e,x,b+1,z,k_\half,k_1)}
+q^y[e]\PPBWmonomial_{(1,y,e-1,x+1,b,z,k_\half,k_1)}\\
&\quad+q^{y-1-e+x}[x]\PPBWmonomial_{(1,y,e+1,x-1,b,z,k_\half,k_1)}.
\end{align*}
the claim \eqref{PBWmonact} follows for $\lambda=(0,y,e,x,b,z,k_\half,k_1)$.  Finally it is obvious for  $\lambda=(f,y,e,x,b,z,k_\half,k_1)$. 

The formulas \eqref{fformula}--\eqref{eformula}  show that $\filt_{\leq i}\filt_{\leq j}\subset \filt_{\leq(i+j)}$. The description of  $\op{gr}_\filt\coideal$ follows. 
		\end{proof}
		\begin{remark}\label{grading}
			 Setting $\deg \fCoideal = \deg \fNewCoideal =-1, \deg \dCoideal_\half = \deg \dCoideal_1 = \deg \bo = \deg Z =0, $ and $ \deg \eCoideal = \deg \eNewCoideal = 1$ defines a $\Z$-grading on $\coideal$.
		\end{remark}

		\begin{definition}\label{plusminus}
			Let $\minusCoideal$ and $\plusCoideal$ be the subalgebras of $\coideal$ generated by $\fCoideal,\fNewCoideal$ respectively $\eCoideal,\eNewCoideal$.
		\end{definition}
			\begin{cor}\label{cor1}
			The subalgebras $\minusCoideal$, $\plusCoideal$ have bases $\left\{\fCoideal^f\fNewCoideal^y\right\}_{f,y\in \N}$ respectively $\left\{\eCoideal^e\eNewCoideal^x\right\}_{e,x\in \N}$. 
		\end{cor}
		\begin{proof}
			Since $q \eCoideal\eNewCoideal = \eNewCoideal\eCoideal$ and $q\fCoideal\fNewCoideal = \fNewCoideal\fCoideal$ by  \Cref{commutation relations}, the proposed elements span and hence are a basis by  \cref{PBW basis}.
		\end{proof}
		
The  subalgebras $\minusCoideal$, $\plusCoideal$ give, with the Cartan subalgebra $\CartanPart$, triangular decompositions:
			\begin{cor}[Triangular decompositions]\label{triangular}
			Multiplication defines vector space isomorphisms $\minusCoideal\otimes  \plusCoideal \otimes \CartanPart\cong \coideal$ and $\minusCoideal\otimes \CartanPart\otimes  \plusCoideal \cong \coideal$ with subalgebras  $\minusCoideal$, $\plusCoideal$, $\CartanPart$.			
			\end{cor}
	\begin{proof} The first claim follows directly from  \cref{PBW basis} using \cref{cor1} and \cref{Cartanalg}. For the second claim we additionally use \Cref{easycomm} and the commutator relations \eqref{EZ FZ commutator}, \eqref{XZ YZ commutator}. 
	\end{proof}
			\begin{definition}\label{Defweights}
			Given a $\coideal$-module $M$, we call simultaneous eigenspaces for  $\csaZW$ \emph{weight spaces}, its elements \emph{weight vectors} and the eigenvalues  \emph{weights}. A \emph{maximal vector} of $M$ is a weight vector killed by $\eCoideal $ and $\eNewCoideal $.
		\end{definition}
		\begin{notation}\label{notationweights}
			We abbreviate a weight space $M(\dCoideal_\half,\dCoideal_1,\bo,Z
			\mid q^a,q^b,c,d)$ as $M(a,b,c,d)$.  		
		\end{notation}

\section{Verma modules and classification of finite dimensional irreducible modules}
We next define Verma modules for $\coideal$ and study their irreducible quotients. 
	\subsection{Definition and basics}
		\newcommand{\bweight}{\beta}
	
		Let $U'_{>0}$ be the subspace of $\plusCoideal$ spanned by $\{\eCoideal^e\eNewCoideal^x\}_{e+x>0}$.
		\begin{lem}
			The subspace $U_{\geq0}=\plusCoideal\CartanPart$ is a subalgebra of $\coideal$ and  
			$ I = U'_{>0}\CartanPart $ is a two-sided ideal in $U_{\geq0}$. Moreover, $U_{\geq0}/I \cong \CartanPart$ as algebras and $\coideal$ is a free $\plusCoideal\CartanPart$-module.
		\end{lem}
		\begin{proof}
			That $U_{\geq0}$ is a subalgebra follows immediately from \cref{serre relations as commutators}, \cref{EZ FZ commutator}, \cref{XZ YZ commutator}. These relations also imply that $I$ is a two-sided ideal. By \cref{PBW basis}, the quotient is isomorphic to $\CartanPart$, and  $\coideal$ is a free $\plusCoideal\CartanPart$-module with basis $\{\fCoideal^f \fNewCoideal^y \mid f,y\in \N\}$.
		\end{proof}

		\begin{definition}
			Let $\bweight,\zeta \in \groundring$ and $\kappa_\half,\kappa_1\in \Z$. Then the \emph{Verma module with highest weight} $(\kappa_\half,\kappa_1,\bweight,\zeta )$ is the $\coideal$-module 
			\[
				\verma(\kappa_\half,\kappa_1,\bweight,\zeta )\coloneqq \coideal\otimes_{\plusCoideal\CartanPart} \groundring_{(\kappa_\half,\kappa_1,\bweight,\zeta)},
			\]
			where $\groundring_{(\kappa_\half,\kappa_1,\bweight,\zeta)}$ is the one dimensional $\plusCoideal\CartanPart$-module with $\plusCoideal\cdot 1=0$ and where $\dCoideal_\half, \dCoideal_1, \bo,Z$ act by multiplication with $q^\kappa_\half,q^\kappa_1, \bweight,\zeta$ respectively. We denote  $\kappa=\kappa_\half-\kappa_1$.
			We refer to any nonzero scalar multiple of $1\otimes  1\in \verma(\kappa_\half,\kappa_1,\bweight,\zeta )$ as a \emph{highest weight vector} of the Verma module. 
		\end{definition}
\begin{rem}By definition, a highest weight vector of $\verma(\kappa_\half,\kappa_1,\bweight,\zeta )$ is maximal of weight $(\kappa_\half,\kappa_1,\bweight,\zeta )$ in the sense of \Cref{Defweights}. Note the compatibility  with \Cref{notationweights}.
\end{rem}		
\cref{PBW basis} directly implies the following basis theorem.
		\begin{prop}\label{Verma module PBW basis}
			The Verma module $\verma(\kappa_\half,\kappa_1,\bweight,\zeta )$ has a basis given by the vectors \[
				\{ \fCoideal^f \fNewCoideal^y v \mid f,y\in \N\},\quad\text{where $v$ is a highest weight vector.}
			\]
			
		\end{prop}
		The definitions and the universal property of tensor products imply the following: 
		\begin{lem}[Universal property]\label{universal property of verma}
			Let $M$ be a $\coideal$-module with maximal vector $v$ of weight $(\kappa_\half,\kappa_1,\bweight,\zeta )$. Then there exists a unique morphism $\verma(\kappa_\half,\kappa_1,\bweight,\zeta ) \to M$  of~ $\coideal$-modules sending $1\otimes 1$ to $v$.
		\end{lem}
		From the universal property it is clear that two Verma modules $\verma(\kappa_\half,\kappa_1,\bweight,\zeta )$ and $\verma(\kappa_\half',\kappa_1',\bweight',\zeta' )$ are isomorphic if and only if $\kappa_\half=\kappa_\half'$, $\kappa_1=\kappa_1'$, $\bweight=\bweight'$ and $\zeta = \zeta'$. 
		\begin{prop}\label{unique irreducible quotient}
			Any Verma module has a unique irreducible quotient. 
		\end{prop}
		\begin{proof}
		The basis in \cref{Verma module PBW basis} is a $\dCoideal_\half$-eigenbasis and any highest weight vector must have the maximal $\dCoideal_\half$-weight.  By \Cref{easycomm}, a  Verma module  has a unique maximal proper submodule. 		
 \end{proof}
		\begin{remark}\label{Problem}
			Note that the proof of \cref{unique irreducible quotient} parallels the usual Verma module approach, utilizing the fact that the basis from \cref{Verma module PBW basis} is a $\dCoideal_\half$-weight basis, i.e. an eigenbasis for the action $\dCoideal_\half$. However, this basis is not a weight basis for the \emph{entire} Cartan subalgebra $\csaZW$ (in fact it is not possible to extend the adjoint action of $\dCoideal_\half,\dCoideal_1$ to the entire Cartan subalgebra).
		\end{remark}

		\subsection{Magical operators}
		We address  the issue from \cref{Problem}  and define adaptations of the generators of $\coideal$ which allow to construct weight space decompositions for good  Verma modules.
		
		\begin{notation}\label{quantum mu integer notation}
			For $\mu \in \groundring^*$  denote $[\mu;n] = (q-q\inv)\inv(q^n\mu-q^{-n}\mu\inv )$. Note that $[\mu q^i;0] = [\mu;i] $. In particular, $[q^i;n] = [i+n]$ is the ordinary quantum integer. We have $[\rootOfMinusOne;n]=[\rootOfMinusOne;-n]$ for all $n\in \Z$.
		\end{notation}

			\begin{definition}[Magical operators]\label{catharina operators}
				For arbitrary $\eta\in \groundring^*$ let \begin{align*}
					\Epm_\pm(\eta) \coloneqq \eCoideal \pm \eta^{\pm1} \eNewCoideal, \quad
					\Fpm_\pm(\eta) \coloneqq \fCoideal \mp \eta^{\mp1} \fNewCoideal. 
				\end{align*}
			\end{definition}
				The magical operators have surprisingly nice properties which we show next.				
				\begin{prop}\label{easycomm2}
				The following commutation relations hold in $\coideal$.
				\begin{equation*}
					F_\pm(\eta) \kCoideal = q^2 \kCoideal F_\pm(\eta)
					,\quad 
					E_\pm(\eta) \kCoideal = q^{-2} \kCoideal E_\pm(\eta)
					.
				\end{equation*}
			\end{prop}
			\begin{proof}
				This holds immediately by \cref{csa and [XY]} and the fact that $\kCoideal$ commutes with $\bo$.
			\end{proof}
			
			\begin{prop}\label{E+-F+ B0 weight}
				Let $M$ be a $\coideal$-module. If $\eta \in \groundring^*$ and $v\in M(\bo\mid [\eta;0])$,  then\begin{equation*}
					\Epm_\pm(\eta)  v \in  M(\bo\mid [\eta;\pm1])
					,\quad
					\Fpm_\pm(\eta) v  \in M(\bo\mid  [\eta;\mp1])
					. 
				\end{equation*}
			\end{prop}
			\begin{proof}
				We calculate directly
				\begin{align*}
					\bo \Epm_\pm(\eta) v &= \bo \eCoideal v \pm \eta^{\pm1} \bo \eNewCoideal v 
					=  (q\inv \eCoideal \bo +\eNewCoideal) v \pm \eta^{\pm1} (q \eNewCoideal\bo + \eCoideal) v
					\\
					&= (q\inv \bweight \pm \eta^{\pm1} ) \eCoideal v
					+ (1\pm q\eta^{\pm1}\bweight)\eNewCoideal v
					=  [\eta;\pm1 ] \Epm_\pm(\eta)  v. 
						\qedhere
				\end{align*}
			\end{proof}
			\begin{lemma}\label{E+-F+-}
				The following commutation relations hold in $\coideal$.
				   \begin{gather}
				   \begin{aligned}\label{EFcommute}
					   E_+(q\inv\eta  )E_-(\eta) =  E_- ( q\eta )E_+(\eta) 
					   , 
					   \quad  &
					   E_+( q\eta )F_-(\eta) =  F_- ( q\eta )E_+(\eta)
					   ,  \\
					   F_+( q\inv \eta )E_-(\eta) =  E_- ( q\eta )F_+(\eta) 
					   ,
					   \quad& F_+( q\eta )F_-(\eta) =  F_- ( q\inv\eta  )F_+(\eta) 
					   .
				   \end{aligned}
				   \end{gather}
		   \end{lemma}
		   \begin{proof}Using \cref{serre relations as commutators} we have
			   \begin{eqnarray*}
					E_+(q\inv\eta  )E_-(\eta)
				   & = &  (\eCoideal + q^{-1}\eta \eNewCoideal) (\eCoideal- \eta \inv \eNewCoideal ) 
				   = \eCoideal^2  - \eta \inv  \eCoideal \eNewCoideal  + q^{-1}\eta  \eNewCoideal \eCoideal  - q^{-1} \eNewCoideal^2  \\
				   &=& \eCoideal^2  - q^{-1}\eta \inv   \eNewCoideal \eCoideal  + \eta \eCoideal \eNewCoideal  - q ^{-1} \eNewCoideal^2  
				   =(\eCoideal - q^{-1} \eta \inv \eNewCoideal)(\eCoideal + \eta \eNewCoideal) \\
				   &=&E_-(q \eta )E_+(\eta ). 
			   \end{eqnarray*}
			   Using the involution from \cref{involution}, we get $F_+(q\eta )F_-(\eta ) =  F_- (q\inv \eta )F_+(\eta ) $.
				Abbreviating $[E_+,F_-](\eta )=(E_+(q\eta )F_-(\eta ) - F_-(q\eta )E_+(\eta )) $, we have
			   \begin{eqnarray*}
				   [E_+,F_-](\eta ) 
				   & = &  ([\eCoideal,\fCoideal ]  - \eta [ \fCoideal,\eNewCoideal ]_{q}
				   + \eta  [\eCoideal,\fNewCoideal ]_{q} +q \eta ^2[\eNewCoideal,\fNewCoideal]) \\
				   &=& [\kCoideal;0]  (1-\eta ^2) 
				   +\eta  (q\bo [\kCoideal;0] - q^{-1}[\kCoideal;0]\bo)  
				   \\
				   &=& \eta [\kCoideal;0] (\eta \inv -\eta +(q-q^{-1})\bo) 
				   =
				   \eta [\kCoideal;0] (\eta \inv -\eta +(q-q^{-1})[\eta ;0] )  =0.
			   \end{eqnarray*}
			   In the second equality we used \cref{csa and [XY]} and $
				   [\eCoideal,\fNewCoideal]_q 
				   - [\fCoideal,\eNewCoideal]_{q} 
				   = (q- q^{-1})[\kCoideal;0]\bo.
			   $
			   
			   The calculation for $[E_-,F_+] $ is similar.  
		   \end{proof}
		   
		   \begin{remark}
				There is some freedom in the definition of the operators $E_\pm,F_\pm$.
			   Namely, the operators $E_\pm$, together with the 
			    $
				   F_+' (\eta )=  F + \eta \inv \fNewCoideal ',\, F_-'(\eta ) = F - \eta~ \fNewCoideal ',
			   $ with $\fNewCoideal '= [\bo,\fCoideal]_q$ form a family of operators which satisfy the same commutation relations as in \cref{E+-F+-}. We  could use  them instead. Alternatively, we could also fix $F_\pm$ and modify $E_\pm$ analogously.
		   \end{remark}
			\newcommand{\idempotent}{\mathbbm{1}}
			\begin{notation}
				Let $\idempotent_\eta$ be the projection to the weight space $V(\bo\mid [\eta;0])$.
				We define operators \begin{align}
				\Epm_\pm\coloneqq \sum_\eta \Epm_\pm(\eta)\idempotent_{[\eta;0]}, \quad
				\Fpm_\pm\coloneqq \sum_\eta \Fpm_\pm(\eta)\idempotent_{[\eta;0]}.
				\end{align}
Note that $E_+^i w = 
				E_+( q^{i-1}\eta) \cdots E_+(q\eta) E_+(\eta) w,
			$ for $w\in V(\bo\mid [\eta;0])$ and  similarly for $E_-,F_\pm$. 
			\end{notation}
			\begin{lemma}\label{E+F+ commutator}
				The following commutator relations hold:
				\begin{gather}	
				\begin{aligned}[E_+ , F_+ ]\idempotent_\eta
					&= \left( (1+q^{-2}) [\kCoideal ;0] - \eta\inv Z- \eta W\right) \idempotent_\eta
					\label{E+F+ any verma}
					,\\
					[E_-,F_-]\idempotent_\eta
					&= \left( (1+q^{-2}) [\kCoideal ;0] + \eta Z + \eta\inv W\right) \idempotent_\eta
					.
				\end{aligned}
				\end{gather}
			\end{lemma}
			\begin{proof}We calculate:
				\begin{eqnarray*}
					[E_+,F_+] \idempotent_\eta &=& (E_+ (\eta q\inv ) F_+(\eta ) - F_+(\eta q) E_+(\eta)	) \idempotent_\eta 
					\\
					&=&( [{\eCoideal},F] - \eta\inv [{\eCoideal},\fNewCoideal ]_{q\inv} -\eta [F,\eNewCoideal]_{q\inv } - q\inv [\eNewCoideal,\fNewCoideal ])\idempotent_\eta
					.
				\end{eqnarray*}
				Using the last equation in \cref{csa and [XY]} and the involution in \cref{involution} we are done.
			\end{proof}

			\newcommand{\newAlpha}{\tilde{\alpha}}
			\begin{lem}\label{ZW F+- commutator}
				Let $\eta\in \C(q)$ and let $v$ be a simultaneous eigenvector of $\bo,Z$ with $\bo$-weight $[\eta;0]$. Then the following holds.
				\begin{gather}	\label{Z commutator}
				\begin{aligned}
					[Z,\Fpm_+]_{q}	v  =   \eta\inv q [2]  \Fpm_+ \kCoideal \inv  v,
					\quad 
					[Z,\Fpm_-]_{q} v  = -  \eta  q[2]  \Fpm_- \kCoideal \inv  v,
					\\ 
					[Z,\Epm_+]_{q^{-1}} v = - \eta\inv q\inv[2]  \kCoideal\inv  \Epm_+v,
					\quad
					[Z,\Epm_-]_{q^{-1}} v =  \eta q\inv [2] \kCoideal\inv  \Epm_-v.
				\end{aligned}
				\end{gather}
				Let $v$ be a common eigenvector of $\bo,W$ of $\bo$-weight $[\eta;0]$, then
				\begin{gather}
				\begin{aligned}
					\label{W commutator}
					[W ,\Fpm_+]_{q\inv }  v=   \eta q\inv [2]   \kCoideal   \Fpm_+v,
					\quad 
					[W ,\Fpm_-]_{q\inv }  v=  -  \eta\inv q\inv[2]  \kCoideal   \Fpm_- v,
					\\ 
					[W,\Epm_+]_{q} v = - \eta  q[2]   \Epm_+  \kCoideal v,
					\quad
					[W,\Epm_-]_{q} v =  \eta\inv q [2]   \Epm_- \kCoideal v.
				\end{aligned}
				\end{gather}
			\end{lem}
			\begin{proof} 
				Let $v$ be a common eigenvector of $\bo,Z$ of $\bo$-weight $[\eta;0]$. We calculate				\begin{eqnarray*}
					Z\Fpm_+ v &=& Z (\fCoideal-\eta\inv \fNewCoideal)v 
					\\
					&=& \left(q \fCoideal Z -q\inv [2]\kCoideal\inv \fNewCoideal - q\eta\inv \fNewCoideal Z+q\eta\inv [2]   \fCoideal \kCoideal\inv + \eta\inv (q^3-q\inv)\fNewCoideal  \bo\kCoideal\inv \right) v
					\\
					&=& \left( q \Fpm_+Zv +q\eta\inv [2]   \fCoideal \kCoideal\inv-  q[2]\fNewCoideal\kCoideal\inv +\eta\inv (q^3-q\inv)[\eta;0]\fNewCoideal \kCoideal\inv\right) v
					\\
					&=& \left( q \Fpm_+Zv +q\eta\inv [2]   \fCoideal \kCoideal\inv-  q[2]\fNewCoideal\kCoideal\inv +q[2](1-\eta^{-2})\fNewCoideal \kCoideal\inv\right) v
					\\
					&=& \left( q \Fpm_+Zv +\eta\inv q[2]\Fpm_+ \kCoideal \inv\right)  v,
					\end{eqnarray*}
				using \Cref{MPIrelations}, \Cref{easycomm} and the fact 
 that $(q^3-q\inv)[\eta;0]=(q^3-q\inv)(q-q^{-1})\inv(\eta-\eta\inv)=(q^2+1)(\eta-\eta\inv)$ and that $Zv$ and  $\kCoideal \inv v$ are scalar multiples of $v$.
				
			\iffalse
				and the left hand side has $\bo$-weight $[\eta;-1]$, on the other side so does $\Fpm_+Zv$ because by assumption $Zv$ is a scalar multiple of $v$. so we simply ignore the $\Fpm_-v$ contribution which has $\bo$-weight $[\eta;1]\ne [\eta;-1]$ since $\eta^2\ne -1$.
				\fi
				Similarly, let $v$ be a simultaneous eigenvector of $\bo,W$ with $\bo$-weight $[\eta;0]$ and calculate
				\begin{eqnarray*}
					W \Fpm_+ v &=& W (\fCoideal-\eta\inv \fNewCoideal )v
					\\
					&=&\left(q\inv \fCoideal W -q\inv[2]\kCoideal \fNewCoideal  + q\inv(1-{q^{-4}}) \fCoideal \bo\kCoideal
					- \eta\inv q^{-1}(\fNewCoideal W- [2] \kCoideal  \fCoideal ) \right) v
					\\
					& = & q\inv \Fpm_+ W v +\left(\eta\inv q^{-1}[2]-q(1-{q^{-4}})(q-q\inv)\inv(\eta-\eta\inv)\right) \kCoideal\fCoideal )
					\\
					& = & q\inv \Fpm_+ W v+\eta q\inv[2](\kCoideal\fCoideal-\eta\inv \kCoideal\fNewCoideal )v
					\\
					& = & q\inv \Fpm_+ W v+\eta q\inv[2]\kCoideal\Fpm_+v,
					% \newAlpha_\eta\inv (q[2] + q^{-1}[2]+q^{-2}(q^4-1) [\eta;-1]) \kCoideal \Fpm_+ v
					%\\
					%& = & q\inv \Fpm_+ W v + \newAlpha_\eta\inv q\inv [2]  (1  + \eta^{2}) \kCoideal \Fpm_+ v
				\end{eqnarray*}
				noting that $\eta\inv q^{-1}[2]-q(1-{q^{-4}})(q-q\inv)\inv(\eta-\eta\inv)=\eta\inv(1-q^{-2})+(1-q^{-2})(\eta-\eta\inv)=\eta q\inv[2]$.
				
				We verified  the formulas involving $\Fpm_+$. The calculations for $\Fpm_-$ are analogous. By applying the involution in \cref{involution} we get also the desired result for $\Epm_\pm$.
			\end{proof}
			
			The magical operators interact well with eigenspaces for the Cartan subalgebra: 
			\begin{thm}[Magical operators]\label{E+- F+- act on weight space any verma} 
				Let $M$ be a  $\coideal$-module. Then we have 
				\begin{gather}
				 \begin{aligned}
					F_\pm M({\kappa_\half},{\kappa_1},[\eta;0],\zeta) 
						&\; \subseteq  \;
					M({\kappa_\half-1},{\kappa_1 +1}, [\eta; \mp  1], q(\zeta\pm \eta^{\mp1} q^\kappa[2]))
					\label{F+- act on weight space any verma}%
					,
					\\
					E_\pm M({\kappa_\half},{\kappa_1 }, [\eta;0],\zeta) 
						&\; \subseteq \;
					M({\kappa_\half+1},{\kappa_1-1},[\eta;\pm1],q\inv(\zeta \mp \eta^{\mp1} q^{-2-\kappa}[2]))
					%\label{E+- act on weight space any verma}%
					.
				\end{aligned}
				\end{gather}
				If $v'\in M(\kCoideal,\bo,Z,W\mid q^\kappa,[\eta;0],\zeta,\omega)$ then, using \Cref{quantum mu integer notation}, we have  for any $a,b\in\N$,
				\begin{gather}\label{ZW eigenval}
				 \begin{aligned}
					F_+^a F_-^b v' &\in M(Z\mid q^{a+b}(\eta\inv q^{a-\kappa-1}[2][a] -\eta q^{b-\kappa-1} [2][b] + \zeta)),
					\\
					F_+^a F_-^b v' &\in M(W\mid q^{-a-b}(\eta q^{\kappa-a-1} [2][a] -\eta\inv q^{\kappa-b-1}  [2][b] + \omega)).
				\end{aligned}
				\end{gather}
			\end{thm}
			\begin{proof} 
				The claims \cref{F+- act on weight space any verma} follow from \cref{Z commutator}.
				Let $v'\in M(\kCoideal,\bo,Z,W\mid\kappa,[\eta;0],\zeta,\omega)$. By \cref{ZW F+- commutator}, 
				\begin{eqnarray}
					Z F_+^a v' = (\eta\inv q^{-\kappa+2a-1}[2][a] + q^{ a } \zeta) F_+^{a}  v'	,
					\quad 
					Z F_-^b v' = (-\eta q^{-\kappa+2b-1}[2][b] + q^{b} \zeta) F_-^{b}  v'	,
					\label{Z F^a+- any Verma}
					\\
					W F_+^a v' =(\eta q^{\kappa-2a-1}[2][a] + q^{ -a } \omega) F_+^{a}  v',
					%n-2*i+kappa-1
					\quad 
					W F_-^b v' = (-\eta\inv q^{\kappa-2b-1}[2][b] + q^{-b} \omega) F_-^{b}  v'.
					%-n-2*i-1+kappa)
					\label{W F^a+- any Verma}
				\end{eqnarray}
				Combining the formulas in \eqref{Z F^a+- any Verma} we get
				\begin{eqnarray*}
					Z F_+^a F_-^b v' &=& (\eta\inv q^{-b-\kappa+2b+2a-1}[2][a] + q^{ a } (-\eta q^{-\kappa+2b-1}[2][b] + q^{b} \zeta)  ) F_+^{a}  F_-^b v' 
					\\
					&=& (\eta\inv q^{-\kappa+b+2a-1}[2][a] -\eta q^{-\kappa+a+2b-1}[2][b] + q^{a+b} \zeta ) F_+^{a}  F_-^b v'
					,
				\end{eqnarray*}
				  The $W$-weight of $F_+^a F_-^b v'$ can be calculated similarly using \eqref{W F^a+- any Verma}, and thus \eqref{ZW eigenval} holds. 
			\end{proof}
\subsection{Action on Verma modules} We next consider Verma modules of the form $\verma(\kappa_\half,\kappa_1,[\mu;0],\zeta)$.  For any $x\in \groundring$, there exists $\mu$ in some algebraic extension of $\groundring$  such that $[\mu;0]=x$. 
			\begin{lemma}\label{Banff}
				Let $v$ be a highest weight vector of a Verma module $\verma(\kappa_\half,\kappa_1,[\mu;0],\zeta)$. Then $v$ is an eigenvector of $W$ with eigenvalue $\omega=q^{-2}(\zeta - (q^\kappa-q^{-\kappa})[\mu;0] )$.
			\end{lemma}
			\begin{proof}
				This follows directly from the relation \cref{Z in W}.
			\end{proof}
			
			\begin{prop}[Action on weight vectors]\label{E+- scalars any verma}
				Let $\mu,\zeta \in \groundring$ and $\kappa_\half,\kappa_1\in \Z$. Consider a Verma module with highest weight $(q^{\kappa_\half},q^{\kappa_1},[\mu;0],\zeta)$. Let $v$ be a highest weight vector and  $\omega$ its  $W$-weight. 
				
				With \Cref{quantum mu integer notation},  the following holds then for any  $a,b\in \N$, 
				\begin{gather}\label{E+ F+aF-bv any verma}
				\begin{aligned}
					&\Epm_+ \Fpm_+^a \Fpm_- ^b v = (q\inv(\Gamma_{a}^+-M_{a}^+) - q^{a-1}\mu\inv \zeta -\mu q^{1-a}\omega )[a] \Fpm_+^{a-1} \Fpm_-^{b} v
					,
					\\ 
					&\Epm_- \Fpm_+^{a} \Fpm_-^{b} v = (q\inv(\Gamma_{b}^- -M_{b}^-)+  q^{1-b}\mu\inv\omega +\mu q^{b-1}\zeta )[b] \Fpm_+^{a} \Fpm_-^{b-1} v,
					%\label{E- F+aF-bv any verma}
										%\label{Gamma+ val}
					%	\iffalse
					%&\text{where} \quad \Gamma_{a}^+ =  [2][\kappa-a+1]-[a-1]\newAlpha_{\mu^{2}_{(-a)} q^{1+\kappa}} \label{Gamma+ val},\\
					%&\Gamma_{b}^- = [2][\kappa-b+1]- [b-1]\newAlpha_{\mu^2_{(b)}q^{-\kappa-1}} \label{Gamma- val},\fi
				\end{aligned}
				\end{gather}
				where $\Gamma_{c}^\pm =  [2][\kappa-c+1]$, $M_c^\pm=[c-1]\newAlpha_{\mu^{2} q^{\pm(1+\kappa-2c)}}$ with $\newAlpha_\mu =\mu+\mu\inv$.
			\end{prop}
			\begin{proof}We only show the first formula, since the second works analogously. We argue by induction on $a$ respectively.  For $a=0$  the formula holds \Cref{E+-F+-}.  Assuming the formula holds for $a\geq0$ we  show it for $a+1$. 
	By \cref{E+F+ any verma}, with $\eta=\mu q^{b-a}$ by \eqref{F+- act on weight space any verma}, and the eigenvalue formulas  \cref{ZW eigenval}, we have 
				\begin{eqnarray*}
				\Epm_+ \Fpm_+^{a+1} \Fpm_- ^b v &=&(\Epm_+\Fpm_+) \Fpm_+^{a} \Fpm_-^{b} v = (\Fpm_+\Epm_+ 
					+ (1+q^{-2}) [\kCoideal ;0] 
					- \mu\inv q^{a-b} Z
					- \mu q^{b-a} W
					)\Fpm_+^{a} \Fpm_-^{b} v
					\\
					&= &\left(( q\inv(\Gamma_{a}^+-M_a^+)- q^{a-1}\mu\inv \zeta -\mu q^{1-a}\omega )[a]
					+q\inv[2] [\kappa-2a-2b]\right)\Fpm_+^{a} \Fpm_-^{b} v\\
					&&-\mu\inv q^{a-b} q^{a+b}\left(\zeta +\mu\inv q^{a-\kappa-1}[2][a]-\mu q^{b-\kappa-1}[2][b]\right)\Fpm_+^{a} \Fpm_-^{b} v\\
					&&-\mu q^{b-a} q^{-a-b}\left(\omega +\mu q^{\kappa-a-1}[2][a]-\mu\inv q^{\kappa-b-1}[2][b]\right)\Fpm_+^{a} \Fpm_-^{b} v.
	\end{eqnarray*}	
Collecting the terms involving $\zeta$, and similarly for $\omega$, we get the desired values, namely 
\[-(q^{a-1}[a]+q^{2a}\mu\inv)\zeta=-(q^{2a-1}-q\inv+q^{2a+1}-q^{2a-1})\mu\inv(q-q\inv)\inv\zeta=-q^a\mu\inv[a+1]\zeta.\]	
For terms with  $\mu^{\pm2}$  we want $-q\inv M_{a+1}^+=-q\inv(M_a^+[a]+\mu^{-2} q^{2a}q^{a-\kappa}[2][a]+\mu^2 q^{-2a}q^{\kappa-a}[2][a])$, i.e. 
\[([a-1](q^{1+\kappa-2a}\mu^2+q^{2a-1-\kappa}\mu^{-2}) +\mu^{-2} q^{3a-\kappa}[2]+\mu^2 q^{\kappa-3a}[2])=
(q^{-1+\kappa-2a}\mu^2+q^{2a+1-\kappa}\mu^{-2})[a+1].\]
This  is equivalent to $q^{\pm(\kappa-2a+1)}[a-1]+q^{\pm(\kappa-3a)}[2]=q^{\pm(\kappa-2a-1)}[a+1]$ which is easy to check. 
The remaining pieces should be $q\inv\Gamma^+_{a+1}[a+1]$. Equivalently,  $\Gamma^+_{a+1}[a+1][2]\inv=[\kappa-a][a+1]$ should be
\[ [\kappa-a+1][a]+[\kappa-2a-2b]
+q^{2a+b-\kappa}[b]+q^{\kappa-2a-b}[b]= [\kappa-a+1][a]+[\kappa-2a], 
\]
which is again easy to check. Thus, the desired formula holds for $a+1$. This finishes the proof.
\end{proof}
		\subsection{Good Verma modules}
			We now study Verma modules using the operators $\Epm_\pm$ and $\Fpm_\pm$. Depending on their action,  we divide the Verma modules into good and exceptional ones.

			The nicest values for $\mu$ are $q^n$ for $n\in  \Z$, in which case we will see in \cref{rational representations} that the $\Fpm_-$ ``decreases'' $\bo$-weights while $\Fpm_+$ ``increases'' them. This is not true in general. Indeed, if $[\rootOfMinusOne;0]$ is a $\bo$-weight using \cref{quantum mu integer notation}, then on this $\bo$-weight space we have $\Epm_+(\rootOfMinusOne) = \Epm_-(\rootOfMinusOne) $ and $\Fpm_+(\rootOfMinusOne) = \Fpm_-(\rootOfMinusOne) $, and we will see in \Cref{not diagonalizable} that this drastically changes the behaviour of the Verma module. 
		
			\begin{definition}
				Let $\mu,\zeta \in \groundring$ and $\kappa_\half,\kappa_1\in \Z$. A Verma module $\verma(\kappa_\half,\kappa_1,[\mu;0],\zeta)$ is called a \emph{good Verma module} if $\mu\ne \pm q^l \rootOfMinusOne$ for any $l\in \Z$.
			\end{definition}
			\begin{thm}[Weight basis for good Verma modules]\label{weight basis good verma}
			 Let  $\verma=\verma(\kappa_\half,\kappa_1,[\mu;0],\zeta)$ be a good Verma module with  highest weight vector $v$. Then 
                                 \begin{enumerate}
                                 \item  the vectors 
					$\{ \Fpm_+^a \Fpm_-^b v \mid a,b\in \N \}$ form a basis of $\verma$ consisting of weight vectors, and 
				\item these vectors have distinct weights.  In particular all weight spaces are one dimensional.
				\end{enumerate}
			\end{thm}
			\begin{proof}
				By \Cref{Verma module PBW basis} and \cref{E+- F+- act on weight space any verma}, $\Fpm_+^a \Fpm_-^b v$ is a nonzero weight vector for any $a,b\in \N$.
				The $(\kCoideal, \bo)$-weight of $\Fpm_+^a \Fpm_-^b v$ is $(q^{\kappa_\half-a-b},q^{\kappa_1+a+b},[\mu; b-a])$. 
				
			      Suppose the weights of $\Fpm_+^a \Fpm_-^b v$ and $\Fpm_+^{a'} \Fpm_-^{b'} v$ agree for some $(a,b),(a',b')\in \N^2$. Then $q^{\kappa_\half-a-b} = q^{\kappa_\half-a'-b'}$ and $q^{\kappa_1+a+b} = q^{\kappa_1+a'+b'}$ imply $a+b=a'+b'$ and $[\mu;b-a]=[\mu;b'-a']$ and thus $\mu = \pm \sqrt{-q^{n}}$ where $n=a-b+a'-b'$. Moreover, $n$ has to be odd since the Verma module is good. By \cref{ZW eigenval} and $b+a=a'+b'$, the $Z$-eigenvalues of $\Fpm_+^a \Fpm_-^b v$ and $\Fpm_+^{a'} \Fpm_-^{b'} v$ are equal if and only if
				\begin{equation*}
					-q^{-n/2-\kappa-1+a}[a] + q^{n/2-\kappa-1+b} [b]
					=
					- q^{-n/2-\kappa-1+a'}[a'] + q^{n/2-\kappa-1+b'}[b'].
				\end{equation*} 
				Evaluation at $q=1$ implies $b-a=b'-a'$. Thus $(a,b)=(a',b')$ holds, as  $a+b=a'+b'$.  Hence, under our assumption on $\mu$, the vectors $\Fpm_+^a \Fpm_-^b v$ have distinct weights and are  therefore  linearly independent. It remains  to  show that good Verma modules are spanned by these weight vectors. For this let $M\subset \verma(\kappa_\half,\kappa_1,[\mu;0],\zeta)$ be the submodule generated by $\Fpm_+^a \Fpm_-^b v$.  By \cref{Verma module PBW basis} it suffices to show that $M$ contains all $ \fCoideal^f \fNewCoideal^y v$. We argue  by induction on $f+y$, the case $f=y=0$ being clear. For $f+y>0$ we may assume $f>0$,  since $\fCoideal\fNewCoideal = q \inv\fNewCoideal\fCoideal$. By induction hypothesis, \[
				\fCoideal^{f-1} \fNewCoideal^y v  = \sum_{a,b} c_{a,b} \Fpm_+^a \Fpm_-^b v,
				\]
				The $\bo$-weight of $\Fpm_+^a \Fpm_-^b v$ is $[\mu; b-a]$. By the definition of $\Fpm_\pm$, 
				for $\eta\ne \rootOfMinusOne$ we have 
				\begin{gather}
				\begin{aligned}	\label{EF in E+- F+-}
					\eCoideal \idempotent_{[\eta;0]} = \frac{ \eta\inv \Epm_+  + \eta \Epm_- }{\eta+\eta\inv}\idempotent_{[\eta;0]}
					&,\;\;
					\fCoideal \idempotent_{[\eta;0]} = \frac{\eta \Fpm_+  + \eta\inv \Fpm_- }{\eta+\eta\inv}\idempotent_{[\eta;0]}
					,
					\\
					\eNewCoideal \idempotent_{[\eta;0]} = \frac{\Epm_+  - \Epm_- }{\eta+\eta\inv}\idempotent_{[\eta;0]}
					&,\;\;  
					\fNewCoideal \idempotent_{[\eta;0]} =  \frac{\Fpm_--\Fpm_+ }{\eta+\eta\inv}\idempotent_{[\eta;0]}.
				\end{aligned}
				\end{gather}
				Since we assumed $\mu$ is not in $q^\Z\rootOfMinusOne$, the same holds for all $\mu q^{b-a} $ thanks to \Cref{E+- scalars any verma}.   Then 
				\[
				 \fCoideal^f \fNewCoideal^y v = \fCoideal\fCoideal^{y-1} \fNewCoideal^f v= \sum_{a,b} c_{a,b}\fCoideal \Fpm_+^a \Fpm_-^b v
				 = \sum_{a,b} c_{a,b} \frac{\mu q^{b-a} \Fpm_+ + \mu\inv q^{a-b}\Fpm_- }{\mu q^{b-a}+\mu \inv q^{a-b}} \Fpm_+^a \Fpm_-^b v
				 .
				\]
				This shows that any vector of the form $\fCoideal^f \fNewCoideal^y v$ can be written as a linear combination of vectors of the form $\Fpm_+^{a} \Fpm_-^{b} v$ and  the claim holds.  Thus, these weight vectors form a basis as desired.
			\end{proof}

			\begin{thm}[Homs between good Vermas]\label{Homs between vermas or hw vecs in vermas}
				Let $\mu,\zeta,\mu',\zeta' \in \groundring$ and $\kappa_\half,\kappa_1,\kappa_\half',\kappa_1'\in \Z$ such that $M = \verma(\kappa_\half,\kappa_1,[\mu;0],\zeta)$ is a good Verma module. Then  $\Hom_{\coideal}(\verma(\kappa_\half',\kappa_1',[\mu';0],\zeta'),M)$ is nonzero, and then equal to $\groundring$, iff there exists $i\in\N$ such that  one of the following  hold:
				 \[1)\quad
					\begin{cases}
						\zeta =  \mu q^{-i}[\kappa-i]-\mu\inv q^{i-\kappa}[i] , \\
						\kappa_\half' = \kappa_\half - i-1, \\ 
						\kappa_1' = \kappa_1 + i+1, \\
						[\mu';0] = [\mu;-(i+1)], \\
						\zeta' =\mu q [\kappa-i] + \mu\inv q^{2i+1-\kappa}[i+2] ;
					\end{cases}
					\quad 2)\quad
					\begin{cases}
						\zeta =  \mu q^{i-\kappa}[i]-\mu\inv q^{-i}[\kappa-i] , \\
						\kappa_\half' = \kappa_\half - i-1, \\ 
						\kappa_1' = \kappa_1 + i+1, \\
						[\mu';0] = [\mu;i+1], \\
						\zeta' = - \mu q^{2i+1-\kappa}[i+2] - \mu\inv q[\kappa-i];
					\end{cases}
				\]\[
					3)\quad\begin{cases}
						0\leq i\leq \kappa, \\
						\zeta = \mu q^{-i}[\kappa-i]-\mu\inv q^{i-\kappa}[i] , \\
						\kappa_\half' = \kappa_1 + 2, \\ 
						\kappa_1' = \kappa_\half - 2, \\
						[\mu';0] = [\mu;\kappa-2i], \\
						\zeta' = \mu\inv q^{i-\kappa} [i+2]-\mu q^{-i}[\kappa-i+2] ;
					\end{cases}
					\quad
				         4)\quad
					(\kappa_\half',\kappa_1',[\mu';0],\zeta')=(\kappa_\half,\kappa_1,[\mu;0],\zeta). \]
			\end{thm}
			\begin{proof}
			Pick a highest weight vector $v$ of $M$. Since it has maximal $\dCoideal_\half$-weight, the endomorphism ring of $M$ is one dimensional. Suppose we are not in the fourth case, thus $M\ne \verma(\kappa_\half',\kappa_1',[\mu';0],\zeta')$. By \cref{universal property of verma} there is a nonzero homomorphism from $\verma(\kappa_\half',\kappa_1',[\mu';0],\zeta')$ to $M$ if and only if there exists a maximal vector $w$ in $M$ of weight $(\kappa_\half',\kappa_1',[\mu';0],\zeta')$. Let $v\in M$ be a highest weight vector. Since $M$ is good, $M$ decomposes by \cref{weight basis good verma} into one dimensional weight spaces spanned by $\Fpm_+^a \Fpm_-^b v$ for $a,b\in \N$. Hence we may assume that $w=\Fpm_+^a \Fpm_-^b v$ for some $a,b\in \N$. Using  again that $M$ is a good Verma module and \cref{EF in E+- F+-}, we see that the conditions  $\eCoideal w = \eNewCoideal w=0 $ are equivalent to $\Epm_+ w = \Epm_- w=0$ which  are in turn equivalent to 
			 \begin{align*}
					R_a &:= q\inv\Gamma_{a}^+ - \mu\inv q^{a-1}\zeta -\mu q^{1-a}\omega =0,
					&S_b &:= q\inv\Gamma_{b}^- + \mu\inv q^{1-b}\omega +\mu q^{b-1}\zeta =0,
				\end{align*}
				where $\omega$ denotes the $W$-weight of $v$. 
				The terms $R_a$ and $S_b$ 
				are independent of  $b $ and $a$ respectively. Since $\Fpm_+^a \Fpm^b_-v$ is a maximal vector, $\Epm_+ \Fpm_+^a v =0,  \Epm_- \Fpm_-^b v =0$.
				By \cref{E+-F+-}, we then also have automatically $\Epm_-\Fpm_+^a v = \Epm_+\Fpm_-^b v=0$. Thus, both $\Fpm_+^a v $ and $ \Fpm_-^b v$, are maximal vectors in $M$. \hfill\\
				Now $R_a$ and $S_a$ are linear expressions in $\zeta$. Solving $R_a=0$ and  $S_b=0$ we get 
				\begin{align*}
				%	\label{zeta val F+av}
					\zeta &=  [\mu;\kappa-2a+2] - q^{-\kappa} [\mu;0]
					 = \mu q^{-i}[\kappa-i]-\mu\inv q^{i-\kappa}[i] \quad &&\text{for $R_a=0$, and}\quad\\
					  \zeta &= [\mu;2b-2-\kappa] - q^{-\kappa} [\mu;0] 
					= \mu q^{i-\kappa}[i]-\mu\inv q^{-i}[\kappa-i]\quad &&\text{for $S_b=0$}. 
				\end{align*}
				The conditions mean precisely that $ \Fpm_+^a v$ respectively $\Fpm_-^b v$ is a maximal vector.
				Setting $i=a-1$ or $i=b-1$ we get  the conditions from case 1) or 2) respectively. If both, $R_a=0$ and $S_b=0$, hold, then the substitution $i\mapsto \kappa - i$ relates the two solutions and we must have $0\leq i\leq \kappa$. 
				By computing the weight of the maximal vector $\Fpm_+^a \Fpm_-^b v$ we get case 3).
			\end{proof}
			
			\begin{remark}\label{hom into subvermas}
Verma modules are free over $\minusCoideal$ and thus torsion free. Therefore, nonzero homomorphisms between Verma modules are injective. \cref{Homs between vermas or hw vecs in vermas}  describes the inclusions.  Writing in the first case $\kappa' = \kappa-2i-2, \; \mu' =\mu q^{-i-1}, \; i' = \kappa-i$ gives $
					\zeta' = \mu' q^{i'-\kappa'}[i']-(\mu'q^{i'})
					\inv [\kappa'-i'] $,
				which has of the same form as the $\zeta$ in the second case. If $i'\geq 0$, then there is another Verma module mapping into $\verma({\kappa_\half'},{\kappa_1'},[\mu';0],\zeta')$. Similarly, in the second case writing $\kappa' = \kappa-2i-2, \; \mu' =\mu q^{i+1}, \; i' = \kappa-i$ gives $
					\zeta' = -\mu'q^{-i'} [i'-\kappa'] - ((\mu')\inv q^{i'-\kappa'}) [i'] $,
				which  has  the same form as $\zeta$ in the first case. Consequently there is another Verma module mapping into $\verma({\kappa_\half'},{\kappa_1'},[\mu';0],\zeta')$ if $i'\geq 0$. \hfill\\
By the proof of \Cref{Homs between vermas or hw vecs in vermas}  the morphisms map the chosen highest weight vectors to $\Fpm_+^a v $, $ \Fpm_-^b v$, $\Fpm_+^a\Fpm_-^bv$ respectively. If the dominance condition from \eqref{dominant}  below holds, then the intersection of the two larger Verma submodules is exactly the small one by \Cref{weight basis good verma}.\hfill\\
In summary,  good Verma modules have the same inclusion behaviour as  classical  Verma modules for  $\mathfrak{gl}_2\times\mathfrak{gl}_2$.  In that case every Verma module is the outer tensor product $M(\lambda)\boxtimes M(\mu)$ of two Verma modules and each factor  is either simple or has a unique proper Verma submodule. 		 
							\end{remark}
									\begin{lemma}\label{MPIformula}
				Let  $\mu\in \groundring^*$, $i,\kappa\in \Z$. Then the following holds by definition 
				\[[\mu;\kappa-2i] - q^{-\kappa} [\mu;0] = \mu q^{-i}[\kappa-i]-\mu\inv q^{i-\kappa}[i].\]
			\end{lemma}
We give a  \emph{dominance condition} \eqref{dominant}  for finite dimensional quotients of good Verma modules: 
			\begin{thm}[Good irreducibles]\label{fd quot good verma}
				Consider a good Verma module $\verma=\verma({\kappa_\half},{\kappa_1},[\mu;0],\zeta)$  with highest weight vector $v$. Then it has a finite dimensional irreducible quotient $L$ if and only if  
\begin{equation}\label{dominant}
					{\zeta = [\mu;\kappa-2i] - q^{-\kappa} [\mu;0]}\quad\text{for some $i\in \N$ with $0\le i \le \kappa$.}
\end{equation}
				In this case $L$ is unique and has a weight basis 
					$\{\Fpm_+^a \Fpm_-^b v,\mid 0\le a\leq i 
							,\, 0\le b\le \kappa-i\}.$
			\end{thm}
			\begin{proof}	
				Suppose $\zeta = [\mu;\kappa-2i] - q^{-\kappa} [\mu;0]$. By \Cref{MPIformula} we can apply case 1)-2) in \cref{Homs between vermas or hw vecs in vermas}. 
			Its proof shows that for $i_0= i+1,j_0=j+1$, the vectors $\Fpm_+^{i_0} v, \Fpm_-^{j_0} v$ are maximal. 
			They generate a submodule of $\verma$ whose quotient is spanned by \cref{weight basis good verma}
			by  finitely many basis vectors $\Fpm_+^a \Fpm_-^b v$ with $0\le a< i_0, 0\le b< j_0$.  This implies that  $\verma$ has a finite dimensional irreducible quotient. 
			
			Conversely, if $\verma$ has a finite dimensional quotient $L$ then there exists $(c,d)\not=(0,0)$ such that  $\Epm_\pm\Fpm_+^c \Fpm_-^d v=0$ or equivalently, the  scalars appearing in \eqref{E+ F+aF-bv any verma} vanish. By \cref{Banff}, this means 
\[(q^c\mu\inv+q^{-c}\mu)\zeta=\mu^{-c}(q^\kappa-q^{-\kappa})[\mu;0]+[2][\kappa-c+1]-[c-1](\mu^2q^{\kappa-2c+1}+\mu^{-2}q^{2c-\kappa-1}\]  for the first scalar. This has a unique solution in  $\zeta$ which one checks is of the form \eqref{dominant} with $i=c-1$.    The second scalar gives then $i=\kappa-d-1$.   By \cref{E+- scalars any verma}  $\Epm_+$ and $\Epm_-$ act on $\Fpm_+^a \Fpm_-^b v$ by a scalar which is independent  of $b$ and $a$ respectively.  
Inserting $\zeta$ into the scalar we get $\Epm_+ \Fpm_+^a \ne 0 $ for $a<c$ and $\Epm_- \Fpm_-^b \ne 0 $ for $b<d$. By  \cref{weight basis good verma}, the finite dimensional irreducible quotient has a weight basis of the desired form.
		\end{proof}
			\begin{remark}
				By \cref{W commutator}, each $\Fpm_+^i \Fpm_-^j v$ is also an eigenvector for $W$. By \eqref{csa} one could also  take  $W$ instead of $Z$ to define a  Cartan algebra. All our results hold analogously for this alternative choice. 	
				\end{remark}		
\Cref{Homs between vermas or hw vecs in vermas} with \Cref{fd quot good verma} implies BGG-type resolutions of good irreducible modules. 
\begin{cor}[BGG resolutions]\label{BGG resolutions}
Assume $M = \verma(\kappa_\half,\kappa_1,[\mu;0],\zeta)$ is a good Verma module with dominant highest weight \eqref{dominant}. The its irreducible finite dimensional quotient  $L$ has a resolution
\small
\begin{equation*}
\verma(\kappa_1+2,\kappa_\half-2,[\mu;\kappa-2i],\zeta_{3,i}') 
\hookrightarrow
\begin{array}{c}
\verma(\kappa_\half-i-1,\kappa_1+i+1,[\mu;-(i+1)],\zeta_{1,i}') \\
\oplus\\
\verma(\kappa_1+i-1,\kappa_\half-i+1,[\mu;\kappa-i+1],\zeta_{2,\kappa-i}') 
\end{array}
\rightarrow
M
\twoheadrightarrow
L,
 \end{equation*} 
\normalsize
 where the $\zeta_{r,j}'$ are the $\zeta'$ for $i=j$ from case r)  in \cref{Homs between vermas or hw vecs in vermas}. 
  \end{cor}
\begin{proof}The maps are constructed in the proof of \cref{Homs between vermas or hw vecs in vermas}.  By \cref{hom into subvermas}, the first map  are the inclusions in the respective summands with one of them adjusted by a sign to get a differential. The second map is a surjection by definition of $L$ with kernel equal to the sum of the two Verma modules by the proof of   \cref{fd quot good verma}. The sequence is exact, since the intersection of the images of the Verma modules is exactly the Verma module on the left by  \cref{hom into subvermas}.  
\end{proof}

			\begin{remark}
				There are good Verma modules with non-specialisable $\bo$-weights, a special case being $\mu=\pm\rootOfMinusOne q^{n/2}$ for some odd integer $n$. Nevertheless, they still behave like the nicest Verma modules. Classically, a similar  phenomenon appears in the representation theory of $\gl_2$ where non-integral weights can give rise to finite dimensional irreducible representations as long as the difference of the two weight components is a non-negative integer. The $Z$-weights in these Verma modules are still specialisable, reflecting the fact that its classical counterpart $z$ lies in $\sl_2\times \sl_2$ and hence has integral weights on blocks containing finite dimensional modules. 
			\end{remark}

		\subsection{Exceptional Verma modules}\label{bad vermas}
			Verma modules are in general not weight modules:
			\begin{prop}\label{not diagonalizable}
				Let  $w \in M(\bo\mid [\rootOfMinusOne;0])$ for a $\coideal$-module $M$.  If  $\fCoideal w , \fNewCoideal w$ are linearly independent, then the $\bo$-action  is not diagonalisable on $M$, and  not on $\verma({\kappa_\half},{\kappa_1},[\rootOfMinusOne;m],\zeta)$ for $m\in \Z$.
			\end{prop}
						
			\begin{proof}
			By \cref{serre relations as commutators}, the matrix of $\bo$ on the subspace spanned by $\fCoideal v, \fNewCoideal v$ is $
					\left(\begin{smallmatrix}
						q\inv [\rootOfMinusOne;0] & 1 \\
						1 & q[\rootOfMinusOne;0]
					\end{smallmatrix}\right).
				$ It has minimal polynomial $(x\mp[2](q- q\inv )\inv \rootOfMinusOne)^2$. It  is not diagonalizable and  $\bo$ does not act  diagonalisably on $M$. Let $m\in \Z$ and fix a highest weight vector $v$ of $\verma({\kappa_\half},{\kappa_1},[\rootOfMinusOne;m],\zeta)$.  Set  $w= F_{\pm}^m v$ for $\pm m\geq 0 $. Then $w$ satisfies the assumptions, since  it  has $\bo$-weight $[\rootOfMinusOne;0]$  by \cref{E+-F+ B0 weight}, and  $\fCoideal w, \fNewCoideal w$ are linearly independent, since the Verma module is torsion free, \Cref{PBW basis} and \cref{Verma module PBW basis} \end{proof}
			
			\begin{cor}A Verma module has a weight space decomposition if and only if it is good. 
			\end{cor}
			\begin{proof}Good Verma modules have a weight space decomposition by \Cref{weight basis good verma}.  If it is not good, then applying $\Fpm$  sufficiently many times will produce a weight vector of $\bo$ weight $[\iota; 0]$. 
					\end{proof}

			We call Verma modules of the form $\verma({\kappa_\half},\kappa_1,[\rootOfMinusOne;m],\zeta)$ for $m\in \Z$ \emph{exceptional Verma modules}. Since good Verma modules have weight bases by \cref{weight basis good verma}, \cref{not diagonalizable} shows that there is no nonzero morphism from an exceptional Verma modules to a  good one, since such a morphism would be an inclusion. We expect that there are also no morphisms  in the opposite direction:
			\begin{conj}
				Any morphism from a good Verma module to an exceptional one is zero.
			\end{conj}

			\subsection{Exceptional quotients} We briefly discuss now finite dimensional quotients of exceptional Verma modules. For this note that the subalgebra generated by $\eCoideal,\fCoideal,\kCoideal$ is isomorphic to $\Uq(\sl_2)$ as an algebra which allows us to apply $\Uq(\sl_2)$-theory.

			Although exceptional Verma modules do not have a weight basis, their finite dimensional irreducible quotients might nevertheless have one. This is illustrated by the following example.
			\begin{eg}[Weight quotients] \label{exweightmodules} Consider the subalgebra $\coideal$ of generated by $\eCoideal,\fCoideal,\dCoideal_i$ with its obvious isomorphism to  $\Uq(\gl_2)$. For  $r,\kappa\in \N$ consider the type I quantum analogue of the irreducible representation  of highest weight $(\kappa+r,r)$.  Let $v_0$ be a highest weight vector and let $ v_i = \fCoideal^{i} v_0$. Then the   $\Uq(\gl_2)$-action extends to an action of $\coideal$ by setting $\bo v_i=[\rootOfMinusOne; i]v_i$. Indeed,
				\small
				 \begin{align*}
					(\bo \fCoideal^2 - [2]\fCoideal \bo \fCoideal +\fCoideal^2\bo )v_i = & ([\rootOfMinusOne;  i+2] - [2][\rootOfMinusOne; i+1] +[\rootOfMinusOne ; i] )\fCoideal^2 v_i =0
					,\\
					(\bo \eCoideal^2 - [2]\eCoideal \bo \eCoideal +\eCoideal^2\bo )v_i = & ([\rootOfMinusOne;  i-2] - [2][\rootOfMinusOne; i-1] +[\rootOfMinusOne ; i] )\eCoideal^2 v_i =0
					,\\
					(\bo^2 \fCoideal - [2] \bo \fCoideal \bo  +\fCoideal\bo^2 )v_i = &([\rootOfMinusOne; i]^2-[2][\rootOfMinusOne; i][\rootOfMinusOne; i+1]+[\rootOfMinusOne; i+1]^2) \fCoideal  v_i = \fCoideal v_i, 
					\\
					(\bo^2 \eCoideal - [2] \bo \eCoideal \bo  +\eCoideal\bo^2 )v_i = &([\rootOfMinusOne; i]^2-[2][\rootOfMinusOne; i][\rootOfMinusOne; i-1]+[\rootOfMinusOne; i-1]^2) \eCoideal  v_i = \eCoideal  v_i.
				\end{align*}
				\normalsize
				In this case $\fNewCoideal v_i$ has the same $\dCoideal_i$-weight as $\fCoideal v_i$, hence $\fNewCoideal v_i = \beta\fCoideal v_i$ for some scalar $\beta$. Expanding $\beta\fCoideal v_i= (\bo\fCoideal -q\inv \fCoideal \bo ) v_i$ gives $
					\beta =  [\rootOfMinusOne; (i+1)] -q\inv [\rootOfMinusOne;-i] = q^i\rootOfMinusOne$. 
				Then, \[
					Zv_i = (\eCoideal \fNewCoideal - q\inv \fNewCoideal\eCoideal) v_i
					=  q^i\rootOfMinusOne(\eCoideal\fCoideal-q^{-2}\fCoideal\eCoideal)v_i
					= q^i\rootOfMinusOne([i+1][\kappa-i]-q^{-2}[i][\kappa-i+1])v_i
					.
				\]
				This shows that the $v_i$'s are weight vectors for the entire algebra $\CartanPart$ and that they form a weight basis of this $(\kappa+1)$-dimensional irreducible $\coideal$-representation. On the other hand,
				by \cref{universal property of verma}, it is a finite dimensional quotient of the exceptional Verma module $\verma(\kappa+r,r,[ \rootOfMinusOne;0], \rootOfMinusOne[\kappa])$.	
			\end{eg}

			\newcommand{\quotVector}[1]{\bar{#1}}
			
			Finite dimensional quotients of exceptional Verma modules do not always have weight bases:

			\begin{eg}[(Non-)Weight quotients] 
			\label{n=2 bad verma fd quot}
				Consider the Verma module $\verma=\verma(r+2,r,[ q^n\rootOfMinusOne;0],\zeta)$. Then $\verma$ has a finite dimensional irreducible quotient $L$ with $\Uq(\sl_2)$-character $[3]+[1]$ or $[3]$ if and only if $\zeta = \rootOfMinusOne q\inv (q^{-n}+q^n)$ or $\zeta=\rootOfMinusOne q\inv[2]$ respectively.
				
				To see this, let $v\in \verma$ be a highest weight vector. Assume the existence of such $L$. Note that $2$ is the maximal $\kCoideal$ weight in $\verma$, and the $0$-weight space for $\kCoideal$ in the Verma module is $2$-dimensional.  In case  $\chi=[3]$ we need a maximal vector, i.e.  \eqref{E+ F+aF-bv any verma} vanishes. It is easy to calculate that this holds if $\zeta=\rootOfMinusOne q\inv[2]$. Otherwise, $\chi=[3]+[1]$ and  the $2$-dimensional space survives  in $L$. Then $\verma$ must have two linearly independent maximal vectors contained in the span $\verma_{-4}$ of $\fCoideal^2 v,  \fCoideal\fNewCoideal v, \fNewCoideal^2 v$. The action of $\eCoideal$ and $\eNewCoideal$ in this basis are given by the matrices\[
				\begin{pmatrix}
					[2] & \zeta & -q\inv [2]\\
					0 & 0 & [2](\zeta - \rootOfMinusOne q\inv\alpha_n )
				\end{pmatrix}
				,\quad 
				\begin{pmatrix}
					q^{-2}[2](\rootOfMinusOne\alpha_n-q\zeta ) & 0 & 0\\
					[2] & q^{-2}(\rootOfMinusOne [2]\alpha_n -\zeta ) & -q\inv [2]
				\end{pmatrix},
				\]
				where $\alpha_n = q^{-n} + q^n$ and the top and bottom rows are the coefficients of $\fCoideal v$ and $\fNewCoideal v$ respectively. Consider the subspace $\verma_{-4}^{\eCoideal}$ of vectors in $\verma_{-4}$ annihilated  by $\eCoideal$. \hfill\\
If $\zeta \neq \rootOfMinusOne q\inv \alpha_n$, then $\verma_{-4}^{\eCoideal}$ is spanned by $ \zeta\fCoideal^2 v- [2]\fCoideal\fNewCoideal v$, contradicting that $\verma_{-4}^{\eCoideal}$ contains two  linearly independent maximal vectors.  \hfill\\	
 If $\zeta = \rootOfMinusOne q\inv \alpha_n$, then $\verma_{-4}^{\eCoideal}$ is spanned by  $
					\zeta\fCoideal^2 v- [2]\fCoideal\fNewCoideal v 
					,\quad 
					q \fNewCoideal^2 v + \fCoideal^2 v
				$
				which are both annihilated  by $\eNewCoideal$. In this case one checks that $\eCoideal (\zeta \fCoideal v+ [2] \fNewCoideal v) = 0$, but $\eNewCoideal (\zeta \fCoideal v+ [2] \fNewCoideal v) = -q\inv (\alpha_n^2+[2]^2) v\neq 0$ for all $n$. Hence $\verma(r+\kappa,r,[\rootOfMinusOne;q^n],\rootOfMinusOne q\inv \alpha_n)$ has a $4$-dimensional quotient $L$ with basis $
					\quotVector{v}$, the image of $v$, $
					\fNewCoideal \quotVector{v}
						, \, 
					\fCoideal \quotVector{v}
						, \, 
					\fCoideal ^2\quotVector{v} =\rootOfMinusOne\alpha_n q\inv [2]\inv \fCoideal\fNewCoideal \quotVector{v} = q\inv  \fCoideal^2 \quotVector{v}	.
				$
				As $\Uq(\sl_2)$-module, $L$ splits into $\langle \fNewCoideal \quotVector{v}\rangle\oplus\langle \quotVector{v}, \fCoideal \quotVector{v}, \fCoideal^2 \quotVector{v}\rangle$. Neither of the summands are $\coideal$-submodules, hence $L$ is irreducible as a $\coideal$-modules. Further, depending on $n$, we have two different situations:
				\begin{itemize}
					\item  If $n\ne 0$, then we have $
					\fCoideal \quotVector{v} + \rootOfMinusOne q^n \fNewCoideal \quotVector{v} \in L(\bo\mid [\rootOfMinusOne ; n+1]),
					\fCoideal \quotVector{v} - \rootOfMinusOne q^{-n} \fNewCoideal \quotVector{v} \in L(\bo\mid [\rootOfMinusOne ; n-1]).
				$
				In particular, the $\bo$-action on $L$ is  diagonalizable. Moreover, $L$ is in fact a weight module.
				\item If $n=0$, the matrix of $\bo$ on the subspace spanned by $\fCoideal\quotVector{v},\fNewCoideal\quotVector{v}$ is given by 
				\[
					\left(\begin{smallmatrix}
						2q\inv \rootOfMinusOne  (q-q\inv)\inv  & 1 \\
						1 & 2q\rootOfMinusOne (q-q\inv)\inv 
					\end{smallmatrix}\right). 
				\]
				It  has minimal polynomial $\left(t- \frac{q+q\inv}{q-q\inv }\rootOfMinusOne\right)^2=(t-[q\rootOfMinusOne;0])^2$, and  $\bo$ is not diagonalisable  on $L$.
				\end{itemize}
\end{eg}

			\begin{eg}[$\Uq(\sl_2)$-character]\label{char}
		By similar arguments as given above, one can show that $\verma=\verma(r+3,r,[ q^n\rootOfMinusOne;0],\zeta)$ has a finite dimensional quotient if and only if $\zeta$ takes one of the values in the following table, in which case the corresponding $\Uq(\sl_2)$-character of the quotient is as indicated:
				\begin{center}
					\begin{tabular}{|c|c|c|c|c|c|}
						\hline
						$\zeta$ & $\rootOfMinusOne q^{n}[3] $ & $\rootOfMinusOne(q^{n-1}[2] + q^{-n-2})$  & $\rootOfMinusOne(q^{n-2} + q^{-n-1}[2])$  & $\rootOfMinusOne q^{-n}[3] $ \\
						\hline
						$\Uq(\sl_2)$-character & $[4]$ & $[4]+[2]$ &  $[4]+[2]$ & $[4]$ \\
						\hline
					\end{tabular}
				\end{center}

$\verma=\verma(r+4,r,[ q^n\rootOfMinusOne;0],\zeta)$ has a finite dimensional quotient if and only if $\zeta$ takes one of the values

				\begin{center}
					\begin{tabular}{|c|c|c|c|c|c|}
						\hline
						$\zeta$ 
						& $\rootOfMinusOne q^{n}[4] $ 
						& $\rootOfMinusOne(q^{n-1}[3] + q^{-n-3})$ 
						& $\rootOfMinusOne(q^{n-2}[2] + q^{-n-2}[2])$ 
						& $\rootOfMinusOne(q^{n-1} + q^{-n-1}[3])$  
						& $\rootOfMinusOne q^{-n}[4] $ \\
						\hline
						$\sl_2$ character & $[5]$ & $[5]+[3]$ & $[5]+[3]+1$ & $[5]+[3]$ & $[5]$ 
						\\
						\hline
					\end{tabular}
				\end{center}
			\end{eg}
	These examples hint towards a pattern for the possible $\zeta$-values and $\Uq(\sl_2)$-characters. By mimicking the proof of \cref{fd quot good verma}, using the fact that  $
					[\rootOfMinusOne; n+\kappa-2j] - q^{-\kappa} [\rootOfMinusOne;n]
				= \rootOfMinusOne (q^{n-j}[\kappa-j]+ q^{j-\kappa-n}[j] ),$
			we see that exceptional Verma modules which admit finite dimensional quotients satisfy 			\begin{equation}
					\zeta= \rootOfMinusOne (q^{n-j}[\kappa-j]+ q^{j-\kappa-n}[j] ), \quad 0\neq j\neq \kappa-1.
					\label{zeta value in nonclassical verma}
				\end{equation}	
			We expect that it in fact characterises those Verma modules:		
			\begin{conj}[Finite dimensional exceptional quotients]\label{exceptional verma fd quotient}
				Let $ n\in \Z$. The exceptional Verma module $\verma(r+\kappa,r,[\rootOfMinusOne;n],\zeta )$ has a finite dimensional irreducible quotient $L$  if and only if	\eqref{zeta value in nonclassical verma} holds. 	
				In this case the $\Uq(\sl_2)$-character of $L$ is $[\kappa+1]+[\kappa-1]+\cdots +[\kappa+1-2\min\{j,\kappa-j\}]$.
			\end{conj}
\begin{rem}One might want to define an analogue of the BGG-category  $\mathcal{O}$ for $\coideal$-modules. 
There are at least two reasons, why exceptional Verma modules might not be chosen as objects in $\mathcal{O}$. 
Firstly they are not weight modules and secondly  their submodules are not well-behaved.  

To illustrate this we go back to \Cref{n=2 bad verma fd quot}. Let $M_{-2}^{{\eCoideal},X}$ be the subspace of maximal vectors in the span of  $\fCoideal^2 v,  \fCoideal\fNewCoideal v, \fNewCoideal^2 v$.  Assume first  the highest $\kCoideal$-weight is $\mu = \rootOfMinusOne q^n$ with $n\geq 2$. In $\verma$, the $(\geq -4)$-weight spaces for $\kCoideal$ are then all diagonalizable with respect to $\bo$ and the behaviour of the exceptional Verma module is as of a good Verma module, cf. \Cref{hom into subvermas}.  When $n=1$, $\verma_{-4}^{{\eCoideal},X}$ is spanned by $\Fpm_\pm\Fpm_+ v$. When  $n=0$,  $ \Fpm_+^2 v $ is still a maximal vector, but $\bo$ acts on the two dimensional space $\verma_{-4}^{{\eCoideal},X}$ via a $2\times 2$ Jordan block with eigenvalue $[\rootOfMinusOne q^2;0]$. In  fact, the kernel of the finite dimensional irreducible quotient is an \emph{extension} of two highest weight modules of the same highest weight. In category $\mathcal{O}$ for semisimple Lie algebras such extensions do not exist. 
\end{rem}

\subsection{Finite dimensional modules}
We finish by showing that, after a suitable field extension, every finite dimensional irreducible $\coideal$-module appears as a quotient of some Verma module. 
	\begin{lem}\label{nilpotent EFXY}
	On any finite dimensional $\coideal$-module, ${\eCoideal},{\fCoideal}, X, \fNewCoideal$ act nilpotently. 
	\end{lem}
	\begin{proof}
		By the presentation of $\coideal$ and \cref{csa and [XY]} the triples $({\eCoideal},{\fCoideal},\kCoideal)$ and $(X,-q\fNewCoideal , \kCoideal)$ both generate an algebra isomorphic to  $\Uq(\sl_2)$ and thus the claims follow from  $\Uq(\sl_2)$-theory.
	\end{proof}	
	\begin{prop}\label{hw vec exists}
		After a suitable field extension, any finite dimensional $\coideal$-module has a weight vector killed by ${\eCoideal}, X$,  and  is therefore a quotient of some Verma module.
	\end{prop}

	\begin{proof}
		For any $A\subset\coideal$, let $V^A\subset V$ denote the subspace of vectors killed by $A$. By \cref{nilpotent EFXY}, $V ^{\eCoideal}\ne 0$. By
		\cref{serre relations as commutators}, $V^{\eCoideal}$ is invariant under $X$ and thus $V ^{{\eCoideal},X}\ne0$ again by  \cref{nilpotent EFXY}.  By \cref{serre relations as commutators}, $V^{{\eCoideal},X}$ is invariant under $\dCoideal_i,\bo$, and then also under $Z,W$ by \cref{EZ FZ commutator}-\cref{YW commutator}. After a suitable extension of scalars, we can find a simultaneous eigenvector, i.e. a weight vector in $V^{{\eCoideal},X}$.
	\end{proof}

\section{Rational representations and tensor product representations}\label{rational representations}
In this section we consider a particularly nice class of finite dimensional representations, namely the rational representations, and study the tensor products $\vecrep^{\otimes d }$. 
\subsection{Rational representations}
Since the triple $(\fCoideal,\kCoideal. \eCoideal)$ generates a subalgebra of $\coideal$ isomorphic to $\Uq(\sl_2)$,  the action of  $\kCoideal$ on any finite dimensional representation  is diagonalizable with eigenvalues $\pm q^n$ with $n\in \Z$.  We now ask for a similarly nice behaviour for the action of $\bo$.

		\begin{definition}[Rationality conditions]\label{rational representation def}
			We say that a $\coideal$-module $V$ is a \emph{rational representation} if $V$ is finite dimensional, the $\bo$-action is diagonalizable with eigenvalues being quantum integers, and the $\dCoideal_i$'s are also diagonalizable with eigenvalues being in $q^\Z$.
		\end{definition}
			\begin{eg}\label{Visrational}
			The natural representation $\vecrep$ from \Cref{nota2} is rational. In fact, setting $x^\pm=\basis_{\negidx{\half}} \pm \basis_\half, y^\pm = \basis_{\negidx{1}}\pm\basis_1 $, and $W_\pm= \langle x^\pm, y^\pm \rangle$, we have $\vecrep=W_+ \oplus W_-$ as $\coideal$-modules. Moreover, one checks directly that\[
				\bo x^\pm = \pm x^\pm 
					,\quad 
				\bo y^\pm = 0
					,\quad
				\dCoideal_\half x^\pm = q^{\pm 1} x^\pm
					,\quad
				\dCoideal_\half y^\pm = y^\pm
					,\quad
				\dCoideal_1 x^\pm = x^\pm
					,\quad
				\dCoideal_1 y^\pm = q^{\pm 1} y^\pm
					.
			\]
		\end{eg}	
		\begin{lem}\label{quots}
			Quotients and subrepresentations of rational representations are rational.
		\end{lem}
		\begin{proof}It suffices to show that on subrepresentations $V'$ of rational representations $V$, the $\bo$-action is still diagonalizable with eigenvalues being quantum integers.
			Suppose $0\ne w\in V'$. Write $w= \sum_{i\in \Z} a_i  v_i $ as a linear combination of  $\bo$-weight  vectors $v_i\in V$ of weight $[i]$. Then the coefficients of $\bo^j w$ assemble into the Vandermonde matrix $([i]^j)_{i,j}$, which is invertible since the quantum integers are distinct.  Hence  $v_i\in V'$,  for each $i$, as it is a linear combination of $\bo^j w$'s.		\end{proof}
		\newcommand{\partitionSetZhalf}{\mathbf{P}_{1/2}}
		\newcommand{\halfIntegersSet}{\Z+1/2}
The irreducible finite dimensional quotient $\ratIrrep(\kappa_\half,\kappa_1,[n],[n+\kappa-2i] - q^{-\kappa}[n])$  of the Verma module $\verma(\kappa_\half,\kappa_1,[n],[n+\kappa-2i] - q^{-\kappa}[n])$ with $n\in \Z$, $\kappa_\half,\kappa_1 \in \Z$, $\kappa_\half>\kappa_1$ is a rational representation. The rationality conditions are obvious except of finite-dimensionality which also holds  because 
the highest weight $(q^{\kappa_\half},q^{\kappa_1},[n], [n+\kappa-2i] - q^{-\kappa}[n])$ satisfies the dominance condition \eqref{dominant}.  
Via \eqref{whyhalves},  irreducible rational representations for $\coideal$ should be labelled by pairs of elements from 
\begin{align}
			\partitionSetZhalf = \{ (\lambda_1,\lambda_2) \mid \lambda_1- \lambda_2 \in \N,\,  2\lambda_i\in \Z  \}.
		\end{align}

				\begin{thm}\label{classification rational rep}
			Let $V$ be an irreducible rational representation. Then  
			$V\cong\ratIrrep(\kappa_\half,\kappa_1,[n],\zeta )$ with
			$n\in \Z$, $\kappa_\half,\kappa_1 \in \Z$, $\kappa_\half>\kappa_1$, and $\zeta= [n+\kappa-2i] - q^{-\kappa}[n]$ for some $0\leq i\leq \kappa$. 
		
			Moreover, taking the highest weight gives a bijection	
			 \begin{align}	\label{bijection irrep rational to partitions}
				\left\{ 
				\begin{minipage}[c]{5.2cm}
				isomorphism classes of irreducible 
				 \\rational representations of $\coideal$ 
				\end{minipage}
				\right\}\;\;
				& \cong\qquad\qquad\qquad\partitionSetZhalf\times \partitionSetZhalf.\\
			\ratIrrep(\kappa_\half,\kappa_1,[n],[n+\kappa-2i] - q^{-\kappa}[n])\quad\;\;& \mapsto \;\;
				\left((\frac{\kappa_\half+ n}{2}, \frac{\kappa_\half +n }{2} -i ), (\frac{\kappa_\half - n}{2}, \frac{\kappa_\half - n }{2} -\kappa+ i ) \right).\nonumber
		\end{align} 
In particular, $V$ decomposes over  $\groundring$ into a direct sum of $1$-dimensional weight spaces. 
		\end{thm}
		\begin{proof}	
			By \cref{hw vec exists}, any rational representation $V$ is over a suitably extended field,  a quotient of some Verma module $\verma$. By  \cref{fd quot good verma}, $\verma$ must be of the form  $\verma=\verma(\kappa_\half,\kappa_1,[n],\zeta)$ with  the conditions as desired.  Again by  \cref{fd quot good verma}, $V$ has a  basis  consisting of weight vectors given by vectors of the form $\Fpm_+^a \Fpm_-^b v$ in $\verma$, 			
			where $v$ is a highest weight vector. The weight spaces are in particular at most $1$-dimensional. 
		 Moreover the weights of $V$ must satisfy the following.  The  $\dCoideal_i$-weights are obviously  in $q^\Z$ whereas the  $\bo$-weights are quantum integers by \cref{E+-F+ B0 weight}, and the $Z$-weight lie in $\Z[q^{\pm1}]$  lie in $\Z[q^{\pm1}]$ by  \cref{F+- act on weight space any verma}, their $Z$-weight. Thus the weight spaces are already defined over $\groundring$ and we do not need to pass to a field extension. \hfill\\
Note that the assignment in \eqref{bijection irrep rational to partitions} is well-defined. Now it suffices to show that it induces a bijection between the set of tuples $(\kappa_\half,\kappa_1,[n],[n+\kappa-2i] - q^{-\kappa}[n])$ and the set $\partitionSetZhalf\times \partitionSetZhalf$. This is clear, since  $((a,b),(c,d))\mapsto (a+c,b+d,[a-c],[b-d]-q^{b+d-a-c}[a-c])$ defines an inverse. 
		\end{proof}	
			\newcommand{\factorWedge}{\eta}
			\newcommand{\factorVone}{x}
			\newcommand{\factorVtwo}{y}
			\newcommand{\funnierVector}{\Lambda}	
			\subsection{Polynomial representations}	
	Classically, for $\gl_n$, the set of irreducible  rational representations contains the set of polynomial representations which are often defined as summands in some tensor power of the natural representation. 
	We therefore refer to subrepresentations of a finite direct sums of $\vecrep^{\otimes d}$'s as  \emph{polynomial representations}.  We claim they are rational in our sense:	
\begin{thm}[Funny vectors]\label{funny vectors}
			If $v\in \vecrep^{\otimes d}(\bo,[n])$ with $d\in\N$, $n\in\Z$, then using \Cref{nota2},
			\begin{equation}\label{fun}
				v\otimes ( \basis_{\negidx{\half}} + q^n \basis_\half ) \in \vecrep^{\otimes d+1}(\bo,[n+1]) 
			,\quad 
			v\otimes (   \basis_{\negidx{\half}} -q^{-n}\basis_{\half} )  \in \vecrep^{\otimes d+1}(\bo,[n-1]).
			\end{equation}
			In particular, the action of $\bo$ on $\vecrep^{\otimes d }$ is diagonalizable with quantum integers as eigenvalues.
		\end{thm}
Since finite direct sums of rational representations are rational, we obtain with \Cref{quots}:
\begin{cor}Any polynomial representation, in particular  $\vecrep^{\otimes d}(\bo,[n])$,  is rational. 
		\end{cor}

		\begin{proof}
			For \eqref{fun} see \cite[Thm. 2.6]{StWoj-coidealDiagrammatics}. The authors worked with $\Uq'(\gl_1\times \gl_1)$, but the element $\bo$ has exactly the same form, so the proof there applies here.  It shows that if $v$ is an eigenvector of $\bo$ with eigenvalue $[n]$, then $v\otimes  ( \basis_{\negidx{\half}} + q^n \basis_\half )$ is an eigenvector of $\bo$ with eigenvalue $[n+1]$,   and $v\otimes (\basis_{\negidx{\half}} -q^{-n}\basis_{\half} )$ is an eigenvector of $\bo$ with eigenvalue $[n-1]$. Thus \eqref{fun}  holds.
			
Using $\Delta(\bo) = 1\otimes \bo +\bo\otimes K_0\inv $ one directly checks that $v\otimes \basis_{\pm1}$ is again an eigenvector of $\bo$ with eigenvalue $[n]$.  Together with  \eqref{fun}  we con construct from $v$ four, which equals the dimension of $\vecrep$, linearly independent new eigenvectors in  $\vecrep^{\otimes {d+1}}$ with eigenvalues being quantum integers.  Now the theorem follows by  induction on $d$ with base case $d=1$ given by \Cref{Visrational}.  Clearly, the $\dCoideal_i$'s are also diagonalizable with eigenvalues in $q^\Z$.
\end{proof}
\begin{definition}[Funny vectors] The following vectors will be important: 			
\begin{equation}
				\factorVone^\pm_{n} \coloneqq \vecBasis[\negidx{\half}] \pm q^n \vecBasis[\half]
				,\quad
				\factorVtwo^\pm_{n} \coloneqq \vecBasis[\negidx{1}] \pm q^n \vecBasis[1] \in \vecrep,
				\quad 
				\factorWedge_{n}^\pm \coloneqq 
				q\inv \factorVtwo^\pm_{n}\otimes \factorVone^\pm_{n}- \factorVone^\pm_{n}\otimes \factorVtwo^\pm_{n-2}
				\in \vecrep^{\otimes 2}
				.
				\label{factor def}
		\end{equation}		
\end{definition}
\begin{rem}\label{J1action}
The vectors $\factorVone^\pm_{0} , \factorVtwo^\pm_{0}$ are eigenvectors for the $K$-matrix from \Cref{Kmatrix}, and thus for  $J_1$ from \Cref{Js},  with  $\kmatrix\factorVone^\pm_{0} \otimes w =\pm\factorVone^\pm_{0} \otimes w $ and  $\kmatrix \factorVtwo^\pm_{0} \otimes w =\pm \factorVtwo^\pm_{0} \otimes w $ for any $w\in \vecrep^{\otimes d}$. 
\end{rem}
		\begin{ex}\label{J2action} 		
The following vectors are maximal in $\vecrep^{\otimes 2}$ 
		\[
				\factorVone^+_{0} \otimes \factorVone^+_{1}
				,\quad
				\factorVone^+_{0} \otimes \factorVtwo^-_{-1}
				,\quad 
				\factorVtwo^-_{0} \otimes \factorVone^+_{-1}
				,\quad
				\factorVtwo^-_{0} \otimes \factorVtwo^-_{1}
				,\quad
				\factorWedge_{0}^\pm.
			\]	
		We leave it to the reader  to  compare with the classical situation, and to show that these vectors are in fact representing.  We see here that  the naive tensoring of weight vectors in $\vecrep$ does not give weight vectors in $\vecrep^{\otimes 2}$ which reflects the fact that $\coideal$ is not a Hopf algebra. However, by incorporating the parameters $n$ in \eqref{factor def}, we obtain maximal vectors in $\vecrep^{\otimes 2}$.  Moreover, the vector $\factorWedge_{0}^\pm\in\vecrep^{\otimes 2}$ is an  eigenvector for $J_2$ with eigenvalue $\pm q^2$.  To see this recall $J_2=H_1H_0H_1$ and  check, using \eqref{eqn:R-matrix generic form}, that $\factorWedge_{0}^\pm\in\vecrep^{\otimes 2}$ is in fact an eigenvector for $H_1$ with eigenvalue $-q$  and an eigenvector for $H_0$ with eigenvalue $\pm1$ by \Cref{J1action}. These phenomena will generalise to higher tensor powers.
\end{ex}		
\subsection{Quantum wedges}
In this section we construct some maximal vectors in  $\vecrep^{\otimes d}$ which generate quantum analogues of tensor products of wedge product representations.  We use these vectors later to identify  irreducible polynomial representations. We start with  some combinatorics.
\begin{definition}
A \emph{parity standard bitableau} is a standard two-row bitableaux of shape $((s,s),(t,t))$  for some $s,t \in \N$, such that the entries are odd in the first rows and even in the second rows.
\end{definition}
\begin{definition}\label{specialfilling}
A  parity standard bitableau  where the numbers are filled with the numbers in order from top to bottom and then left to right is called \emph{special}. 
	\end{definition}
Note that a parity standard bitableau is completely determined by the first row, and every rectangular two-row bipartition  has a unique special filling, see  \eqref{exparitystandard} for  $({\youngcenter{3,3},\youngcenter{2,2}})$.
We can rebuild any parity standard bitableau from $(\emptyset,\emptyset)$ by successively adding columns of two boxes in the order specified by the first row of the bitableau. 

We illustrate this by an example:	
\begin{eg}	 Of the following bitableaux the first two are parity standard and the first special.
\begin{equation}\label{exparitystandard}
\left(\ytabcenter{
						1 & 3 & 5\\
						2 & 4  & 6
						}
					\, ,\,
					\ytabcenter{
						7 & 9 \\
						8 & 10 
						}\right)
						,\quad\quad
				\left(\ytabcenter{
					1 & 3 & 9\\
					2 & 4  & 10
					}
				\, ,\,
				\ytabcenter{
					5 & 7 \\
					6 & 8  
					}\right)
					,\quad\quad
					\left(\ytabcenter{
						1 & 3 & 6\\
						2 & 4  & 10
						}
					\, ,\,
					\ytabcenter{
						5 & 7 \\
						9 & 8  
						}\right).
			\end{equation}
			The first one is special since it can be obtained from $(\emptyset,\emptyset)$ by adding the columns $\ytabcenter{1\\2}\;$, $\ytabcenter{3\\4}\;, \ldots\;, \ytabcenter{9\\10}$ in order.  The second can be built by adding the columns $\ytabcenter{1\\2}$ and $\ytabcenter{3\\4}$ to the first component, then the columns $\ytabcenter{5\\6}$ and $\ytabcenter{7\\8}$ to the second component, and finally the column $\ytabcenter{9\\10}$ to the first component.
			\end{eg}		
					\begin{rem}\label{path}A parity standard bitableau can be viewed graphically as a path in the directed graph, where vertices are bipartitions of shape $((s,s),(t,t))$ and  arrows are given by the inductive  rules in \cref{funny wedges} as follows (where we draw only a finite piece of the graph):
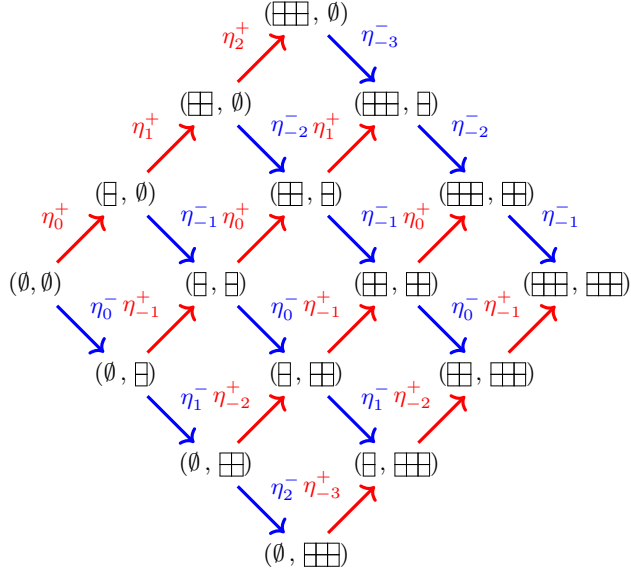
\ytableausetup{boxsize=2pt}
{\small			
\begin{equation*}
					\begin{tikzpicture}[scale=1.2]
						\foreach \i in {1,2,3} {
							\foreach \j in {1,2,3} {
								\node (\i\j) at ({\i+\j},{\i-\j}) {$(\youngcenter{\i,\i}\, ,\,\youngcenter{\j,\j})$};
							}
						}
					\foreach \i in {0} {
						\foreach \j in {1,2,3} {
							\node (0\j)at ({\j},{-\j}) {$(\emptyset\, ,\,\youngcenter{\j,\j})$};
						}
					}
					\foreach \i in {0} {
						\foreach \j in {1,2,3} {
							\node (\j0) at ({\j},{\j}) {$(\youngcenter{\j,\j}\, ,\,\emptyset)$};
						}
					}
					\node (00) at (0,0) {$(\emptyset,\emptyset)$};

					\foreach \i/\iplus in {0/1,1/2,2/3} {
						\foreach \j/\jplus [evaluate=\j as \n using int(\i-\j),evaluate=\j as \minusn using int(\j-\i)]  in {0/1,1/2,2/3} {
								\draw[->,red] (\i\j) -- (\iplus\j) node[midway,above left] {$\factorWedge^+_{\n}$};
							
								\draw[->,blue] (\i\j) -- (\i\jplus) node[midway,above right] {$\factorWedge^-_{\minusn}$};
							
						}
					}

					\foreach \i in {3} {
						\foreach \j/\jplus [evaluate=\j as \n using int(\i-\j),evaluate=\j as \minusn using int(\j-\i)]  in {0/1,1/2,2/3} {
							
								\draw[->,blue] (\i\j) -- (\i\jplus) node[midway,above right] {$\factorWedge^-_{\minusn}$};
							
						}
					}

					\foreach \i/\iplus in {0/1,1/2,2/3} {
						\foreach \j [evaluate=\j as \n using int(\i-\j),evaluate=\j as \minusn using int(\j-\i)]  in {3} {
								\draw[->,red] (\i\j) -- (\iplus\j) node[midway,above left] {$\factorWedge^+_{\n}$};
							
						}
					}
					
					\end{tikzpicture}
				\end{equation*}}
				%\ytableausetup{smalltableaux}
			
			\noindent	The tensor $\funnierVector(S,T)$ below is then exactly the tensor product of all factors attached to the arrows of the path corresponding to $(S,T)$. 
				In general different paths give different tensors in $\vecrep^{\otimes 2d}$. 
			\end{rem}

To each parity standard bitableau $(S,T)$ with $2d$ boxes, we associate a vector $\funnierVector{(S,T)}\in\vecrep^{\otimes 2d}$.  
			
			\begin{definition}\label{funny wedges}
				Let $(S,T)$ be a parity standard bitableau with $2d$ boxes. Define $\funnierVector{(S,T)}$ by 
\[	\funnierVector(\emptyset,\emptyset) = 1\in\groundring=\vecrep^{\otimes 2\cdot0},\quad\text{and otherwise}\quad 
					\funnierVector{(S,T)} =\funnierVector{(S',T')}  \otimes \eta^\pm_n  \,\in\, \vecrep^{\otimes 2d}.
				\]
Here,  $(S',T')$ is the parity standard bitableau obtained from $(S,T)$ by removing the column $c$ with the largest numbers. And $\eta^\pm_n$ is as in \eqref{factor def} with  $n$ being the absolute value of the difference of the number of columns in $S'$ and $T'$. The sign is $+$ if  $c$ got removed in $S$ and $-$ if it got removed in $T$.  
			\end{definition}

\begin{eg} \label{exemptyempty} By convention, the vector $\funnierVector(\emptyset,\emptyset)$ is the basis vector of the trivial representation of $\coideal$. Thus  $\bo\funnierVector(\emptyset,\emptyset)=\fCoideal\funnierVector(\emptyset,\emptyset)=\eCoideal\funnierVector(\emptyset,\emptyset)=0$, $\kCoideal \funnierVector(\emptyset,\emptyset)=\funnierVector(\emptyset,\emptyset)$. 
\end{eg}					
			
	\begin{eg}
For the special parity standard bitableau in \cref{exparitystandard}, we have 
			\[
				\funnierVector\left(\ytabcenter{
					1 & 3 & 5\\
					2 & 4  & 6}
				\, ,\,
				\ytabcenter{
					7 & 9 \\
					8 & 10 
					}\right) =  \factorWedge_{0}^+\otimes \factorWedge_{1}^+ \otimes \factorWedge_{-2}^-\otimes \factorWedge_{-1}^-\otimes \factorWedge_{0}^+.\]
			This is also an eigenvector for $\bo,\kCoideal$ with eigenvalues $[5]$ and $1$ respectively. It also turns out to be a maximal vector and a common eigenvector for the Jucys--Murphy elements with eigenvalues \[
				1,q^{2},q^{-2},1,-1,-q^{2},-q^2,-1,q^4,q^2.
			\]
This can be checked inductively using the following theorem. 
\end{eg}			

		\begin{thm}[Quantum wedges] \label{funnier vectors trivial rep}
				Let  $(S,T)$ be a parity standard bitableau with  $2d$ boxes. Let $\content$ be the content function, and  $m$ the difference of the number of boxes in $S$ and $T$. 
				\begin{enumerate}\item \label{max1} We have $
						\bo \cdot\funnierVector(S,T) = [m]\, \funnierVector(S,T)
						,\quad 
						\kCoideal\cdot\funnierVector(S,T) = \funnierVector(S,T).
					$
					\item  \label{max2}We have $\eCoideal\,\funnierVector(S,T)= \fCoideal\,\funnierVector(S,T)=0$. In particular, $\funnierVector(S,T)$ is a maximal vector.
					
					\item  \label{max3}The action of the Jucys-Murphy element $J_i$ is determined by contents, namely
				 \[
						J_i \cdot\funnierVector(S,T) = \pm q^{-2\content(i)} \funnierVector(S,T),
					\]
					where the sign is $+$ if the $i$-th box of the tableau is in the first component and $-$ otherwise. The central element $J:=J_1\cdots J_d$ acts on $\funnierVector(S,T)$ by $q^{-2c}$, where $c$ is the sum of all contents.

				\end{enumerate}
				
			\end{thm}

The main technical tool is to incorporate  the Hecke algebra $\heckeSpecialized{d}$ from  \Cref{DefHecke}

			\begin{lem}\label{ridiculous relations}
				The following statements are true for all $m,n\in \Z$. 
\begin{enumerate}
\item Write $\alpha_m=q^m+q^{-m}$ for $m\in \Z$. In $\vecrep^{\otimes 2}$ the following hold. 											
\begin{eqnarray} 
		\label{H eta}
		&H_1 \factorWedge_{n}^\pm = -q \factorWedge_{n}^\pm
					,\quad
					H_1 \factorWedge_{n}^\pm  \otimes \factorVone^\pm _{m} = q\inv \factorWedge_{n}^\pm \otimes \factorVone^\pm _{m}
					, \quad
					 \eCoideal \factorWedge_{n}^\pm  = 0, \quad \fCoideal\factorWedge_{n}^\pm =0,&
					 \\
					\label{H xn xn+1}
					&H_1(\factorVone^\pm _{n} \otimes \factorVone^\pm _{n+1})  = q\inv \factorVone^\pm _{n} \otimes \factorVone^\pm _{n+1},\quad 
					%\label{H yn x n+1}
					H_1	(\factorVtwo^\pm _{n} \otimes \factorVtwo^\pm _{n+1})  = q\inv \factorVtwo^\pm _{n} \otimes \factorVtwo^\pm _{n+1},&
			\end{eqnarray}
			\begin{gather}
					\label{H xplus xminus}
				\begin{aligned}
					H_1( \factorVone^-_{-n} \otimes \factorVone^+_{n-1} )
					& = 
					%\frac{q^{2} + q^{2 \, n}}{q + q^{2 \, n + 1}} 
					\alpha_n\inv  (\alpha_{n-1}\factorVone^+_{n} \otimes \factorVone^-_{-n-1}  
					+
					%\frac{q^{2} - 1}{q + q^{2 \, n + 1}}
					q^{-n}(q\inv-q) \factorVone^-_{-n} \otimes \factorVone^+_{n-1})
					,
					\\
					H_1( \factorVone^+_{n} \otimes \factorVone^-_{-n-1}  ) 
					&=
					%-\frac{{\left(q^{2} - 1\right)} q^{2 \, n}}{q + q^{2 \, n + 1}}
					\alpha_n\inv 
					(
						q^n(q\inv-q)
					\factorVone^+_{n} \otimes \factorVone^-_{-n-1}
					%+\frac{q^{2 \, n + 2} + 1}{q + q^{2 \, n + 1}}
					+\alpha_{n+1}
					\factorVone^-_{-n} \otimes \factorVone^+_{n-1} )
					.
				\end{aligned}
			\end{gather}
\item In $\vecrep^{\otimes 3}$ the following hold. 	
 \begin{gather}						
				\begin{aligned}
						\label{H2 eta n x n+1}
H_1 H_2 (\factorWedge_{m}^\pm \otimes \factorVone^\pm _{m+1}) &=  (q\inv H_2 - q^{-2}) \factorWedge_{m}^\pm \otimes \factorVone^\pm _{m+1}.
				\end{aligned}
				\end{gather}

\item  In $\vecrep^{\otimes 4}$ 	the following hold. 	
                                   {\small \begin{gather}						
				\begin{aligned}
						\label{H12 eta n eta n+1}
J_1 H_1 H_2(\factorWedge_{0}^\pm \otimes \factorWedge^\mp _{-1}) &=  
\mp H_1H_2(\factorWedge_{0}^\pm \otimes \factorWedge^\mp _{-1})
\pm (q \inv - q) H_2 (\factorWedge_{0}^\pm \otimes \factorWedge^\mp _{-1})
\pm(1-q^{-2}) \factorWedge_{0}^\pm \otimes \factorWedge^\mp _{-1},\\
					H_1H_2(\factorWedge_{m}^\pm \otimes \factorWedge^\pm_{m+1}) &=   (q\inv H_2 - q^{-2}) \factorWedge_{m}^\pm \otimes \factorWedge^\pm _{m+1},
				 	\end{aligned}
				\end{gather}}
Moreover, there exists $h_m \in \heckeSpecialized{4}$ such that
					$\factorWedge_{m}^- \otimes \factorWedge^+ _{-m-1} =  h_m (\factorWedge_{m}^+ \otimes \factorWedge^- _{-m-1} ).$
\item Furthermore, the following vectors, in $\vecrep^{\otimes 2}$ respectively $\vecrep^{\otimes 4}$, are eigenvectors 
				\begin{gather}
				\begin{aligned}
					\label{base case funnier vectors}
					\funnierVector\left(\ytabcenter{
						1 \\2  
					},\emptyset\right)&= \factorWedge_{0}^+
					\quad \quad &
					\funnierVector\left(\emptyset,\ytabcenter{
						1 \\2  
					}\right)& =\factorWedge_{0}^-		
						\\
					\funnierVector\left(\ytabcenter{
						1 \\2
					},\ytabcenter{
						3 \\4
					}\right) &=
					\factorWedge_{0}^+ \otimes \factorWedge_{-1}^- 
					\quad \quad &
					\funnierVector\left(\ytabcenter{
						3 \\4
					},\ytabcenter{
						1 \\2
					}\right)&=
					\factorWedge_{0}^- \otimes \factorWedge_{-1}^+ 	
				\end{aligned}
				\end{gather}
for the Jucys--Murphy elements $J_i$ with eigenvalues $q^{-2\content(i)}$.	
\end{enumerate}
\end{lem}
			\begin{proof}
			The formulas follow from a lengthy, but straightforward, calculation which we omit. The second equation in \cref{H xplus xminus} is obtained from the first by applying $H_1$ and using its quadratic relation.
			Take $h_m= a_m + b_m H_2 + c_m H_1H_2+d_m H_3H_2+e_m H_1H_3H_2+f_m H_2H_1H_3H_2$, where \begin{eqnarray*}
					a_m &=&  \frac{{\left(q^{6} - q^{4} - q^{2} + 1\right)} q^{4 \, m}}{{\left(q^{4} + q^{2}\right)} q^{2 \, m} + q^{4 \, m + 6} + 1}
					,
				\end{eqnarray*}\[
					b_m =  \frac{{\left(q^{3} - q\right)} q^{4 \, m} + {\left(q^{3} - q\right)} q^{2 \, m}}{{\left(q^{4} + q^{2}\right)} q^{2 \, m} + q^{4 \, m + 6} + 1}
					,\quad 
					c_m=d_m = -\frac{{\left(q^{4} - q^{2}\right)} q^{4 \, m} + {\left(q^{4} - q^{2}\right)} q^{2 \, m}}{{\left(q^{4} + q^{2}\right)} q^{2 \, m} + q^{4 \, m + 6} + 1}
					,
				\]\[
					e_m =  \frac{{\left(q^{5} - q^{3}\right)} q^{4 \, m} + {\left(q^{5} - q^{3}\right)} q^{2 \, m}}{{\left(q^{4} + q^{2}\right)} q^{2 \, m} + q^{4 \, m + 6} + 1}
					,\quad 
					f_m = \frac{q^{2} + q^{2 \, m + 2}}{q^{2 \, m + 4} + 1}
					,
				\]
and calculate directly its action on $\factorWedge_{m}^+ \otimes \factorWedge^- _{-m-1} $. 			For the last part note that the $J_i$-eigenvalues for $i=1,2$ follow from \eqref{H eta}, see also \cref{J2action}. Using the first equation in \cref{H12 eta n eta n+1} we get	
\begin{align*}
	J_3(\factorWedge_{0}^\pm\otimes \factorWedge_{-1}^\mp) 
	=& H_2H_1J_1H_1H_2(\factorWedge_{0}^\pm\otimes \factorWedge_{-1}^\mp) \\
	=&  H_2H_1(\mp H_1H_2\factorWedge_{0}^\pm\otimes \factorWedge_{-1}^\mp
	\pm (q \inv - q) H_2(\factorWedge_{0}^\pm\otimes \factorWedge_{-1}^\mp)
	\pm (1-q^{-2}) \factorWedge_{0}^\pm\otimes \factorWedge_{-1}^\mp)	\\
	=& \mp  H_2(1+(q\inv-q)H_1)H_2(\factorWedge_{0}^\pm\otimes \factorWedge_{-1}^\mp) 
		\pm (q^{-1}-q) H_2H_1H_2(\factorWedge_{0}^\pm\otimes \factorWedge_{-1}^\mp)\\
	& \mp q (1-q^{-2})H_2(\factorWedge_{0}^\pm\otimes \factorWedge_{-1}^\mp)
	\\
	=&\mp \factorWedge_{0}^\pm\otimes \factorWedge_{-1}^\mp
		\mp (q\inv - q) H_2(\factorWedge_{0}^\pm\otimes \factorWedge_{-1}^\mp)
		\mp (q\inv -q) H_2H_1H_2(\factorWedge_{0}^\pm\otimes \factorWedge_{-1}^\mp)\\
		&  \pm (q^{-1}-q) H_2H_1H_2(\factorWedge_{0}^\pm\otimes \factorWedge_{-1}^\mp) \mp q (1-q^{-2})H_2(\factorWedge_{0}^\pm\otimes \factorWedge_{-1}^\mp)
	= \mp \factorWedge_{0}^\pm\otimes \factorWedge_{-1}^\mp.
\end{align*}
Finally by \eqref{H eta}	we get $J_4(\factorWedge_{0}^\pm\otimes \factorWedge_{-1}^\mp)=q^2 J_3((\factorWedge_{0}^\pm\otimes \factorWedge_{-1}^\mp) ) =\mp q^2 \factorWedge_{0}^\pm\otimes \factorWedge_{-1}^\mp  $, as desired.
			\end{proof}
						
	\begin{proof}[Proof of \Cref{funnier vectors trivial rep}]
				By induction on $d$, we may assume that the theorem holds for bitableaux of less than $2d$ boxes with the base cases given by \Cref{exemptyempty}, \Cref{J2action} and \cref{base case funnier vectors}. Thus we can write  $\funnierVector (S,T) = \funnierVector(S',T') \otimes \factorWedge^\pm_n$ for some $n\in \Z$ in the notation of  \cref{funny wedges}.

By induction, the statement for $\kCoideal$ is obvious and $\bo\funnierVector(S',T')=[m\mp1]\funnierVector(S',T')$ holds. Applying \Cref{funny vectors} to both summands of $\eta^\pm_n$ gives the claim for $\bo$ and proves the first part.

For the second part we consider the comultiplication formulas 
\begin{eqnarray}
					\Delta (\eCoideal) & = & \eCoideal \otimes K_{1}\inv + 1\otimes \eCoideal + (\kCoideal-1 ) \otimes E_{-1}K_{1}\inv\nonumber% \label{comult Ecoideal}
					,\\
					\Delta (\fCoideal) &= & \fCoideal \otimes K_1\inv + 1\otimes \fCoideal + (\kCoideal^{-1}-1 ) \otimes E_{1}K_{-1}\inv
					\label{comult Fcoideal}
					,\\
					\Delta(\bo) &=& \bo \otimes K_0\inv + 1\otimes \bo\nonumber%\label{comult B0}
					.
				\end{eqnarray}
We do induction on the number of columns assuming $B_\pm\funnierVector(S',T')=0$. A short calculation gives that $B_\pm$ and $E_{1}K_{-1}\inv$ $E_{-1}K_1\inv$ kill $\eta^\pm_n$ and thus $\eCoideal\,\funnierVector(S,T)= \fCoideal\,\funnierVector(S,T)=0$ by \eqref{comult Fcoideal}. Since $\eNewCoideal = \bo\eCoideal-q\inv\eCoideal\bo$ and $\funnierVector(S,T)$ are $\bo$-eigenvectors  we see that $\eNewCoideal,$ and hence $E_\pm$, kills $\funnierVector(S,T)$.
				
To prove the last assertion assume that $J_1,\ldots, J_{2d-2}$ satisfy the assertion. Abbreviate $\funnierVector_{d}:=\funnierVector(S,T)$ and let $\funnierVector_{d-1},  \funnierVector_{d-2}, \ldots$  be the vectors associated with the  parity standard bitableaux obtained from $(S,T)$ by successively removing  the columns with the largest entries. 
The last factors of $\funnierVector_{d}$ are of one of the following forms:
\begin{equation}\label{lastfactors}
\factorWedge^+_{n-1} \otimes \factorWedge_{n}^+, \quad\factorWedge^-_{n-1} \otimes \factorWedge_{n}^-,\quad \factorWedge^+_{-n-1} \otimes \factorWedge_{n}^-,\quad \factorWedge^-_{1-n} \otimes \factorWedge_{n}^+ 
\end{equation}
In each of the cases, it suffices to show that 
\begin{equation}\label{toshow}
J_{2d-1} \funnierVector_d= q^{-2\content(2d-1)}  \funnierVector_{d}
\end{equation}
since we get with the definition of $J_{2d}$
\[
					J_{2d}\funnierVector_{d} = H_{2d-1} J_{2d-1} H_{2d-1} \funnierVector_{d} = -q  H_{2d-1} J_{2d-1} \funnierVector_{d} = q^{-2\content(2d-3)}  \funnierVector_{d}
					=			q^{2-2\content(2d-1)} \funnierVector_{d},
				\]
which is the claim from the theorem.  

Consider the first case from \eqref{lastfactors}. We have $\content(2d-1) = \content(2d-3)+1$, since both columns were removed in the same component. To show \eqref{toshow} we calculate
				\begin{eqnarray}
					J_{2d-1}\funnierVector_{d} &=&
					H_{2d-2} H_{2d-3} J_{2d-3} H_{2d-3} H_{2d-2} (\funnierVector_{d-2}\otimes \factorWedge^+_{n-1}\otimes \factorWedge^+_n)
					\nonumber
					\\
					&= &
					H_{2d-2} H_{2d-3} J_{2d-3}(q\inv H_{2d-2} - q^{-2}) (\funnierVector_{d-2}\otimes \factorWedge^+_{n-1}\otimes \factorWedge^+_n)
					\nonumber
					\\
					&= &
					H_{2d-2} H_{2d-3} (q\inv H_{2d-2} - q^{-2})J_{2d-3} (\funnierVector_{d-2}\otimes \factorWedge^+_{n-1}\otimes \factorWedge^+_n)
					\nonumber
					\\
					&=&q^{-2\content(2d-3)} H_{2d-2} H_{2d-3}(q\inv H_{2d-2} - q^{-2}) (\funnierVector_{d-2}\otimes \factorWedge^+_{n-1}\otimes \factorWedge^+_n)
					\nonumber
					\\
					&= & q^{-2\content(2d-3)} (q\inv H_{2d-3}H_{2d-2} H_{2d-3}- q^{-2}H_{2d-2} H_{2d-3}) \funnierVector_{d-2}\otimes \factorWedge^+_{n-1}\otimes \factorWedge^+_n
					\nonumber
					\\
					&= & q^{-2\content(2d-3)} (q\inv (-q)(q\inv H_{2d-2} - q^{-2})- q^{-2}(-q)H_{2d-2}) \funnierVector_{d}
					\nonumber
					\\
					&= & q^{-2\content(2d-3)}  q^{-2} \funnierVector_{d}. \label{J2d-1 eigenval induction}
				\end{eqnarray}
using in order the definition,  \cref{H12 eta n eta n+1}, the definition of $J_{2d-3}$, the induction hypothesis, the braid relation, and x\Cref{ridiculous relations} with the definition of $\funnierVector_{d}$. This shows \eqref{toshow} and establishes the first case in \eqref{lastfactors}.	
The second case from  \eqref{lastfactors} can be  treated analogously. 

Therefore we consider the third case, that means $\factorWedge^+_{n} \otimes \factorWedge_{-n-1}^-$. Then the biggest two numbers $2d-1,2d$ in $(S,T)$ are in $T$, and $2d-3,2d-2$ are in the last column of $S$. Let $s,t$ be the number of columns in $S,T$ respectively. We first claim that
\begin{equation}\label{claiminoffice}
\begin{gathered}
J_{2d-1}J_{2d} \funnierVector (S,T) = q^{-2(t-2)-2(t-1)}\funnierVector (S,T), \\ (J_{2d-1}+J_{2d}) \funnierVector (S,T) = -\left(q^{-2(t-2)}+q^{-2(t-1)}\right)\funnierVector (S,T),
\end{gathered}
\end{equation}
If $d\leq 2$ we are done by \cref{ridiculous relations}. So assume $d\geq 3$. 
Let $\underline\epsilon=(\epsilon_1,\ldots, \epsilon_d)$ be the sequence of signs $\pm$ appearing in the $\eta$'s in the definition of  $\funnierVector (S,T)$, i.e. the signs along the path in  \cref{path}.  Since $d\geq 3$, there exists a permutation $\underline\epsilon'=(\epsilon_1',\ldots,\epsilon'_{d})$  of  $\underline\epsilon$ such that  $\epsilon'_{d-1}=\epsilon'_d$.  This new sequence defines a parity standard tableau $(R,U)$ of the same shape as $(S,T)$. The corresponding $\funnierVector (R,U)$ falls by construction into the first case in  \eqref{lastfactors} and thus \Cref{funnier vectors trivial rep} applies to  $\funnierVector (R,U)$.

\noindent On the other hand, using the $h_m$ from  \cref{ridiculous relations}, we can find $h \in \heckeSpecialized{2d }$ such that $\funnierVector (S,T) = h \funnierVector(\epsilon_1'\ldots \epsilon_{d-2}'\epsilon\epsilon).$ Since $J:=J_1\cdots J_{2d-1}J_{2d}$ is central in $\heckeSpecialized{2d}$, [Mat04, Thm. 3.4], we have 
\begin{eqnarray*}
J \funnierVector (S,T)  = J h \funnierVector(R,U) = h J \funnierVector(R,U) =  q^{-2c} h \funnierVector(R,U) =  q^{-2c} \funnierVector (S,T),
\end{eqnarray*}
where $c$ is the sum of contents (which only depends on the shape). Since, by induction,  $\funnierVector (S',T')$ is an eigenvector for $J_1\cdots J_{2d-3}J_{2d-2}$, so is, by definition, $\funnierVector (S,T)$. Moreover the eigenvalues agree and are, again by induction, equal to the sum $c'$ of contents of $(S',T')$.  Since $(S',T')$ arose from $(S,T)$ by removing the last column in $T$, the first claim in \eqref{claiminoffice} follows. The second is proven analogously.  This also shows the formula for the action of $J$. 

 We still need to refine \eqref{claiminoffice} to an action of $J_i$ for $i=2d, 2d-1$. For this consider the $\heckeSpecialized{2d}$-submodule generated by $\funnierVector(S,T)$ which is by \Cref{quantum Schur Weyl}, and \Cref{Cor:ss}  a direct sum of irreducible $\coideal$-modules. In the notation of \cref{hecke algebra section}  we have 
 \[
					\heckeSpecialized{2d} \funnierVector(S,T) = \bigoplus _{(\lambda,\mu)}\specht_q({\lambda,\mu})^{\oplus m_{(\lambda,\mu)}} 
					\subset \vecrep^{\otimes 2d}.
				\]	
				By the induction hypothesis, $J_i \funnierVector (S,T) = \pm q^{-2\content(i)} \funnierVector (S,T) $ for $1\leq i \leq 2d-2$. In view of \cref{Jucys--Murphy spectrum}, this means $\funnierVector(S',T')$ lies in the isotypical component $I$ of $\specht_q({\operatorname{shape}(S'),\operatorname{shape}(T')})$.  By definition,  $\funnierVector (S,T)$  is then contained in $I\otimes \vecrep^{\otimes 2}$. Now recall that tensoring with $\vecrep$ corresponds to induction in the Hecke algebra representations and induction is combinatorially encoded by the Pieri rule of adding a box in all possible ways.  We need to figure out all possible ways to add boxes such that $J_{2d-1}+J_{2d}$ act by the same scalar as in \eqref{claiminoffice}.  The appearing sign forces us to put both boxes into the second component, and, moreover,  the boxes must have exactly contents $t-1$ and $t-2$.  Since our shape of $T'$ is rectangular this means we add a column to $T'$. Thus,  $\funnierVector (S,T)$ must be  a maximal vector in $\specht_q(\op{shape}(S),\op{shape}(T))$.
\end{proof}

			\begin{remark}\label{eta n has n as B weight}
The graph from \cref{path} might help to keep track of the $\bo$-weights. Namely \cref{funnier vectors trivial rep} shows that  in the definition of $\funnierVector(S,T)$, each  factor $\factorWedge^\pm_n$ comes with a parameter $n$ which is chosen such that  $\pm[n]$ equals the $\bo$-weight of the preceding step in the path for  $(S,T)$. 
 \end{remark}
 \begin{remark}
 Since  $\eCoideal ,\fCoideal $ kill $\funnierVector(\lambda,\mu)$ and  $\kCoideal,\bo$ act on  $\funnierVector(\lambda,\mu)$  by a scalar, we constructed  a $1$-dimensional representation of $\coideal$ inside  in the tensor power spanned by $\funnierVector(\lambda,\mu) $. The classical analogue of this is 
 a tensor product of  $\bigwedge^2M_{\pm}$ where $M_+\oplus M_- = \mathbb{V} $ is the decomposition of the vector representation into irreducible $\classicalFixPointSubalg$-modules.
			\end{remark}
					
\begin{remark} Modules for KLR-algebras, \cite{KLohneR}, \cite{Ro-2-kac-moody,},  can be seen as modules for Hecke algebras with a decomposition into eigenspaces for Jucys--Murphy elements. In light of \Cref{funnier vectors trivial rep} and \Cref{polynomial irrep} below it  would be interesting to make a connection between our weight vectors bases and a KLR-type presentation, as e.g. in \cite{Rostam}, of the Hecke algebra $\heckeSpecialized{2d}$.  The element $h_m$ in \cref{ridiculous relations} should be seen as an analogue of the crossing in the KLR-algebras. 
\end{remark}

\subsection{Representative maximal weight vectors}
The Schur--Weyl duality,  \cref{quantum Schur Weyl},  implies, by the double centralizer property, a decomposition of $\vecrep^{\otimes d }$ into a direct sum of irreducible $(\coideal,\,\heckeSpecialized{d})$-bimodules $L\otimes \specht_q$, where $L$ is an irreducible $\coideal$-module and $\specht_q$ is an irreducible  $\heckeSpecialized{d}$-module. In particular, we have a natural bijection between irreducible representations of $\coideal$ and of $\heckeSpecialized{d}$ appearing in $\vecrep^{\otimes d }$. 
Since in the classical situation the irreducible representations are labelled by two-row bipartitions, \cite{MaSt-complex-reflection-groups}, this set should also label the irreducible representations here. 
In this section we prove this using the following strategy.
The multiplicity spaces of the irreducible summands are irreducible $\heckeSpecialized{d}$-modules. By \cref{Jucys--Murphy spectrum},  the irreducible $\heckeSpecialized{d}$-modules can be separated by the eigenvalues of the Jucys--Murphy elements $J_i$.  
In \Cref{funnier vectors trivial rep}, we constructed maximal vectors in $\vecrep^{\otimes d}$ which correspond to rectangular 2-row bipartitions.
We now extend this set to a representative set labelled by $2$-row bipartitions,
where we call a set of maximal vectors in  $\vecrep^{\otimes d}$  \emph{representative} if it contains exactly one vector from each isotypical  component for $\coideal$. In view of  \cref{quantum Schur Weyl}, such a set allows in principle to construct  a weight basis  for each component by applying $F_\pm$'s and the $\heckeSpecialized{d}$-action.  	
	
Wie start by extending  the notion of parity standard bitableaux to arbitrary $2$-row shapes.	
\begin{definition}\label{all bipartition super funny vecs}
			Given a 2-row bipartition $(\lambda,\mu)$,  its unique \emph{special filling} is given by taking the special filling, denoted  by $(S_\lambda, T_\mu)$, in the sense of \Cref{specialfilling} for the rectangular subbipartition $\left((\lambda_2,\lambda_2),  (\mu_2,\mu_2)\right)$ and inserting the remaining numbers increasingly from left to right. 
\end{definition}			
Recalling 	 \eqref{factor def}  we now extend now the definition of  $\funnierVector(S, T)$. 
\begin{definition}				
To each  2-row bipartition $(\lambda,\mu)$ of $d$ we assign a vector $\superFunnyVector(\lambda,\mu)\in\vecrep^{\otimes d}$ by setting 
				\begin{equation}\label{DefOmega}
				\superFunnyVector(\lambda,\mu)=	\funnierVector(S_\lambda, T_\mu)\otimes \factorVone^+_{r} \otimes\factorVone^+_{r+1} \otimes\cdots  \otimes\factorVone^+_{r+s} \otimes\factorVone^-_{-r-s} \otimes\factorVone^-_{-r-s+1} \otimes\cdots \otimes \factorVone^-_{-r-s+t-1},
				\end{equation}
				where $s = \lambda_1 - \lambda_2, t= \mu_1- \mu_2 , r=\lambda_2-\mu_2$.
			
\end{definition}
\begin{eg} For the bipartition $(\lambda,\mu) = (\youngcenter{4,2},\youngcenter{3,1})$, the special filling and vector  $\superFunnyVector(\lambda,\mu)$ are 
\small
\begin{equation*}
				\left(\ytabcenter{
					1 & 3 & 7 & 8\\
					2 & 4
					}
				\, ,\,
				\ytabcenter{
					5 & 9 & 10\\
					6
					}\right),
					\quad\quad\superFunnyVector(\lambda,\mu) = \funnierVector\left(\ytabcenter{
					1 & 3 \\
					2 & 4
					}
				\, ,\,
				\ytabcenter{
					5 \\
					6
					}\right)
					\factorVone_{1}^+ \factorVone_{2}^+\factorVone^-_{-3}\factorVone^-_{-2}
					=( \factorWedge_{0}^+ \factorWedge_{1}^+ \factorWedge_{-2}^- )\factorVone_{1}^+ \factorVone_{2}^+\factorVone^-_{-3}\factorVone^-_{-2}.
					.
			\end{equation*}

\end{eg}	

	\begin{thm}[Representative maximal vectors]\label{funnier vectors all bipartition}
			 The following holds. 
			\begin{enumerate}
			\item  The vector  $\superFunnyVector(\lambda,\mu)$ associated to a 2-row bipartition $(\lambda,\mu)$ is a maximal vector. 
			\item  It is an eigenvector   
			of $\bo$ and $\kCoideal$ with eigenvalues  $[\lambda_1-\mu_1],\,q^{\lambda_1+\mu_1}$ respectively.
				\item It is a common eigenvector for the Jucys--Murphy elements. The eigenvalue of  $J_i$ is $\pm q^{-2\content(i)}$ where $\content(i)$ is the content of the $i$th box in the special filling of $(\lambda,\mu)$. 
				\end{enumerate}
			\end{thm}
			\begin{proof}  To see that $\superFunnyVector(\lambda,\mu)$ is a maximal vector, set, recalling 	 \eqref{factor def}, 
			\begin{equation}
	\label{funny vector abbreviated}
						z= \factorVone^+_{r} \otimes\factorVone^+_{r+1} \otimes\cdots  \otimes\factorVone^+_{r+s} \otimes\factorVone^-_{-r-s} \otimes\factorVone^-_{-r-s+1} \otimes\cdots  \otimes\factorVone^-_{-r-s+t-1}. 
				\end{equation}		
Using  the comultiplication formula \eqref{comult Fcoideal} we calculate that  					
				\begin{eqnarray*}
					\eCoideal \superFunnyVector(\lambda,\mu) &=& \eCoideal \funnierVector(S_\lambda, T_\mu) \otimes (uz) + \kCoideal^{\pm1} \funnierVector(S',T') \otimes  (\eCoideal z) + (1-\kCoideal^{\pm1}) \funnierVector(S',T') \otimes(E_{-1}K_1\inv z)\\
					&=& \funnierVector(S',T') \otimes(\eCoideal z),
				\end{eqnarray*} 
where for  the second equality  we applied \cref{funnier vectors trivial rep} to $ \funnierVector(S',T')$. Next we claim that $\eCoideal z = 0$. For this we use induction on the number of tensor factors. Clearly we have $\eCoideal \factorVone^\pm _n =0 $ for all $n\in \Z$. For the induction step we observe that the $E_{-1}$  appearing  the  comultiplication formula \eqref{comult Fcoideal}  kills all $\factorVone^\pm _n$ for $n\in \Z$. This proves the claim and $\superFunnyVector(\lambda,\mu)$ is a maximal vector.
				
For the statements in part b.) we use \Cref{funnier vectors trivial rep} to get that $\funnierVector(S_\lambda,T_
\mu)$ is an eigenvector for $\kCoideal$ and $\bo$ with eigenvalues $1$ and $[\lambda_2-\mu_2]$ respectively. Then by construction and \cref{funny vectors}, the $\bo$-eigenvalue of $\superFunnyVector(\lambda,\mu) $ is $[\lambda_1-\mu_1]$. Since $\kCoideal$ is grouplike, its eigenvalue on $\superFunnyVector(\lambda,\mu) $ is the product of the eigenvalues on the factors, which gives the claim.
				
To calculate the spectrum of the Jucys--Murphy elements, we use induction on $\lambda_1-\lambda_2 +\mu_1-\mu_2$, the base case being $\funnierVector(S_\lambda,T_\mu)\otimes \factorVone^\pm_{m}$, since the case $\funnierVector(S_\lambda,T_\mu)$ is contained in \cref{funnier vectors trivial rep}.
Denote by $d$ the total number of boxes in $(\lambda,\mu)$.  
The eigenvalues for $J_1,\ldots, J_{d-1}$ follow from \cref{funnier vectors trivial rep}. We first consider the case where the numbers $d-2,d-1,d$ are in adjacent columns in the special filling of $(\lambda,\mu)$ and then reduce the general base case to this one.

By definition, the assumption on the entries $d-2,d-1,d$ implies that  $\superFunnyVector(\lambda,\mu) = \funnierVector(S',T')\otimes \factorWedge^\pm_n\otimes \factorVone^\pm_{n+1}$, where $(S',T')$ is the parity standard bitableau obtained from $(S_\lambda,T_\mu)$ by removing the column containing $d-2,d-1$. Denoting by $u$ the $J_{d-2}$-eigenvalue of $\superFunnyVector(\lambda,\mu)$,  we calculate \begin{align*}
	J_{d}\superFunnyVector(\lambda,\mu) &= H_{d-1} J_{d-1} H_{d-1} (\funnierVector(S',T')\otimes \factorWedge^\pm_n\otimes \factorVone^\pm_{n+1}) \\
	&= H_{d-1} H_{d-2} J_{d-2} H_{d-2} H_{d-1} (\funnierVector(S',T')\otimes \factorWedge^\pm_n\otimes \factorVone^\pm_{n+1}) \\
	&= H_{d-1} H_{d-2} J_{d-2} (q\inv H_{d-1} - q^{-2}) (\funnierVector(S',T')\otimes \factorWedge^\pm_n\otimes \factorVone^\pm_{n+1}) \\
	&= H_{d-1} H_{d-2} (q\inv H_{d-1} - q^{-2}) J_{d-2} (\funnierVector(S',T')\otimes \factorWedge^\pm_n\otimes \factorVone^\pm_{n+1}) \\
	&= H_{d-1} H_{d-2} (q\inv H_{d-1} - q^{-2}) u (\funnierVector(S',T')\otimes \factorWedge^\pm_n\otimes \factorVone^\pm_{n+1}) \\
	&= u (q\inv H_{d-2}H_{d-1} H_{d-2}- q^{-2}H_{d-1} H_{d-2}) \funnierVector(S',T')\otimes \factorWedge^\pm_n\otimes \factorVone^\pm_{n+1} \\
	&= u (q\inv (-q)(q\inv H_{d-1} - q^{-2})- q^{-2}(-q)H_{d-1}) \funnierVector(S',T')\otimes \factorWedge^\pm_n\otimes \factorVone^\pm_{n+1} = u q^{-2} \superFunnyVector(\lambda,\mu),
\end{align*}
where we used successively the definition of $J_d$, the definition of $J_{d-1}$, \Cref{H2 eta n x n+1}, the definition of $J_{d-2}$, the induction hypothesis, the braid relation, \Cref{ridiculous relations} with the definition of $\superFunnyVector(\lambda,\mu)$.  This shows that the eigenvalue of $J_d$ is $u q^{-2}$.
Since $\content(d) = \content(d-2)+1$ in this case, the claim follows.

Now consider the general base case. By \cref{ridiculous relations} there is an element $h\in \heckeSpecialized{d-1}$ such that $\funnierVector(S_\lambda,T_\lambda)= h \funnierVector(S',T')$ where $(S',T')$ is a parity standard bitableau of the same shape as $(S_\lambda,T_\mu)$, but $d-2,d-1$ lie in the column adjacent to the box containing $d$ in the special filling of $(\lambda,\mu)$. Then $\funnierVector(S',T')\otimes \factorVone^+_{n+1}$ has the desired eigenvalues for $J_1,\ldots, J_{d-1}$ by the above calculation. Since $J:=J_1+\cdots+ J_{d}$ is central in $\heckeSpecialized{d}$, we have \[
	J \superFunnyVector(\lambda,\mu) = J h \superFunnyVector(\lambda,\mu) = h J \superFunnyVector(\lambda,\mu) =  \left(\sum_{i=1}^{d-1} \pm q^{-2\content(i)}\right) \superFunnyVector(\lambda,\mu),
\]
which implies $J_d \superFunnyVector(\lambda,\mu)  = \pm q^{-2\content(d)}\superFunnyVector(\lambda,\mu)$. This establishes the general base case.

Now for an arbitrary bipartition $(\lambda,\mu)$,
let $(\lambda',\mu')$ be the bipartition obtained from $(\lambda,\mu)$ by removing the boxes containing $d$ in the special filling. Then by the induction hypothesis, the theorem holds for the $J_1,\ldots , J_{d-1}$ acting on $\superFunnyVector(\lambda',\mu')$, and then also for their action on $\superFunnyVector(\lambda,\mu) = \superFunnyVector(\lambda',\mu') \otimes \factorVone^\pm_n$. 

It remains to determine the eigenvalue of $J_d$. Note that we may assume $\lambda_1-\lambda_2+\mu_1-\mu_2\geq 2$, since otherwise we are in the general base case treated above.
Now we have two cases. If the box containing $d-1$ and $d$ are in the same component in the special filling of $(\lambda,\mu)$, then, by \eqref{H xn xn+1},  \[
	J_d \superFunnyVector(\lambda,\mu) = H_{d-1}J_{d-1}H_{d-1} \superFunnyVector(\lambda,\mu) =\pm q^{-2\content(d-1)-2}  \superFunnyVector(\lambda,\mu).
\]
In this case $\content(d) = \content(d-1)+1$ and the claim follows.
If the boxes containing $d-1$ and $d$ are in different components, then
$\superFunnyVector(\lambda',\mu') \otimes \factorVone^\pm_n\otimes \factorVone^\mp_{-n-1}$ for some bipartition $(\lambda',\mu')$. We calculate the $\factorVone^+_{n} \otimes \factorVone^-_{-n-1} $ case, the other one is completely analogous. Suppose in this case $\content(d) = s$ and $\content(d-1) = t$ in the special filling of $(\lambda,\mu)$, so that $s-t=n$.
Writing $\alpha_{n}=q^n+q^{-n}$, we compute 
\begin{align*}
	J_d \superFunnyVector(\lambda,\mu) &= H_{d-1} J_{d-1} H_{d-1} (\superFunnyVector(\lambda',\mu') \otimes \factorVone^-_{-n} \otimes \factorVone^+_{n-1}) \\
	& = H_{d-1} J_{d-1} \alpha_n\inv
		\left(\superFunnyVector(\lambda',\mu') \otimes ( 
			\alpha_{n-1}
		 \factorVone^+_{n} \otimes \factorVone^-_{-n-1}  
					+
					q^{-n}(q\inv-q) \factorVone^-_{-n} \otimes \factorVone^+_{n-1} )\right)
					\\
	& =\alpha_n\inv
		H_{d-1}\left(\superFunnyVector(\lambda',\mu') \otimes ( 
			q^{-2s}  \alpha_{n-1}
		 \factorVone^+_{n} \otimes \factorVone^-_{-n-1}  
					-
					q^{-2t}  q^{-n}(q\inv-q) \factorVone^-_{-n} \otimes \factorVone^+_{n-1} )\right)
					\\
	&= q^{-2s} \alpha_n^{-2} \superFunnyVector(\lambda',\mu') \otimes (
			\alpha_{n-1}  (q^{n}(q\inv-q) \factorVone^+_{n} \otimes \factorVone^-_{-n-1}  
			+
			\alpha_{n+1} \factorVone^-_{-n} \otimes \factorVone^+_{n-1} )
					\\
	& \quad - q^{2n} q^{-n}(q\inv-q) (
		 \alpha_{n-1} \factorVone^+_{n} \otimes \factorVone^-_{-n-1}  
					+
					q^{-n}(q\inv-q) \factorVone^-_{-n} \otimes \factorVone^+_{n-1} )
	)
					\\
	&= q^{-2s} \alpha_n^{-2} \superFunnyVector(\lambda',\mu') \otimes (
		(q^n (q\inv-q) \alpha_{n-1} - q^{n}(q\inv-q) \alpha_{n-1} ) \factorVone^+_{n} \otimes \factorVone^-_{-n-1}  
					\\
	& \quad + (\alpha_{n+1} \alpha_{n-1} - (q\inv-q)^2 ) \factorVone^-_{-n} \otimes \factorVone^+_{n-1} 
	)
				 = q^{-2s} \alpha_n^{-2} \alpha_n^2 \superFunnyVector(\lambda,\mu),
\end{align*}
 where we used successively the definition of $J_d$, \cref{H xplus xminus}, the induction hypothesis, $s-t=n$, \Cref{H xplus xminus} again and collected terms. Since $\content(d)=s$, this proves the claim.
			\end{proof}

We finally  show, among others,  that the $\superFunnyVector(\lambda,\mu)$ form a representative set of maximal vectors.

			\newcommand{\eigenvalW}{\omega}
			\newcommand{\Kweight}{\kappa'}
			\newcommand{\wtseqB}{\mathfrak{b}}
			\newcommand{\wtseqK}{\kappa}
			
			\newcommand{\coeff}{\xi}
			\begin{thm}[Decomposition of tensor powers]\label{polynomial irrep}The following hold.
			\begin{enumerate}
				\item As a $(\coideal,\,\heckeSpecialized{d})$-bimodule $\vecrep^{\otimes d }$ decomposes into irreducibles bimodules
						\begin{equation}\label{bimod decomp}
							\vecrepClassical^{\otimes d} = \bigoplus_{(\lambda,\mu)} L(\lambda ,\mu)  \boxtimes \specht_q({\lambda,\mu}),
						\end{equation}
				
where the direct sum is over bipartitions $(\lambda,\mu)$ of $d$ with at most two rows. 

Here $\specht_q({\lambda,\mu})$ is the irreducible $\heckeSpecialized{d}$-module associated to the bipartition $(\lambda,\mu)$.
					\item For any two-row bipartition $(\lambda,\mu)$, the multiplicity space  $L(\lambda,\mu)$ is an irreducible $\coideal$-module, and the maximal vector $\superFunnyVector(\lambda,\mu)$ from \eqref{DefOmega} is contained in a copy of $L(\lambda,\mu) $. 
					\item The $(\dCoideal_\half,\dCoideal_1,\bo,Z,W )$-weight $(q^{\kappa_\half},q^{\kappa_1},[n],\zeta,\omega ) $ of~ $\superFunnyVector(\lambda,\mu)$, with $\kappa = \kappa_\half-\kappa_1$, equals
					\begin{gather}
						\begin{aligned}
						\kappa_\half = \lambda_1+\mu_1 
						, \qquad
						\kappa_1 = \lambda_2+\mu_2
						,\qquad
						n = \lambda_1-\mu_1
						,\\
						\zeta = [\lambda_2-\mu_2] - q^{-\kappa} [n]
						,
						\quad
						\omega = q^{-2} [\lambda_2-\mu_2]  -q^{\kappa-2} [n].
					\end{aligned}
						\label{parameters}
					\end{gather}
					
In particular, the bipartition $(\lambda,\mu)$ uniquely determines the isomorphism class of $L(\lambda,\mu)$.

					\item The bijection \cref{bijection irrep rational to partitions} induces a bijection
					\begin{equation}\label{isoclasses}
						\left\{ 
				\begin{minipage}[c]{4.5cm}
				\begin{center}
				isomorphism classes of\\ irreducible 
				 summands in $\vecrep^{\otimes d}$
				 \end{center}
				\end{minipage}
				\right\}\;\;
				\cong
				\{ \text{ 2-row bipartitions of $d$} \}.
				\end{equation}
				\end{enumerate}
			\end{thm}

We observe that the $\omega$-values in \cref{parameters} satisfy the following recursion formula.
\begin{lemma}
	Let $(\lambda,\mu)$ as in \cref{polynomial irrep}. Assume further that either $\lambda_1-\lambda_2\geq 1$ or $\mu_1-\mu_2\geq 1$
	and let $(\lambda',\mu')$ be the bipartition obtained by removing the box $b$ containing $d$ in the special filling. 
	Let $\omega,\omega'$ be the associated $\omega$-values from \cref{parameters}.  Then \begin{equation}
		\omega = \omega' \mp q^{\kappa'-1\pm m},\quad \quad\quad\text{where $\kappa'=\lambda_1+\mu_1-\lambda_2-\mu_2-1, m=\lambda_1-\mu_1\mp 1$.}
					\label{omega recursion}
	\end{equation}
 Here, the sign is $-$ if $b$ is in the first component and $+$ otherwise.
\end{lemma}

			\begin{proof}[Proof of \cref{polynomial irrep}]  
			 \cref{quantum Schur Weyl}  and  the double centralizer theorem imply  a multiplicity free bimodule decomposition, but we still need to verify the labelling set in  \cref{bimod decomp}. 

Given a $2$-row bipartition $(\lambda ,\mu )$, there is  by \cref{funnier vectors all bipartition} a standard bitableau $(S,T)$ of shape $(\lambda ,\mu) $ with a maximal vector $\superFunnyVector(\lambda,\mu)$ having the asserted weights for $\bo,\kCoideal$.  Looking at the action of the Jucys--Murphy elements described by the last part of \cref{funnier vectors all bipartition}, we see that $\superFunnyVector(\lambda,\mu)$ lies in the bimodule $L(\lambda,\mu )\boxtimes \specht_q(\lambda ,\mu )$. This proves the second part of the theorem.

From now on denote $\kappa=\lambda_1+\mu_1-\lambda_2-\mu_2$. 
We next show by induction that each $\superFunnyVector(\lambda,\mu)$ is a $W$-eigenvector with eigenvalue satisfying the recursion formula \cref{omega recursion}. Expanding  \Cref{DefsXYZW}, \begin{align}
	W & = \fCoideal\bo\eCoideal -q\inv \fCoideal\eCoideal\bo - q\inv \bo\eCoideal\fCoideal + q^{-2}\eCoideal\bo\fCoideal
	\nonumber
	\\
	& = \fCoideal\bo\eCoideal -q\inv \fCoideal\eCoideal\bo - q\inv \bo\fCoideal\eCoideal -q\inv \bo [\kCoideal;0] +q^{-2} \eCoideal\bo\fCoideal.
	\label{expand W}
\end{align}
Since the first three terms kill $\superFunnyVector(\lambda,\mu)$ and the third acts by a scalar we know, 
we just need to calculate the action of $q^{-2} \eCoideal\bo\fCoideal$. We do this via induction on $\lambda_1-
				\lambda_2$ and $\mu_1-\mu_2$. We start from $(\lambda,\mu)=((s,s),(t,t)) $. By \cref{funnier vectors trivial rep}, $\eCoideal\bo\fCoideal \superFunnyVector(\lambda,\mu) =0 $, and $-q\inv \bo [\kCoideal;0] \superFunnyVector(\lambda,\mu) = 0$. Since in this case $\omega=
					q^{-2}[s-t] - q^{\kappa-2}[s-t] = (q^{-2} - q^{-2})[s-t]	=0,
				$
				the base case is proved.

				Assume now we can obtain a bipartition $(\lambda',\mu')$ by removing one box from the first row of $(\lambda,\mu)$. Then by definition we have $ \superFunnyVector(\lambda,\mu )=\superFunnyVector(\lambda',\mu')\otimes \factorVone^\pm_{m}$. 
				By induction hypothesis, $\superFunnyVector(\lambda',\mu') $ is a weight vector. Let $(\omega', [m],q^{\Kweight})$ be its $(W,\bo,\kCoideal)$-weight. Thus by \cref{expand W}, it is also an eigenvector for $\eCoideal\bo\fCoideal$ of eigenvalue $\xi=  q^2( \omega'+ q\inv [m][\Kweight] ) $.
				Applying the comultiplication, see \eqref{comult Fcoideal}, we obtain
				\small
				\begin{eqnarray*}
					&&\eCoideal \bo\fCoideal \superFunnyVector(\lambda,\mu ) \\
					&=&
					(\eCoideal \bo\fCoideal  
						\otimes 
					(K_{-1} K_0 K_1)\inv
					+
					\eCoideal\fCoideal
						\otimes 
					\bo (K_1 K_{-1})\inv 
					+ 
					(\kCoideal-1)\bo 
						\otimes 
					E_{-1}(K_{1}K_0)\inv \fCoideal \\
					&&
					+
					(\kCoideal\inv-1)\bo 
						\otimes
					\eCoideal K_0\inv E_{1}K_{-1}\inv 
					+
					\bo\otimes \eCoideal K_0\inv \fCoideal
					) \superFunnyVector(\lambda',\mu') \otimes \factorVone^\pm_{m},
					\\
					&=& \xi\, \superFunnyVector(\lambda',\mu') \otimes \factorVone^\pm_{m} + [\Kweight] \superFunnyVector(\lambda',\mu') \otimes (\vecBasis[{\half}]  \pm q^{\pm m}\vecBasis[\negidx{\half}]) 
					+(q^{\Kweight}-1) [m] \superFunnyVector(\lambda',\mu') \otimes (\pm q^{\pm m}\vecBasis[{\half}] )
					\\
					&& + (q^{-\Kweight}-1) [m] \superFunnyVector(\lambda',\mu') \otimes \vecBasis[\negidx{\half}] 
					+ [m] \superFunnyVector(\lambda',\mu') \otimes \factorVone^\pm_{m}
					\\
					&=&
					(\xi+[m]) \superFunnyVector(\lambda',\mu') \otimes \factorVone^\pm_{m}
					+ \superFunnyVector(\lambda',\mu') \otimes
					 [\Kweight] (\vecBasis[{\half}]  \pm q^{\pm m} \vecBasis[\negidx{\half}]) \pm (q^{\Kweight}-1) [m] q^{\pm m}\vecBasis[{\half}] + (q^{-\Kweight}-1) [m] \vecBasis[\negidx{\half}] )
					\\
					&=& (\xi \pm [\Kweight]q^{\pm m} +q^{-\Kweight}[m] )
					\superFunnyVector(\lambda',\mu') \otimes \factorVone^\pm_{m},
				\end{eqnarray*}
				\normalsize
				where we used the following easy identities  					
				\begin{align*}
						\eCoideal K_0\inv E_{1}K_{-1}\inv  \factorVone^\pm_m &= \vecBasis[\negidx{\half}],&
					E_{-1}(K_{1}K_0)\inv \fCoideal  \factorVone^\pm_m &= \pm q^m \vecBasis[{\half}] ,
					\\
					(K_{-1} K_0 K_1)\inv \factorVone^\pm_m &= \factorVone^\pm_m,&
					\bo (K_1 K_{-1})\inv \factorVone^\pm_m &=  \vecBasis[{\half}]  \pm q^m\vecBasis[\negidx{\half}]
					\\
					\eCoideal K_0\inv \fCoideal  \factorVone^\pm_m &= \factorVone^\pm_m,&
					\frac{ [\Kweight] \pm (q^{\Kweight}-1)[m]q^{\pm m}}{\pm [\Kweight]q^{\pm m} + (q^{-\Kweight}-1)[m]} &= \pm q^{\pm m}.
				\end{align*}
				This shows  that $\superFunnyVector(\lambda,\mu)$ is an eigenvector for $\eCoideal\bo\fCoideal$, hence by \cref{expand W} also for $W$, with $W$-eigenvalue 
				\begin{equation}\label{omegavalue}
					\omega = q^{-2} (\xi \pm [\Kweight]q^{\pm m} +q^{-\Kweight}[m] )
					- q^{-1} [m\pm1][\kappa].
				\end{equation}
				Now by induction hypothesis we know the values for $(\lambda',\mu')$, namely 
				$m=\lambda_1-\mu_1\mp 1$,  
					$\Kweight = \kappa-1$.
				Inserting these values into the expression \eqref{omegavalue} we get $
					\omega = \omega' +q^{-1} [m][\kappa-1] + q^{-\kappa-1} [m] - q^{-1} [m\pm1][\kappa] \pm q^{\pm m-2}[\kappa-1].$
				Simplifying gives exactly the recursion formula \cref{omega recursion}. Since both sequences agree on the base case, they agree for all two-row bipartitions.

We return to  \cref{bimod decomp} and show that the sum in  runs over the bipartitions $(\lambda,\mu)$ of $d$ with at most two rows. Since a completely analogous statement holds in the classical situation, it suffices to match the dimensions of the irreducible modules. For the Hecke algebra modules this is well-known. For the $\coideal$-modules, we have, by \cref{fd quot good verma}, $\dim L(\lambda,\mu ) = (\lambda_1-\lambda_2+1)(\mu_1-\mu_2+1)$, which agrees with the dimension of its classical analogue.
Comparing  \cref{bijection irrep rational to partitions} with \eqref{parameters}  directly shows \eqref{isoclasses}. 	 		
			\end{proof}

\begin{rem}
All maximal vectors in the tensor power are obtained from some $\superFunnyVector(\lambda,\mu)$ by applying the Hecke algebra action, since the multiplicity spaces in \cref{bimod decomp} are irreducible $ \heckeSpecialized{d}$-modules. 
\end{rem}				
				
\subsection{Clebsch--Gordan formula} In this section we prove a Clebsch--Gordan type formula for $\coideal$. When trying to establish this, we also found a nice normalisation of irreducible modules corresponding to bipartitions where one component is rectangular, which we record here as well.

			\begin{lem}\label{easy irreps}
				Consider $L=L(\lambda,\mu)$ for a pair of two-row partitions $\lambda, \mu$ such that $\mu_1=\mu_2$. Let $v\in L$ a maximal vector and $\kappa,n$ be as in \cref{parameters}.  Then  $E_-$ and $F_-$ act by zero on $L$. Moreover, $\{ v_a= ([a]! )\inv \fCoideal^a v\}_{0\leq a\leq \kappa}$ forms a weight basis of $L$ satisfying  \[
					E_+v_a= (q^{2n-2a}+1) [\kappa-a+1] v_{a-1},\quad F_+ v_a=  (1+q^{2a-2n})[a+1] v_{a+1}.
				\]
				An analogous statement holds for $L(\mu,\lambda)$. 
			\end{lem}

						\begin{ex}
							Consider $L((4,0)\emptyset)$.
				Note that $q^m+q^{-m}=[m+1]-[m-1]$. Then, in the weight basis $\{ v_0,v_1,v_2,v_3,v_4\}$  from \Cref{easy irreps}, the action of $E_+$ and $F_+$ is explicitly given by: \[
					\begin{tikzpicture}[xscale=2.5]
						\foreach \x in {0,1,...,4}
						{
							\node[draw,circle,inner sep=2pt] (\x) at (\x,0) {$v_{\x}$};
						};
						\foreach \x/\y in {0/1, 1/2,2/3, 3/4}
						{
							\draw[bend left=40,in=120,out=60,<-] (\x) to node[above](E\x){}(\y);
							\draw[bend right=40,in=120,out=60,<-] (\y) to node[below](F\x){}(\x);
							
						};
						\foreach \x/\nPlusOne/\nMinusOne/\expo in {0/4/2/3, 1/3/1/2,2/2/0/1, 3/1/-1/0}
						{
							\node[above] at (E\x) {$q^{\expo}([\nPlusOne]-[\nMinusOne])[\nPlusOne]$};
						};

						\foreach \x/\nPlusOne/\nMinusOne/\fScalar/\expo in {0/5/3/1/4, 1/4/2/2/3,2/3/1/3/2, 3/2/0/4/1}
						{
							\node[below] at (F\x) {$q^{\expo}([\nPlusOne]-[\nMinusOne])[\fScalar]$};
						};

					\end{tikzpicture}
				\]
				Here, $E_+$ and $F_+$ send basis vectors to basis vectors multiplied by the indicated scalars. 
			\end{ex}
			\begin{proof}[Proof of \cref{easy irreps}]
				We prove the statements for $L(\lambda,\mu)$. Since $\mu_1=\mu_2$, by \cref{parameters} the weight basis in \cref{fd quot good verma} is in fact $\Fpm_+^av$ and $\Epm_-v_a=\Fpm_-v_a=0$. Thus $L(\lambda,\mu)$ is killed by $E_-$ and $F_-$. 
				Next, by \cref{E+-F+ B0 weight}, the $\bo$ weight of $v_a $ is $[n-a]$, In view of \cref{EF in E+- F+-}, \begin{gather}
					\begin{aligned}
					\eCoideal v_a & = (q^{n-a}+q^{a-n}) \inv (E_+ v_a + E_-v_a) = q^{a-n}(q^{n-a}+q^{a-n}) \inv  E_+ v_a,\\
					\fCoideal v_a & = (q^{n-a}+q^{a-n}) \inv (F_+ v_a + F_-v_a) = q^{n-a}(q^{n-a}+q^{a-n}) \inv F_+ v_a.
					\end{aligned}
					\label{easy action E+F+}
				\end{gather}
				and similarly for ${\fCoideal}$. Now recall that the subalgebra generated by $\eCoideal,\fCoideal,\kCoideal$ is isomorphic to  $\Uq(\sl_2)$, hence we have $
					\eCoideal v_a = [\kappa - a +1] v_{a-1}$, and 
					$ \fCoideal v_a = [a+1] v_{a+1}.
				$
				Combined with \cref{easy action E+F+}, this concludes the proof.
			\end{proof}

			\newcommand{\newHwVec}{\Xi}
			\newcommand{\auxVectors}{\mathsf{w}}
		
			\begin{thm}[Clebsch--Gordan formula]\label{ClebschGordan}
				For any two-row bipartition $(\lambda,\mu)$, there is a multiplicity free decomposition of $\coideal$-modules \begin{align*}
					 L(\lambda,\mu) \otimes \vecrep = 
					 L((\lambda_1+1,\lambda_2),\mu)
					 \oplus L((\lambda_1,\lambda_2+1),\mu)
					 \oplus L(\lambda,(\mu_1+1,\mu_2))
					 \oplus L(\lambda,(\mu_1,\mu_2+1))
					 .
				\end{align*}
				By convention, $L(\lambda,\mu)=0$ if $(\lambda,\mu)$ is not a bipartition.		
			\end{thm}
			\begin{proof}
		Let $v$ be a highest weight vector of $ L(\lambda,\mu)$. Let $n, \kappa ,\zeta$ be as in \cref{polynomial irrep} and  $i=\lambda_1-\lambda_2$. We construct an explicit maximal vector for each summand in $L(\lambda,\mu) \otimes \vecrep$.  For this let  
		\begin{equation}\label{auxilary}
			\auxVectors_+
			= 
			\Fpm_+v\otimes\factorVone^+_{n-1} - q^{1-n}[i]\alpha_{n+\kappa-i} v\otimes\factorVtwo^+_{n-\kappa-2}
			,\quad 
			\auxVectors_-
			= 
			\Fpm_-v\otimes\factorVone^-_{-n-1} - q^{n+1}[\kappa-i]\alpha_{n-i} v\otimes\factorVtwo^-_{-n-\kappa-2} , 
			%q^(1+n)*q_int(k-i)*alpha(n-i) *v.otimes(ypm(-1,n+k+2))
	  	\end{equation}
		with $\alpha_j=q^j+q^{-j}$ for $j\in \Z$,  and use these vector to define
		\begin{align}
			\newHwVec_+=(q^{2i} - 1) \auxVectors_-
			- (q^{2i} + q^{2n})
			\auxVectors_+
			,\quad
			\newHwVec_-=  (q^{2i-2n} + q^{2\kappa})\auxVectors	_-
			+
			(q^{2i} - q^{2\kappa})\auxVectors_+
			.\label{hw vec in ClebschGordan}
		\end{align}
By a dimension count, using the basis from \Cref{fd quot good verma}, 
it  suffices to show that the following vectors are maximal in the  corresponding summand and zero if the summand is zero.  
\begin{equation}\label{CGvectors}
		\begin{array}[c]{|c|c|c|c|}
		\hline
			\newHwVec_+ 
			, &
			\newHwVec_-
			, &
			v\otimes \factorVone_n^+
			,&
			v\otimes \factorVone_n^-\\
			\hline
				L((\lambda_1,\lambda_2+1),\mu)&
		L(\lambda,(\mu_1,\mu_2+1))&
		L((\lambda_1+1,\lambda_2),\mu)&
		L(\lambda,(\mu_1+1,\mu_2))\\
		\hline
		\end{array}
		\end{equation}

To prove our claim, first note that $((\lambda_1+1,\lambda_2),\mu)$, $(\lambda,(\mu_1+1,\mu_2))$ are always two-row bipartitions.  Assuming that $((\lambda_1,\lambda_2+1),\mu) $ is not a bipartition, we have   $i=\lambda_1-\lambda_2=0$ and, by \cref{easy irreps}, $\Fpm_+v=0$. Then $\auxVectors_+=0$, and we obtain $
			\newHwVec_+  = (q^{0} - 1) \auxVectors_-
			- (q^{0} + q^{2n})
			\auxVectors_+ =0
		$ as desired.	
Similarly, if $(\lambda,(\mu_1,\mu_2+1))$ is not a bipartition, then we have $\mu_1-\mu_2=0$ and $\kappa=i$. This  implies $\Fpm_-v=0$ and then $\auxVectors_-=0$. We get  $
			\newHwVec_- =  (q^{2i-2n} + q^{2i})\auxVectors	_-
			+
			(q^{2i} - q^{2i})\auxVectors_+ =0
		$ as desired.
		
We verify the asserted weights of $\newHwVec_i$ and $ v\otimes \factorVone_n^\pm$, see \eqref{parameters}. The $\dCoideal_j$-weights are clear. 
		Now by \cref{funny vectors},
		$v\otimes \factorVone_n^\pm$ are $\bo$-weight vectors of weights $[n\pm1]$, and each term of $\auxVectors_\pm$ is a $\bo$-weight vectors of weight $[n]$.
		Hence $\newHwVec_\pm$ are $\bo$-weight vectors of weight $[n]$ as desired. 
We postpone the calculation of the $Z$-weights since they become  easier after we have shown the maximality of the vectors.  For this it  suffices to show that the vectors are killed by $\eCoideal$, since then, being $\bo$-weight vectors,  they are also killed by $\eNewCoideal$, see  \cref{serre relations as commutators}. Using \eqref{comult Fcoideal} we obtain 
 \begin{align*}
			&\eCoideal (v\otimes\basis_1)  = q^\kappa v\otimes \basis_{\half}
			,
			&\eCoideal (v\otimes\basis_{\negidx{1}})  = v\otimes \basis_{\negidx{\half}} ,
			\quad
			&\eCoideal (\Fpm_\pm v\otimes \factorVone^{\pm}_{\pm n-1}) = ([\kappa]\mp q^{\mp n}\zeta ) v\otimes q\factorVone^{\pm}_{\pm n-2} ,
		\end{align*}
		which we can plug into \eqref{auxilary} to deduce that $\eCoideal\auxVectors_\pm=0$ and the maximality of  $\newHwVec_+,\newHwVec_-$ follows.
		For $v\otimes \factorVone^\pm_{n}$ it is enough to note that both $v$ and $\factorVone^\pm_{n}$ are killed by $\eCoideal$.
		
		Now using the fact that the claimed vectors are highest weight, a lengthy computation, see \cref{lengthycalculation},  shows that the $Z$-weights of $\newHwVec_+,\newHwVec_-,v\otimes\factorVone^\pm_n$ are $
			[\kappa+n-2i+1]-q^{1-\kappa}[n]
			, \;
			[\kappa+n-2i-1]-q^{1-\kappa}[n]
			$, and $ 
			[\kappa+n-2i]-q^{-1-\kappa}[n\pm 1],
		$
		respectively. Comparing with \cref{parameters}, we are done.		\end{proof}

	\section{Central characters and Harish-Chandra homomorphism}
	We finish by studying the centre of $\coideal$ with some applications.
			\subsection{The centre}
	We explicitly describe the centre $Z(\coideal)$ of $\coideal$ 
		\begin{thm}\label{centre}
			The center  $Z(\coideal)$ of $\coideal$ is isomorphic to $\groundring[\det^{\pm1},C_1,C_2,C_3]$ where		\begin{eqnarray*}			
			&	{\textstyle\det^{\pm1}} = (\dCoideal_\half\dCoideal_1)^{\pm1}	
			,\quad\quad
			C_1 = Z\dCoideal_\half +q\inv[2]\bo\dCoideal_1 
			,\quad \quad
			C_3 =   q^2W\dCoideal_1+q[2]\bo \dCoideal_{\half} 
				,&\\ 
			&C_2 =   \frac{q^{2}\kCoideal+q^{-2}\kCoideal\inv }{(q-q\inv )^{2}}
					-\bo^2 \kCoideal\inv - Z \bo  - Y X 
					+ (q - q\inv) Y \eCoideal \bo + q \fCoideal\eCoideal.
		\end{eqnarray*}	
		\end{thm}

		\begin{proof} We first claim that the $C_i$'s are central. This follows for $C_3$ from $C_1$ by applying  \cref{involution}. 
			By \cref{csa} and \cref{easycomm}, $C_1,C_2$ commute with $\CartanPart$.  Using additionally  \cref{EZ FZ commutator}--\cref{XZ YZ commutator}, a short calculation for  $C_1$ and a longer one for  $C_2$ shows the claim.   The subalgebra $\mathcal{C}$ generated by the $C_i$'s therefore embeds into the centre $Z(\coideal)$ of $\coideal$. It remains to show surjectivity.		
			
			By \cite[Cor.4.15]{Ko-braided-module}, $Z(\coideal)$ is a polynomial algebra over $\groundring[\det^{\pm1}]$ generated by three elements which are the unique (up to scalars) elements in $\adl\Uq(\gl_4) L_i \cap Z(\coideal)$, where \[
				L_1 = D_{\negidx{1}}^{-2} 
					, \quad
				 L_2 = (D_{\negidx{1}}D_{\negidx{\half}})^{-2}
				, 
				\quad L_3 = (D_{\negidx{1}}D_{\negidx{\half}}D_{\half})^{-2}
				\in \Uq(\gl_4),
			\]
			the $D_i$'s are as in \cref{convention: quantum group}. Consider the following renormalized central elements: \[
				C_1'=(q^2-1)\textstyle\det\inv C_1 , \quad
				C_2' =(q^2-1)^2\textstyle\det\inv C_2,\quad
				C_3'=(1-q^{-2})\textstyle\det^{-2}C_3 \in Z(\coideal).
			\] 
			It suffices to show that  $C_i'\in \adl\Uq(\gl_4) L_i$, since then \cite[Cor. 4.15]{Ko-braided-module} implies that the $C_i$'s, and thus the $C_i$'s, generate the centre as algebra. 
			By expansion in the PBW basis of $\Uq(\gl_4)$ we show:
			{\small
				\begin{align*}
					 C_1'  = & \left(-q^{-2}\adl(F_{-1}F_{0}F_{1}) 
						+ q\inv \adl(F_{0}F_{1}E_{1}) 
						-\adl(F_{1}E_{0}E_{1}) 
						+ q \adl(E_{-1}E_{0}E_{1}) \right) (L_1),
					\\
					 C_3'  = & \left(-q^{-2}\adl(F_{1}F_{0}F_{-1}) 
						+ q\inv \adl(F_{0}F_{-1}E_{-1})
						-\adl(F_{-1}E_{0}E_{-1}) 
						+ q \adl(E_{1}E_{0}E_{-1}) \right) (L_2),
					\\
					C_2' =&\left(
						q^4 \adl(E_{0}E_{-1}E_{1}E_{0}) 
						+ \adl(F_{0}F_{1}F_{-1}F_{0}) 
						-(q^4+q^2) \adl(F_{0}E_{0}) 
						-q \adl(F_{1}F_{-1}F_{0}E_{0}) \right) (L_3)
						\\ &
						+\left(q^3 \adl(F_{1}F_{0}E_{1}E_{0}) 
						+q^3 \adl(F_{-1}F_{0}E_{-1}E_{0}) 
						-q^3 \adl(F_{0}E_{-1}E_{1}E_{0}) 
						+(q^5+q) \right) (L_3)
						.
				\end{align*}}
			using the generators of $\Uq(\gl_4)$ from \cref{convention: quantum group}. Hence, $C_i'\in \adl\Uq(\gl_4) L_i$ and we are done.
			\end{proof}		
	\begin{remark}
		In \cite{Ko-braided-module}, moreover a basis of $Z(\coideal)$ is constructed by taking partial traces of the universal $K$-matrix. The algebra generators $C'_0=1, C'_1,C'_2,C'_3$ in the proof of \cref{centre} are, up to scalars, these partial traces.  Constructing the centre via partial traces is a standard tool for quantum groups, \cite{center-quantum-group-Baumann}. Kolb's construction works in general for any quantum symmetric pair subalgebra of finite type, but is hard to compute in practice. 		
	\end{remark}

	\begin{cor}\label{centralcharacters}
		Let $\kappa,n,i\in \Z$. The central character of a finite dimensional irreducible rational representation $\ratIrrep(\kappa_\half,\kappa_1,[n],[n+\kappa-2i]-q^{-\kappa}[n])$ with  dominant highest, \cref{dominant}, is given by \begin{align*}
				C_1 &\longmapsto (q^{\kappa_\half}[n+\kappa-2i]+q^{\kappa_1-2}[n])
				,\quad\;\;&
				C_3 &\longmapsto (q^{\kappa_1}[n+\kappa-2i]+q^{\kappa_\half-2}[n])
				,\\	
				C_2 &\longmapsto (q^{2+\kappa}+q^{-2-\kappa})(q-q\inv)^{-2} -[n+\kappa-2i][n],\quad\;\;&
				\det &\longmapsto q^{\kappa_\half +\kappa_1}.
			\end{align*}	
			In particular, the centre separates the rational representations. 		
	\end{cor}	
	
	\begin{proof}
	We just evaluate the action of the $C_i$'s on a highest weight vector.  To see that they separate the rational representations it suffices to show that the values of the $C_i$'s uniquely determine $\kappa_\half, \kappa_1, \kappa, n, i$.  We use now a trick we learned from Stefan Kolb. Consider the following matrix. \hfill\\
\begin{minipage}[c]{0.32\textwidth}
	$\begin{pmatrix}
	[n+\kappa-2i]&q^{-2}[n]\\
	q^{-2}[n]&[n+\kappa-2i].
\end{pmatrix}$
\end{minipage}
\begin{minipage}[c]{0.68\textwidth}
If has determinant $[n+\kappa-2i]^2-q^{-4}[n]\not=0$ unless $(n,\kappa)=(0,2i)$, since the first summand is invariant under $q\mapsto q\inv$ and the second is not unless $n=0$. Thus, $C_1,C_3, C_4$ determine $q^\kappa_\half, q^\kappa_1, q^\kappa$, and then  $q^{-\kappa_\half} C_1-q^{-\kappa_1}C_3$ provides $[n]$ which in turn  gives  $i$ via  $C_1$ or $C_3$.  In case $(n,\kappa)=(0,2i)$ we obtain from $C_4$ and $C_2$  the value of $\kappa_\half+\kappa_1$  and $\kappa$ respectively, hence get  $\kappa_\half$ and $\kappa_1$.\qedhere
\end{minipage}
\end{proof}

	   \subsection{Harish-Chandra homomorphism}   
		\newcommand{\harishchandraKernel}{I}
		\newcommand{\hcHomomorphism}{\operatorname{pr}}

		To define an analogue of the Harish-Chandra homomorphism for $\coideal$ let $\hcHomomorphism\colon \coideal \longrightarrow \CartanPart$
		be the projection to the subspace $\CartanPart$ using \Cref{PBW basis}.  That is, it kills all PBW monomials $\PBWmonomial_{f,y,e,x,\beta,\zeta,\kappa_\half,\kappa_1}$ with $f+y+e+x>0$ and is the identity on $\CartanPart$.  
\begin{lem}There is a unique  algebra  involution $\op{W}_{\gl_2}\colon \csaZW\mapsto \csaZW$ given by 
\begin{equation}\label{st}
\dCoideal_\half \mapsto q^{-2}\dCoideal_1, \;\;\dCoideal_1\mapsto q^2 \dCoideal_\half,\quad \kCoideal \mapsto q^4\kCoideal\inv, \;\;\bo\mapsto Z+\bo\kCoideal\inv, \;\; Z\mapsto -q^4\kCoideal Z+(1-q^4)\bo.
\end{equation}
\end{lem}
\begin{proof} By the PBW theorem, $\csaZW$ is a Laurent-polynomial algebra. The map is specified on generators (consistently with inverses and on $\kCoideal$),  it suffices to show it is an involution which is easy.
\end{proof}

		\begin{thm}\label{HC}
			The restriction of  $\hcHomomorphism$ to $Z(\coideal)$ defines an injective algebra homomorphism 
			\[\xi=\hcHomomorphism_{Z(\coideal)}\colon \quad Z(\coideal)\longrightarrow\CartanPart
			\]  
			\end{thm}
with image invariant under $\op{W}_{\gl_2}$. We call $\xi$ the \emph{Harish-Chandra homomorphism} of $\coideal$.
		\begin{proof}

			\newcommand{\degZeroPart}{U^K}
			We first determine the kernel $\operatorname{ker}(\xi)$ of $\xi$ as a linear subspace. For this consider the space 	
			 \[
				\degZeroPart = \operatorname{span}\{\PBWmonomial_{f,y,e,x,\beta,\zeta,\kappa_\half,\kappa_1}\mid f+y-e-x=0\}
			\] 
			of degree $0$ elements in the sense of \cref{grading}. By \cref{commutation relations} and \cref{PBW basis}, $\degZeroPart$ is the centraliser of $\kCoideal$, hence $Z(\coideal)\subset \degZeroPart$.  By construction, the kernel of $\hcHomomorphism$ restricted to $\degZeroPart$ equals
\[
				\harishchandraKernel = \operatorname{span}\{\PBWmonomial_{f,y,e,x,\beta,\zeta,\kappa_\half,\kappa_1}\in \degZeroPart\mid f+y>0\}.
			\] 
			the kernel of $\hcHomomorphism$ restricted to $Z(\coideal)$ is then 
  $\harishchandraKernel \cap Z(\coideal)$.  Thus, $\operatorname{ker}(\xi)=\harishchandraKernel \cap Z(\coideal)$. 
  
  We claim that $\harishchandraKernel \cap Z(\coideal)=Z(\coideal)\cap \harishchandraKernel $ is an ideal of $Z(\coideal)$. We verify this by showing $\degZeroPart\harishchandraKernel \degZeroPart \subset \harishchandraKernel$. By \cref{PBW basis} we have $\harishchandraKernel =  \degZeroPart\cap {\tilde{U}}'_-\coideal $ where ${\tilde{U}}'_- =\{\fCoideal^f\fNewCoideal^y\mid f+y>0\}$. Since $\harishchandraKernel\subset  {\tilde{U}}'_-\coideal $, we have $\harishchandraKernel\cdot\harishchandraKernel \subset \harishchandraKernel$. By construction, it is also clear that $\harishchandraKernel\CartanPart\subset \harishchandraKernel$. It follows from $\degZeroPart=\harishchandraKernel+ \CartanPart$ that $\harishchandraKernel\degZeroPart\subset \harishchandraKernel$. 
On the other hand, by the commutator formulas in \cref{commutation relations} and \cref{EZ FZ commutator}-\cref{XZ YZ commutator}, we have $\CartanPart \minusCoideal\CartanPart \subset  \minusCoideal\CartanPart $, and $\CartanPart \plusCoideal\CartanPart \subset \plusCoideal\CartanPart $. Then,  $\CartanPart\harishchandraKernel\subset \harishchandraKernel$, thus $\degZeroPart\harishchandraKernel \subset \harishchandraKernel$ and the claim is proven.

Thus $\xi$ is an algebra homomorphism. To see injectivity note that a central element $z\in Z(\coideal)$ acts on any finite dimensional irreducible $\coideal$-module by multiplication with  $\hcHomomorphism(z)$. If  $z\in\operatorname{ker}(\xi)$, then $z$ acts on all polynomial representations by zero, hence must  be zero as element in $\Uq(\gl_4)$).\hfill\\
For the  $\op{W}_{\gl_2}$ -invariance one easily checks using \eqref{st} that $\op{W}_{\gl_2}(\xi(C_i))=\xi(C_i)$ for $1\leq i\leq 4$. 
	\end{proof}		

        \newcommand{\reflWeyl}{s}
        \subsection{Weyl group action}
Recall  that classical central characters for  $\gl_2\times\gl_2$ are controlled by the dot-orbits $W\ldotp \lambda$ of the Weyl group $\op{W}=S_2\times S_2$ on $\mathfrak{h}\times\mathfrak{h}$, where  $\mathfrak{h}:=\mathbb{R}\oplus\mathbb{R}$  with $S_2$-dot-action 
\[
        \text{$ \reflWeyl\ldotp(\lambda_1,\lambda_2) \coloneqq (\lambda_2-1,\lambda_1+1)$,  and set $(\lambda_1,\lambda_2) >\reflWeyl\ldotp(\lambda_1,\lambda_2)$ if $\lambda_1>\lambda_2$. }
        \]
For $(\lambda,\mu)\in P\times P$ with $P=\{(\lambda_1,\lambda_2)\mid \lambda_1-\lambda_2\in\Z, 2\lambda_i\in\Z\}$,  let  $\verma_{\lambda,\mu}$  the associated Verma module via          
 $((a,b),(c,d))\mapsto (a+c,b+d,[a-c],[b-d]-q^{b+d-a-c}[a-c])$ generalising \eqref{bijection irrep rational to partitions}.

\begin{prop}\label{hom integral Vermas}
Let $(\lambda, \mu)$, $(\lambda', \mu')\in P\times P$. Then  we have 
 $\Hom(\verma_{\lambda',\mu'}, \verma_{\lambda,\mu})=\groundring$ if $(\lambda', \mu')\in \op{W}\ldotp (\lambda, \mu)$ with $\lambda\geq\lambda'$, $\mu\geq\mu'$,  and $\Hom(\verma_{\lambda',\mu'}. \verma_{\lambda,\mu})=0$ otherwise. 
\end{prop}
\begin{proof}In the notation of  \cref{Homs between vermas or hw vecs in vermas}, $\Hom(\verma_{\lambda',\mu'}, \verma_{\lambda,\mu})\neq0$ if and only if  \[
				[\lambda_2-\mu_2]-q^{-\kappa}[n]=  q^{n-i}[\kappa-i]-q^{n-\kappa+i}[i] 
				\text{ or } 
				 q^{n+i-\kappa}[i]-q^{n+i}[\kappa-i] \text{ for some $i\in \N$}.
			\]
Evaluating  these Laurent polynomials at $q=1$ gives the solutions  $i=\lambda_1-\lambda_2$ and $i=\mu_1-\mu_2$ in  $i\in\N$ respectively. By \cref{Homs between vermas or hw vecs in vermas}, cf. \cref{hom into subvermas}, $\Hom(\verma_{\lambda',\mu'}, \verma_{\lambda,\mu})\neq 0$ exactly in case 
\[\Hom(\verma_{\lambda',\mu'}, \verma_{\lambda,\mu})=\groundring \text{ if and only if }(\lambda',\mu')\in\{(s\ldotp\lambda,t\ldotp\mu)\mid (s,t)\in\op{W}=S_2\times S_2\},\qedhere\]
		\end{proof}	
\begin{remark}The involution $\op{W}_{\gl_2}$ from \eqref{st}  should be seen as analogue of $(s,s)\in \op{W}$, the generator of the Weyl group of $\gl_2$ diagonally embedded into $\gl_2\times \gl_2$, since   $\op{W}_{\gl_2}(h)$  acts on $\verma_{s\ldotp\lambda,s\ldotp\mu}$ as $h\in\csaZW$ acts on $\verma_{\lambda,\mu}$. We did not find an analogue of $(s,e)$ or $(e,s)$, but we like to finish with an  observation.
\end{remark}
\begin{observation}
There exists a finite integral extension $\widehat{\csaZW}$ of $\csaZW$ with the action of $\op{W}_{\gl_2}$ extended to an action of $\op{W}$ such that the image of the Harish-Chandra homomorphism is contained in the $\op{W}$-invariants, i.e. $\xi\colon Z(\coideal)\hookrightarrow\CartanPart\cap (\widehat{\csaZW})^{\op{W}}$. 
\end{observation}
\begin{proof}
Recall that $\csaZW=\groundring[\dCoideal_\half^{\pm1}, \dCoideal_\half^{\pm1} ,\bo,Z]$. Consider \[
\widehat{\csaZW} \coloneqq\groundring\left[\dCoideal_j^{\pm{1}}, L_j^{\pm{1}}  , Q\mid j\in\{\half,1\}\right]/(Q^2 - \dCoideal_\half \dCoideal_1L_1L_1)
,
\] 
which contains $\csaZW$ as a subalgebra via 
$\dCoideal_j^{\pm}\mapsto \dCoideal_j^{\pm}, \bo\mapsto [L_\half;0], Z\mapsto [L_1;0]-[L_\half;0] \dCoideal_\half\inv\dCoideal_1$. Now there is an algebra involution $\op{W}_s$ on  $\widehat{\csaZW}$ given by
\[
Q\mapsto Q, \quad
\dCoideal_\half\mapsto q\inv QL_\half\inv, \quad\dCoideal_1\mapsto q QL_1\inv,\quad L_\half \mapsto q\inv Q\dCoideal_\half \inv,\quad L_1\mapsto qQ\dCoideal_1\inv.\]
It is a straightforward calculation to verify that this fixes the elements $\xi(C_i)$, $1\leq i\leq 4$. By composing with $\op{W}_{\gl_2}$ we obtain the action of the other simple reflection of $\op{W}$.
\end{proof}

\stoptoc

\section{Calculation of the $Z$-weights in \cref{ClebschGordan}}\label{lengthycalculation}
	In this section we denote $\alpha_j=q^j+q^{-j}$ for $j\in \Z$.
	We expand the definition of $Z$ to get \begin{equation}\label{Z expanded}
            Z   = - q^{-2} \fCoideal \bo \eCoideal 
                    -q\inv \bo\fCoideal\eCoideal 
                    -  q\inv \fCoideal\eCoideal \bo
                    -q\inv \bo [\kCoideal ;0 ]
					+
					\eCoideal \bo\fCoideal.  
	\end{equation}
	The calculation of $Z$-weights of $v\otimes \factorVone\pm_n$ in \cref{ClebschGordan} is the same as in the proof of \cref{polynomial irrep}, i.e. checking the recursion relation \cref{omega recursion}. We omit the details here.
	
	To determine the $Z$-weights of $\auxVectors_\pm$ in \cref{auxilary} 
	 we ``only'' need to calculate the $\eCoideal\bo\fCoideal$-eigenvalue of $\auxVectors_\pm$, since the first three terms in \cref{Z expanded} kill $v\otimes \basis_{1},v\otimes\basis_{\negidx{1}},\Fpm_\pm v\otimes \factorVone^{\pm}_{\pm n-1}$. We collect some  formulas. 
	\begin{lemma}
		The following formulas hold:
		\begin{gather}
			\label{B0B1 action on simple tensor}
		\begin{aligned}
		\bo \fCoideal (v\otimes \basis_{1}) = &
		 ~q^{1-n}[n+1]\alpha_n\inv \Fpm_- v\otimes \basis_1 
		 + q^{n+1}[n-1]\alpha_n\inv \Fpm_+ v\otimes \basis_{1} 
		;\\
		\bo\fCoideal (v\otimes \basis_{\negidx{1}}) = & 
		~q^{-n}[n+1]\alpha_n\inv \Fpm_- v\otimes \basis_{\negidx{1}}
		+ q^{n}[n-1]\alpha_n\inv \Fpm_+ v\otimes \basis_{\negidx{1}}
		;\\
		\bo \fCoideal(\Fpm_+ v\otimes \factorVone^{+}_{n-1}) 
		=& 
		~q^{n-1}[n-1]\alpha_{n-1}\inv  \Fpm_+^2 v\otimes \factorVone^{+}_{n-2} 
		+ (q^{1-n}+[n])\alpha_{n-1}\inv \Fpm_+\Fpm_-v\otimes\basis_{\half}
		\\%
		&+ (q^{-1}+q^{-n}[n])\alpha_{n-1}\inv \Fpm_+\Fpm_-v\otimes\basis_{\negidx{\half}}
		+ q^{2-\kappa}[n-1] \Fpm_+v\otimes \factorVtwo^{+}_{n+\kappa-3}
		,\\
		\bo \fCoideal(\Fpm_- v\otimes \factorVone^{-}_{-n-1}) 
		=&
		~q^{2-\kappa}[n+1] \Fpm_- v\otimes \factorVtwo^{-}_{-n+\kappa-3}
		+ q^{-n-1}[n+1]\alpha_{n+1}\inv  \Fpm_-^2 v\otimes \factorVone^{-}_{-n-2}
		\\%
		&+ (q^{n}[n]-q\inv )\alpha_{n+1}\inv \Fpm_+\Fpm_-v\otimes\basis_{\negidx{\half}}
		+ (q^{n+1}-[n])\alpha_{n+1}\inv \Fpm_+\Fpm_-v\otimes\basis_{\half}
		.
		\end{aligned}
		\end{gather}
	\end{lemma}
	\begin{proof}
		By \eqref{comult Fcoideal} and \cref{EF in E+- F+-} we have the formulas	$\fCoideal (v\otimes \basis_{1}) =  
		~q^{1-n}\alpha_n\inv \Fpm_- v\otimes \basis_1 
		+ q^{n+1}\alpha_n\inv \Fpm_+ v\otimes \basis_{1} 
		$	and $\fCoideal (v\otimes \basis_{\negidx{1}}) = 
		~q^{-n}\alpha_n\inv \Fpm_- v\otimes \basis_{\negidx{1}} 
		+ q^{n}\alpha_n\inv \Fpm_+ v\otimes \basis_{\negidx{1}},$ and the following 
		 \begin{align*}
		\fCoideal(\Fpm_+ v\otimes \factorVone^{+}_{n-1}) =& 
		~q^{n-1}\alpha_{n-1}\inv  \Fpm_+^2 v\otimes \factorVone^{+}_{n-2} 
		+ q^{1-n}\alpha_{n-1}\inv \Fpm_+\Fpm_-v\otimes\factorVone^{+}_{n-2} 
		+ q^{2-\kappa} \Fpm_+v\otimes \factorVtwo^{+}_{n+\kappa-3}
		,\\%
		\fCoideal(\Fpm_- v\otimes \factorVone^{-}_{-n-1}) =&
		~q^{2-\kappa} \Fpm_- v\otimes \factorVtwo^{-}_{-n+\kappa-3}
		+ q^{-n-1}\alpha_{n+1}\inv  \Fpm_-^2 v\otimes \factorVone^{-}_{-n-2}
		+ q^{n+1}\alpha_{n+1}\inv \Fpm_+\Fpm_-v\otimes\factorVone^{-}_{-n-2}
		\end{align*}
		Now the formulas \eqref{B0B1 action on simple tensor} follow by applying the comultiplication of $\bo$.	
	\end{proof}
	\begin{lemma}
		The following formulas hold:
			\begin{align*}
			&\eCoideal\bo \fCoideal(\Fpm_+^2 v\otimes \factorVone^{+}_{n-2}) = 
			q^{2-n}[2][i-1]\alpha_{\kappa+n-i-1} \Fpm_+ v\otimes \factorVone^{+}_{n-3}
			,\\
			&\eCoideal\bo \fCoideal(\Fpm_+\Fpm_- v\otimes \basis_{\half}) =
			q^{-n-1} [i]\alpha_{\kappa+n-i}\alpha_{n-1}\alpha_{n}\inv 
			\Fpm_- v\otimes \basis_{\half} 
			+ q^{n-1} [\kappa-i]\alpha_{n-i}\alpha_{n+1}\alpha_{n}\inv 
			\Fpm_+ v\otimes \basis_{\half}
			,\\
			&\eCoideal\bo \fCoideal(\Fpm_+\Fpm_- v\otimes \basis_{\negidx{\half}}) =
			q^{-n} [i]\alpha_{\kappa+n-i}\alpha_{n-1}\alpha_{n}\inv  
			\Fpm_- v\otimes \basis_{\negidx{\half}}
			+ q^{n} [\kappa-i]\alpha_{n-i}\alpha_{n+1}\alpha_{n}\inv
			\Fpm_+ v\otimes \basis_{\negidx{\half}}
			,\\
			&\eCoideal\bo \fCoideal(\Fpm_+ v\otimes \factorVtwo^{+}_{n+\kappa-3}) =
			q^{-n-1} [i]\alpha_{\kappa+n-i} v\otimes \factorVtwo^{+}_{n+\kappa-2} 
			+ \Fpm_+ v\otimes \factorVone^{+}_{n+2\kappa-5} 
			,
			\\
			&\eCoideal\bo \fCoideal(\Fpm_- v\otimes \factorVtwo^{-}_{-n+\kappa-3}) =
			\Fpm_-v\otimes \factorVone^{-}_{2\kappa-n-5}
			+ q^{n-1}[\kappa-i] \alpha_{n-i}v\otimes \factorVtwo^{-}_{\kappa-n-2}
			,\\
			&\eCoideal\bo \fCoideal(\Fpm_-^2 v\otimes \factorVone^{-}_{-n-2}) = q^{n+2}[2][\kappa-i-1]\alpha_{n-i+1}\Fpm_-v\otimes \factorVone^{-}_{1-n}
			.
		\end{align*}
	\end{lemma}
	 
	 Using these formulas, we are able to calculate the $\eCoideal\bo\fCoideal$ action on $\auxVectors_\pm$:
		\begin{align*}
			 \eCoideal\bo\fCoideal \auxVectors_+ = & 
			 ~q[n-1]\alpha_{n-1}\inv [2][i-1]\alpha_{n+\kappa-i-1} \Fpm_+ v\otimes \factorVone^+_{n-3}
			 \\
			 &
			 - q^{1-2n}[i][n+1]\alpha_{n+\kappa-i} \alpha_n\inv (\Fpm_- v\otimes \factorVone^+_{2\kappa+n-1} + q^{n-1}[\kappa-i]\alpha_{n-i} v\otimes \factorVtwo^+_{\kappa+n+2})
			 \\
			 &
			 + q^{n} [\kappa-i]\alpha_{n-i}\alpha_{n+1}\alpha_{n}\inv
			 \Fpm_+ v\otimes \basis_{\negidx{\half}})
			 + q^{-n-1} [i]\alpha_{\kappa+n-i}\alpha_{n-1}\alpha_{n}\inv 
			 \Fpm_- v\otimes \basis_{\half}	
			 \\
			 &
			 +[n-1](q^{n-1}-q^{n-\kappa}[i]\alpha_{n+\kappa-i}\alpha_n\inv ) (q^{\kappa-2}\Fpm_+v\otimes \basis_\half +q^{-n}[i]\alpha_{n+\kappa-i} v\otimes \basis_1)
			 \\
			 &
			 + [n-1](q^{2-\kappa}-q[i]\alpha_{n+\kappa-i}\alpha_n\inv )( 
				\Fpm_+v\otimes \basis_{\negidx{\half}} +q^{-n-1}[i]\alpha_{n+\kappa-i} v\otimes \basis_{\negidx{1}}
			 )
			 \\
			 &
			 +(q^{-1}+q^{-n}[n])\alpha_{n-1}\inv (q^{-n}[i]\alpha_{\kappa+n-i}\alpha_{n-1}\alpha_{n}\inv ) \Fpm_-v\otimes \basis_{\negidx{\half}}
			 \\
			 &
			+ (q^{1-n}+[n])\alpha_{n-1}\inv ( q^{n-1}[\kappa-i]\alpha_{n-i}\alpha_{n+1}\alpha_{n}\inv)\Fpm_+ v\otimes \basis_{1}
			 .
			\end{align*}
		\begin{align*}
			\eCoideal\bo\fCoideal \auxVectors_- =&
			~q[n+2][2][\kappa-i-1]\alpha_{n-i+1}\alpha_{n+1}\inv \Fpm_- v\otimes \factorVone^-_{-n-3}
			\\
			& - q^{2n+1}[n+1][\kappa-i]\alpha_{n-i} \alpha_n\inv (\Fpm_+ v\otimes \factorVone^-_{-n-3} + q^{-n-1}[i]\alpha_{\kappa+n-i} v\otimes \factorVtwo^-_{-n-\kappa})
			\\
			& +q^{n+1}\alpha_{n+1}\inv (q\inv[n] -q^{-n-2})(
				q^n[\kappa-i]\alpha_{n-i}\alpha_{n+1}\alpha_{n}\inv \Fpm_+ v\otimes \basis_{\negidx{\half}}
				\\
				&
				+ q^{-n} [i]\alpha_{\kappa+n-i}\alpha_{n-1}\alpha_{n}\inv 
				\Fpm_- v\otimes \basis_{\negidx{\half}}
			)
			\\
			& +q^{n+1}\alpha_{n+1}\inv (1-q^{-n-1}[n])(
				q^{-n-1}[i]\alpha_{\kappa+n-i}\alpha_{n-1}\alpha_{n}\inv
				\Fpm_-v\otimes \basis_{{\half}} 
				\\
				&
				+ q^{n-1}[\kappa-i]\alpha_{n-i}\alpha_{n+1}\alpha_{n}\inv \Fpm_+ v\otimes \basis_{\half}
			)
			\\
			&+[n+1](q^{2-\kappa}-q[\kappa-i]\alpha_{n-i}\alpha_n\inv )(
				\Fpm_-v\otimes \basis_{\negidx{\half}} + q^{n-1}[\kappa-i]\alpha_{n-i} v\otimes \basis_{\negidx{1}}
			)
			\\
			&+[n+1](q^{-n-\kappa}[\kappa-i]\alpha_{n-i}\alpha_n\inv -q^{-n-1})(
				q^{\kappa-2}\Fpm_-v\otimes \basis_{\half} + q^{n}[\kappa-i]\alpha_{n-i} v\otimes \basis_{1}
			).
		\end{align*}
	 
	Altogether, using  \cref{hw vec in ClebschGordan} ,we get
	 $
	 	\eCoideal\bo\fCoideal \newHwVec_\pm = [\kappa+n-2i\mp1]+[\kappa-2][n]
	 $. By \cref{Z expanded}, we get the desired $Z$-weights of $\newHwVec_\pm$ by adding their $-q\inv \bo [\kCoideal ;0 ]$-weights. 

\resumetoc

\bibliographystyle{alpha}
\bibliography{references-copy}
\end{document}